\documentclass[a4paper]{amsart}
\usepackage{amssymb,textcomp}

\usepackage{amsthm}
 \newtheorem{thm}{Theorem}
 \newtheorem{prop}[thm]{Proposition}
 \newtheorem{lem}[thm]{Lemma}
 \newtheorem{cor}[thm]{Corollary}
\theoremstyle{definition}

 \newtheorem{rem}[thm]{Remark}
 \numberwithin{equation}{section}
\title[Hankel operators between vector-valued Bergman spaces]{LITTLE HANKEL OPERATORS BETWEEN VECTOR-VALUED BERGMAN SPACES ON THE UNIT BALL}


\author[D. B\'ekoll\'e ]{ David B\'ekoll\'e} 
\address{ Department of Mathematics, Faculty of Science, University of Yaounde I; PO. Box 812 Yaounde-Cameroon}
\email{dbekolle@gmail.com}
\author[H. O. Defo]{Hugues Olivier Defo}
\address{ Department of Mathematics, Faculty of Science, University of Yaounde I; PO. Box 812 Yaounde-Cameroon}
\email{deffohugues@gmail.com}

\author[E. L. Tchoundja]{Edgar L. Tchoundja}
\address{ Department of Mathematics, Faculty of Science, University of Yaounde I; PO. Box 812 Yaounde-Cameroon}
\email{tchoundjaedgar@yahoo.fr}

\author[B. D. Wick]{Brett D. Wick}
\address{ Department of Mathematics, Washington University - St. Louis. One Brookings Drive, St. Louis, MO 63130-4899 USA}
\email{wick@math.wustl.edu}

\thanks{The second author would like to acknowledge the support of the GRAID program of IMU/CDC. He would also like to thank the International Centre for Theoretical  Physic (ICTP),~~~~Trieste (Italy) for partially supporting our visit to the centre where we have progressed in this work.} %% optional
\thanks{B. D. Wick's research partially supported in part by NSF grant DMS-1800057 as well as ARC DP190100970.} %% This will be filled in the journal office.

\begin{document}

\vspace{18mm} \setcounter{page}{1} \thispagestyle{empty}

\begin{abstract}
In this paper, we study the boundedness and the compactness of the little Hankel operators $h_b$ with operator-valued symbols $b$ between different weighted vector-valued Bergman spaces on the open unit ball $\mathbb{B}_n$ in $\mathbb{C}^n.$ More precisely, given two complex Banach spaces $X,Y,$ and $0 < p,q \leq 1,$ we characterize those operator-valued symbols $b: \mathbb{B}_{n}\rightarrow \mathcal{L}(\overline{X},Y)$ for which the little Hankel operator $h_{b}: A^p_{\alpha}(\mathbb{B}_{n},X) \longrightarrow A^q_{\alpha}(\mathbb{B}_{n},Y),$ is a bounded operator. Also, given two reflexive complex Banach spaces $X,Y$ and $1 < p \leq q < \infty,$ we characterize those operator-valued symbols $b: \mathbb{B}_{n}\rightarrow \mathcal{L}(\overline{X},Y)$ for which the little Hankel operator $h_{b}: A^p_{\alpha}(\mathbb{B}_{n},X) \longrightarrow A^q_{\alpha}(\mathbb{B}_{n},Y),$ is a compact operator.
 
\end{abstract}

\maketitle

\section{\textbf{Introducton}}  %% Please avoid complex formulas in (sub)titles

It is well known that Hankel operators constitute a very important class of operators in spaces of analytic functions. The study of these operators on different analytic spaces is not only motivated by the mathematical challenges it raises, but also by many applications on applied mathematics and in physics (see for example \cite{Peller} for more information). In this paper, we are interested on the boundedness and the compactness problem of the little Hankel operator with operator-valued symbols on weighted vector-valued Bergman spaces on the unit ball.

Throughout this paper, we fix a nonnegative integer $n$ and let $$\mathbb{C}^n = \mathbb{C} \times \cdots \times \mathbb{C}$$ denote the $n-$dimensional Euclidean space. For $$z = (z_{1},\cdots,z_{n}),\hspace{1cm} w = (w_{1},\cdots,w_{n}),$$ in $\mathbb{C}^{n},$ we define the inner product of $z$ and $w$ by $$ \langle z,w \rangle = z_{1}\overline{w_{1}} + \cdots + z_{n}\overline{w_{n}},$$ where $\overline{w_{k}}$ is the complex conjugate of $w_{k}.$ The resulting norm is then $$ |z| = \sqrt{\langle z,z \rangle} = \sqrt{|z_{1}|^2 + \cdots + |z_{n}|^2 }.$$ Endowed with the above inner product,  $\mathbb{C}^n$ become a Hilbert space whose canonical basis consists of the following vectors
$$e_{1} = (1,0,\cdots,0),~~e_{2}=(0,1,0,\cdots,0),~~ \cdots,~~e_{n} =(0,\cdots,0,1).$$
 The open unit ball in $\mathbb{C}^n$ is the set $$\mathbb{B}_n = \lbrace z \in \mathbb{C}^{n}: |z| < 1 \rbrace.$$
 
When $\alpha >-1,$ the weighted Lebesgue measure $\mathrm{d}\nu_{\alpha}$ in $\mathbb{B}_n$ is defined by $$\mathrm{d}\nu_{\alpha}(z) = c_{\alpha}(1-|z|^2)^{\alpha}\mathrm{d}\nu(z),\hspace{1cm} z \in \mathbb{B}_{n} $$ where $\mathrm{d}\nu$ is the Lebesgue measure in $\mathbb{C}^n$ and          $$c_{\alpha} = \dfrac{\Gamma(n+\alpha+1)}{n!\Gamma(\alpha+1)}$$ is the normalizing constant so that $\mathrm{d}\nu_{\alpha}$ becomes a probability measure on $\mathbb{B}_n.$ 
A function defined on the unit ball $\mathbb{B}_n$ will be called a vector-valued function when  it takes its values in some vector space.
If $X$ is a complex Banach space, a vector-valued function $f:\mathbb{B}_{n} \longrightarrow X$ (a $X$-valued function) is said to be strongly holomorphic in $\mathbb{B}_n$ if for every $z \in \mathbb{B}_{n}$ and for every $k \in \lbrace 1,\cdots, n \rbrace,$ the limit 
$$ \displaystyle \lim_{\lambda \longrightarrow 0}\dfrac{f(z+\lambda e_{k})-f(z)}{\lambda}$$ exists in $X,$ where $\lambda \in \mathbb{C}-\lbrace 0 \rbrace$. The space of all $X-$valued strongly holomorphic functions on $\mathbb{B}_n$ will be denoted by $\mathcal{H}(\mathbb{B}_{n},X).$ We will also denote by $H^{\infty}(\mathbb{B}_{n},X)$ the space of all bounded $X$-valued holomorphic functions.
Let $X^{\star}$ denotes the space of all bounded linear functionals $x^{\star}:X \longrightarrow \mathbb{C}$ (the topological dual space of $X$). We say that a vector-valued function $f:\mathbb{B}_{n} \longrightarrow X$ is  weakly holomorphic if for every $x^{\star} \in X^{\star},$ the scalar-valued function $x^{\star}(f): \mathbb{B}_{n} \longrightarrow \mathbb{C}$ is holomorphic in the usual sense. An important result by N. Dunford (\cite{Diestel}) shows that a vector-valued function is strongly holomorphic if and only if it is weakly holomorphic.

\subsection{\textbf{The conjugate $\overline{X}$ of the complex Banach space X}}

In the sequel, we will need the notion of \textgravedbl conjugate\textacutedbl of a complex Banach space (\cite{Roc}).\\

We will use the following definition and notation which can be found in \cite{Roc}. Let $x \in X,$ $x^{\star} \in X^{\star}$ and $\lambda \in \mathbb{C}.$ We define $$(\lambda x^{\star})(x): = \overline{\lambda} x^{\star}(x).$$ We also use the notation $$\langle x,x^{\star} \rangle_{X,X^{\star}} = x^{\star}(x)$$ to represent the \textasciigrave inner product\textasciiacute in the complex Banach space $X.$ We have the following identities $$\langle \lambda x,x^{\star} \rangle_{X,X^{\star}} = \lambda\langle x,x^{\star} \rangle_{X,X^{\star}} = \langle x,\overline{\lambda}x^{\star} \rangle_{X,X^{\star}},$$ so that we have a regular rule of an inner product. 
The complex conjugate $\overline{x}$ of $x \in X,$ is the linear functional on $X^{\star}$ defined by $$\overline{x}(x^{\star}) = \overline{\langle x,x^{\star} \rangle}_{X,X^{\star}},$$ for every $x^{\star} \in X^{\star}.$  Therefore, $$\overline{X} =  \lbrace \overline{x}: x \in X \rbrace $$ is called the complex conjugate of the Banach space $X.$ With the norm defined by $$\|\overline{x}\| := \sup_{\|x^{\star}\|_{X^{\star}} = 1}|\overline{x}(x^{\star})|,$$ $\overline{X}$ becomes a Banach space. Moreover, we have that $\|x\|_{X} = \|\overline{x}\|_{\overline{X}}$ for any $x \in X,$ so that $X$ and $\overline{X}$ are isometrically anti-isomorphic.

\subsection{\textbf{Vector-valued Bergman space}}

In the sequel, we will integrate vector-valued  measurable functions in the sense of Bochner (see \cite{Diestel} for more information). Let $X$ be a complex Banach space. A measurable function $f: \mathbb{B}_{n} \longrightarrow X$ is Bochner-integrable with respect to the measure $\nu_{\alpha}$ in the unit ball $\mathbb{B}_n$ if and only if the Lebesgue integral    
 $$\displaystyle \|f\|_{1,\alpha,X} = \int_{\mathbb{B}_n}\|f(z)\|_{X}\mathrm{d}\nu_{\alpha}(z) $$ is finite.             
For $0 < p < \infty,$ the Bochner-Lebesgue space $L^{p}_{\nu_{\alpha}}(\mathbb{B}_{n},X)$ consists of all vector-valued measurable functions $f:\mathbb{B}_{n} \longrightarrow X$ such that $$\displaystyle \|f\|^p_{p,\alpha,X} = \int_{\mathbb{B}_n}\|f(z)\|^p_{X}\mathrm{d}\nu_{\alpha}(z) < \infty.$$ 
The vector-valued Bergman space $A^{p}_{\alpha}(\mathbb{B}_{n},X)$ is defined by 
$$ A^{p}_{\alpha}(\mathbb{B}_{n},X) = L^{p}_{\nu_{\alpha}}(\mathbb{B}_{n},X) \cap \mathcal{H}(\mathbb{B}_{n},X).$$

The weak Bochner-Lebesgue space $L^{p,\infty}_{\alpha}(\mathbb{B}_{n},X)$ consists of all vector-valued measurable functions $f:\mathbb{B}_n \longrightarrow X$ for which 
$$\|f\|_{L^{p,\infty}_{\alpha}(\mathbb{B}_{n},X)} = \left( \sup_{\lambda > 0}\lambda^{p} \nu_{\alpha}\left(\lbrace z \in \mathbb{B}_{n}: \|f(z)\|_{X} > \lambda \rbrace\right)\right)^{1/p} < \infty.$$ The weak vector-valued Bergman space $A^{p,\infty}_{\alpha}(\mathbb{B}_{n},X)$ is defined by  
$$A^{p,\infty}_{\alpha}(\mathbb{B}_{n},X) = \mathcal{H}(\mathbb{B}_{n},X) \cap L^{p,\infty}_{\alpha}(\mathbb{B}_{n},X).$$

Let $X, Y$ be two complex Banach spaces and $\alpha > -1.$ We have the following two lemmas whose proofs can be found in \cite{Roc}.
 
\begin{lem}\label{premierlem}
Let $T: X \longrightarrow Y$ be a bounded linear operator. If $f:\mathbb{B}_{n} \longrightarrow X$ is $\nu_{\alpha}-$Bochner integrable in the unit ball, then $Tf:\mathbb{B}_{n} \longrightarrow Y$ is $\nu_{\alpha}-$Bochner integrable in the unit ball and we have
$$\displaystyle \int_{\mathbb{B}_n} Tf(z)\mathrm{d}\nu_{\alpha}(z) = T\left( \int_{\mathbb{B}_n}f(z)\mathrm{d}\nu_{\alpha}(z)\right).$$
\end{lem}

\begin{lem}\label{vers2} If $f:\mathbb{B}_{n} \longrightarrow X$  is a $\nu_{\alpha}$-Bochner integrable vector-valued function in the unit ball, then the following inequality holds
$$\displaystyle \left\|\int_{\mathbb{B}_n}f(z) \mathrm{d}\nu_{\alpha}(z)\right\|_{X} \leq \int_{\mathbb{B}_n}\|f(z)\|_{X}\mathrm{d}\nu_{\alpha}(z).$$ 
\end{lem}

\subsection{Vector-valued Lipschitz spaces and vector-valued $\gamma$-Bloch spaces.}

The radial derivative of a vector-valued holomorphic function $f: \mathbb{B}_{n} \longrightarrow X $ denoted $Nf$ is defined for $z \in \mathbb{B}_n$ by
\begin{equation}
\displaystyle Nf(z): = \sum_{j=1}^{n}z_{j}\frac{\partial f}{\partial z_{j}}(z).\label{ref4}
\end{equation}
Let $f \in \mathcal{H}(\mathbb{B}_{n},X)$ and 
$$\displaystyle f(z) = \sum_{k=0}^{\infty}f_{k}(z),\hspace{1cm} z \in \mathbb{B}_n$$ the homogeneous expansion of the function $f$ where $f_k$ are homogeneous holomorphic polynomials of degree $k$ with coefficients in $X.$ For any two real parameters $\alpha$ and $t$ such that neither $n+\alpha$ nor $n+\alpha+t$ is a negative integer, we define an invertible operator $R^{\alpha,t} : \mathcal{H}(\mathbb{B}_{n},X) \rightarrow \mathcal{H}(\mathbb{B}_{n},X)$ as 
\begin{equation}
\displaystyle R^{\alpha,t}f(z): = \sum_{k=0}^{\infty}\dfrac{\Gamma(n+1+\alpha)\Gamma(n+1+k+\alpha+t)}{\Gamma(n+1+\alpha+t)\Gamma(n+1+k+\alpha)}f_{k}(z),\label{diffop}
\end{equation}
where $z \in \mathbb{B}_{n}$ and $\Gamma$ is the classical Euler Gamma function.
For $\gamma \geq 0,$ we denote by $\Gamma_{\gamma}(\mathbb{B}_{n},X)$ the space of vector-valued holomorphic functions $f:\mathbb{B}_{n} \longrightarrow X$ for which there exists an integer $k > \gamma$ such that $$\|f\|_{\gamma,X} = \|f(0)\|_{X} + \sup_{z \in \mathbb{B}_{n}}(1 - |z|^2)^{k-\gamma}\|N^{k}f(z)\|_{X} < \infty,$$ where $N^{k} = N \circ N \circ \cdots \circ N$ $k-$times. The definition of the space $\Gamma_{\gamma}(\mathbb{B}_{n},X)$ is independent of the integer $k$ used.
The space $\Gamma_{\gamma}(\mathbb{B}_{n},X)$ will be called the vector-valued holomorphic Lipschitz space and for $\gamma = 0,$ we write $\mathcal{B}(\mathbb{B}_{n},X) = \Gamma_{0}(\mathbb{B}_{n},X).$ It is clear that $f \in \mathcal{B}(\mathbb{B}_{n},X)$ if and only if $f$ is a vector-valued holomorphic function and $$\|f\|_{\mathcal{B}(\mathbb{B}_{n},X)} = \|f(0)\|_{X} + \sup_{z \in \mathbb{B}_{n}}(1 - |z|^2)\|Nf(z)\|_{X} < \infty.$$ That is, $\mathcal{B}(\mathbb{B}_{n},X) = \Gamma_{0}(\mathbb{B}_{n},X)$ is the vector-valued Bloch space. 
The vector-valued $\gamma$-Bloch space $\mathcal{B}_{\gamma}(\mathbb{B}_{n},X)$ for $\gamma > 0,$ is defined as the space of vector-valued holomorphic functions $f \in \mathcal{H}(\mathbb{B}_{n},X)$ such that 
$$ \sup_{z \in \mathbb{B}_{n}}(1 - |z|^2)^{\gamma} \|Nf(z)\|_{X} < \infty.$$
The little vector-valued  $\gamma$-Bloch space $\mathcal{B}_{\gamma,0}(\mathbb{B}_{n},X)$ for $\gamma > 0,$ is the subspace of $\mathcal{B}_{\gamma}(\mathbb{B}_{n},X)$ consisting of functions $f$ such that
$$\lim_{|z| \rightarrow 1^{-}}(1 - |z|^2)^{\gamma} \|Nf(z)\|_{X} = 0.$$
It is easy to see that  $\mathcal{B}_{1}(\mathbb{B}_{n},X) = \mathcal{B}(\mathbb{B}_{n},X).$ Therefore, the vector-valued $\gamma$-Bloch spaces with $\gamma > 0$ generalize the vector-valued Bloch space. 
Let $\gamma \geq 0.$ The generalized vector-valued Lipschitz space $\Lambda_{\gamma}(\mathbb{B}_{n},X)$ consists of vector-valued holomorphic functions $f$ in $\mathbb{B}_{n}$ such that for some nonnegative integer $k > \gamma,$ we have $$\|f\|_{\Lambda_{\gamma}(\mathbb{B}_{n},X)} = \sup_{z \in \mathbb{B}_n} (1 - |z|^2)^{k-\gamma}\|R^{\alpha,k}f(z) \|_{X} < \infty.$$ We consider the following norm on the generalized vector-valued Lipschitz space $\Lambda_{\gamma}(\mathbb{B}_{n},X)$ by
$$\|f\|_{\Lambda_{\gamma}(\mathbb{B}_{n},X)} = \sup_{z \in \mathbb{B}_{n}}(1 - |z|^2)^{k-\gamma}\|R^{\alpha,k}f(z)\|_{X},$$ where $k > \gamma$ is a nonnegative integer. Equipped with this norm, the generalized vector-valued Lipschitz space $\Lambda_{\gamma}(\mathbb{B}_{n},X)$ becomes a Banach space.
The generalized little vector-valued Lipschitz space $\Lambda_{\gamma,0}(\mathbb{B}_{n},X)$ is the subspace of $\Lambda_{\gamma}(\mathbb{B}_{n},X),$ which consists of functions $f \in \Lambda_{\gamma}(\mathbb{B}_{n},X)$ such that 
\begin{equation}\label{eq1}
\lim_{|z| \rightarrow 1^{-}}(1 - |z|^2)^{k-\gamma}\|R^{\alpha,k}f(z)\|_{X} = 0.
\end{equation}  
When $\gamma = 0$ and $k =1,$ then $\Lambda_{0}(\mathbb{B}_{n},X) = \mathcal{B}(\mathbb{B}_{n},X).$
It is also important to note that as in the classical case, when $0 < \gamma < 1,$ we have $\Lambda_{\gamma}(\mathbb{B}_{n},X) = \mathcal{B}_{1-\gamma}(\mathbb{B}_{n},X).$

\subsection{Little Hankel operator with operator-valued symbol}

Given two complex Banach spaces $X$ and $Y,$ we denote by $\mathcal{L}(X,Y)$ the space of all bounded linear operators $T : X \longrightarrow Y$ endowed with the following norm $$ \|T\|_{\mathcal{L}(X,Y)} = \sup_{\|x\|_{X}=1} \|Tx\|_{Y} = \sup_{\|x\|_{X}=1, \|y^{\star}\|_{Y^{\star}}=1} |\langle Tx,y^{\star} \rangle_{Y,Y^{\star}}|,$$ where $T \in \mathcal{L}(X,Y).$ Then $\mathcal{L}(X,Y)$ is a Banach space.
We consider an operator-valued function $b:\mathbb{B}_{n} \longrightarrow \mathcal{L}(\overline{X},Y)$ and we suppose that $b \in \mathcal{H}(\mathbb{B}_{n}, \mathcal{L}(\overline{X},Y)).$ 
The little Hankel operator with operator-valued symbol $b,$ denoted  $h_b$ is defined for $z \in \mathbb{B}_n$ by
$$\displaystyle h_{b}f(z) := \int_{\mathbb{B}_n}\dfrac{b(w)\overline{f(w)}}{(1- \langle z,w\rangle)^{n+1+\alpha}}\mathrm{d}\nu_{\alpha}(w), \hspace{0,5cm} f \in H^{\infty}(\mathbb{B}_{n},X).$$
In the sequel, we will assume that the symbol $b$ satisfies the following condition
\begin{equation}
\displaystyle \int_{\mathbb{B}_n} \dfrac{\|b(w)\|_{\mathcal{L}(\overline{X},Y)}}{|1-\langle z,w \rangle|^{n+1+\alpha}}\mathrm{d}\nu_{\alpha}(w) < \infty, \hspace{0,5cm} \mbox{for every}~~z \in \mathbb{B}_{n}.\label{hypo1}
\end{equation}
It is easy to check that if $b$ satisfies $\eqref{hypo1},$ then the little Hankel operator $h_{b}$ is well defined on $ H^{\infty}(\mathbb{B}_{n},X).$
%Definir Lipschitz généralisé ici

\subsection{Problems and known results}
The boundedness properties of the little Hankel operator in the classical case (that is, when $X = Y = \mathbb{C}$) have been extensively studied and many results are now well known. For the case $n = 1,$ important references are \cite{Coifman} and  \cite{Zhu_operator}. For $n>1,$ a complete characterization has been obtained by Aline Bonami and Luo Luo in \cite{Bonami} when $p \leq q.$ In 2015, Pau and Zhao \cite{Pau} solved the case $1< q < p <\infty.$ Indeed, they showed that if $b$ is a holomorphic symbol, the little Hankel operator $h_{b}$ extends to a bounded operator from $A^p_\alpha(\mathbb{B}_{n},\mathbb{C})$ into $A^q_\alpha(\mathbb{B}_{n},\mathbb{C}),$ with $1< q < p < \infty,$ if and only if the symbol $b$ belongs to the weighted Bergman space $A^t_\alpha(\mathbb{B}_{n},\mathbb{C})$ where $1/t = 1/q - 1/p.$
We are here concerned with the question of characterizing the operator-valued holomorphic symbols $b$ for which the little Hankel operator $h_{b}$ extends into a bounded operator from $A^p_{\alpha}(\mathbb{B}_{n},X)$ into $A^q_{\alpha}(\mathbb{B}_{n},Y)$ where $0<p,q < \infty.$ In \cite{Aleman} Aleman and Constantin solved this problem for the particular case $n = 1,$ $p = q = 2$ and $X = Y = \mathcal{H}$ where $\mathcal{H}$ is a separable Hilbert space. They showed that the little Hankel operator $h_b$  extends into a bounded operator from $A^2_{\alpha}(\mathbb{B}_{n},\mathcal{H})$ into $A^2_{\alpha}(\mathbb{B}_{n},\mathcal{H})$ if and only if the symbol $b$ belongs to the Bloch space $\mathcal{B}(\mathbb{B}_{n},\mathcal{L}(\mathcal{H})).$ Constantin also obtained in \cite{Constantin} that the little Hankel operator $h_b$ is a compact operator from $A^2_{\alpha}(\mathbb{B}_{n},\mathcal{H})$ into $A^2_{\alpha}(\mathbb{B}_{n},\mathcal{H})$ if and only if the symbol $b$ belongs to the little vector-valued Bloch space $\mathcal{B}_{0}(\mathbb{B}_{n},\mathcal{K}(\mathcal{H})).$ Their results extend clearly the one known in the classical case (when $\mathcal{H} = \mathbb{C}$). 
In \cite{Roc}, Oliver solved this problem in the case $1 < p,q <\infty.$  
Mainly, he showed that for $1 <p < \infty,$ the little Hankel operator $h_{b}$ is bounded from $A^p_{\alpha}(\mathbb{B}_{n},X)$ into $A^p_{\alpha}(\mathbb{B}_{n},Y)$ if and only if the symbol $b$ belongs to the vector-valued Bloch space $\mathcal{B}(\mathbb{B}_{n},\mathcal{L}(\overline{X},Y))$ and this result clearly generalizes the one obtained by  Aleman and Constantin in \cite{Aleman}. Moreover, for $1 <p < q < \infty,$ Oliver showed that the little Hankel operator $h_{b}$ is bounded from $A^p_{\alpha}(\mathbb{B}_{n},X)$ into $A^q_{\alpha}(\mathbb{B}_{n},Y)$ if and only if the symbol $b$ belongs to the $\gamma$-Bloch space $\mathcal{B}_{\gamma}(\mathbb{B}_{n},\mathcal{L}(\overline{X},Y))$ with $\gamma = 1+(n+1+\alpha)\left( \frac{1}{q} - \frac{1}{p} \right).$ Also for $1 < q < p < \infty,$ Oliver showed that the little Hankel operator $h_{b}$ is bounded from $A^p_{\alpha}(\mathbb{B}_{n},X)$ into $A^q_{\alpha}(\mathbb{B}_{n},Y)$ if and only if $b \in A^{t}_{\alpha}(\mathbb{B}_{n},\mathcal{L}(\overline{X},Y)),$ with $1/t = 1/q - 1/p,$ which generalizes the main result in \cite{Pau}. We are also concerned here with the question of characterizing the operator-valued holomorphic symbols for which $h_b$ extends into a compact operator from $A^p_{\alpha}(\mathbb{B}_{n},X)$ into $A^q_{\alpha}(\mathbb{B}_{n},Y)$ where $1<p\leq q<\infty.$

\subsection{Statement of results}

Let $X$ be a complex Banach space and $0 <p \leq 1.$ The topological dual of the Bergman space $A^p_\alpha(\mathbb{B}_{n},X)$ can be identified with the Lipschitz space  $\Gamma_{\gamma}(\mathbb{B}_{n},X^{\star})$ as follows:

\begin{thm}\label{profD} Let $0 < p \leq 1.$ The space $(A^p_\alpha(\mathbb{B}_{n},X))^{\star}$ can be identified with $\Gamma_{\gamma}(\mathbb{B}_{n},X^{\star})$ with $\gamma = (n+1+\alpha)\left( \frac{1}{p}-1\right)$ under the pairing 
\begin{equation}
\displaystyle \langle f,g \rangle_{\alpha,X} = c_{k}\int_{\mathbb{B}_n} \langle f(z),D_{k}g(z) \rangle_{X,X^{\star}}(1 - |z|^2)^{k} \mathrm{d}\nu_{\alpha}(z),\label{intpair}
\end{equation}
where $D_{k}$ is defined by $\eqref{pseudok},$ $k > \gamma,$ is an integer, $g \in \Gamma_{\gamma}(\mathbb{B}_{n},X^{\star})$ and $f \in A^p_\alpha(\mathbb{B}_{n},X).$ Moreover, $$\|g\|_{\Gamma_{\gamma}(\mathbb{B}_{n},X^{\star})} \simeq \sup_{\|f\|_{A^p_\alpha(\mathbb{B}_{n},X)} =1} |\langle f,g \rangle_{\alpha,X}|.$$
\end{thm}

Before stating the next results, we need to make another assumption on the operator-valued symbol $b.$ More precisely, we assume that the operator-valued holomorphic symbol $b$ satisfies the following condition:
\begin{equation}
\displaystyle \int_{\mathbb{B}_{n}} \|b(z)\|_{\mathcal{L}(\overline{X},Y)} \log\left( \dfrac{1}{1 - |z|^2} \right) \mathrm{d}\nu_{\alpha}(z) < \infty.\label{hypimp}
\end{equation}

Let $X$ and $Y$ be two complex Banach spaces. Our contributions to the boundedness problem of the little Hankel operator with operator-valued symbol for $0 <p,q \leq 1$ are the following : 

\begin{thm}\label{principe1}  
Suppose $0 < p \leq 1,$ and $\alpha > -1.$ If the little Hankel operator $h_{b}$ extends to a bounded operator from  $A^p_{\alpha}(\mathbb{B}_{n},X)$ into $A^q_{\alpha}(\mathbb{B}_{n},Y)$ for some positive $q<1,$ then the symbol $b$ is in $\Gamma_{\gamma}(\mathbb{B}_{n},\mathcal{L}(\overline{X},Y))$ with $\gamma = (n+1+\alpha)\left( \frac{1}{p}-1 \right).$
Conversely, if $b$ is in $\Gamma_{\gamma}(\mathbb{B}_{n},\mathcal{L}(\overline{X},Y))$ with
$\gamma = (n+1+\alpha)\left(\frac{1}{p}-1 \right),$ then the little Hankel operator $h_{b}: A^p_{\alpha}(\mathbb{B}_{n},X) \longrightarrow A^{1,\infty}_{\alpha}(\mathbb{B}_{n},Y)$ is a bounded operator. 
\end{thm}
As a direct consequence, we have the following result:
\begin{cor}\label{cons1} Suppose $0 < p \leq 1,$ and $\alpha > -1.$ The little Hankel operator $h_{b}$ extends to a bounded operator from  $A^p_{\alpha}(\mathbb{B}_{n},X)$ into $A^q_{\alpha}(\mathbb{B}_{n},Y)$ for some positive $q<1$ if and only if its symbol $b$ belongs to  $\Gamma_{\gamma}(\mathbb{B}_{n},\mathcal{L}(\overline{X},Y)),$ where
$\gamma = (n+1+\alpha)\left(\frac{1}{p}-1\right).$
\end{cor}

\begin{thm}\label{principal2} 
Let $0 < p \leq 1,$ $\alpha >-1$ and $\gamma = (n+1+\alpha)\left( \frac{1}{p}-1\right).$ The little Hankel operator 
extends to a bounded operator from $A^p_{\alpha}(\mathbb{B}_{n},X)$ into $A^{1}_{\alpha}(\mathbb{B}_{n},Y)$ if and only if for some integer $k > \gamma,$
\begin{equation}
 \|N^{k}b(w)\|_{\mathcal{L}(\overline{X},Y)} \leq \dfrac{C}{(1-|w|^2)^{k-\gamma}}\left(\log \dfrac{1}{1-|w|^2} \right)^{-1}\hspace{1cm}w \in \mathbb{B}_{n}.\label{pri2}
\end{equation} 
\end{thm}

\begin{thm}\label{agene1} Suppose $1 < p \leq q <\infty.$   
The little Hankel operator $h_{b}: A^p_{\alpha}(\mathbb{B}_{n},X)\\
\rightarrow A^{q}_{\alpha}(\mathbb{B}_{n},Y)$ is a bounded operator if and only if $b \in \Lambda_{\gamma_{0}}(\mathbb{B}_{n},\mathcal{L}(\overline{X},Y)),$ where $$\gamma_{0} = (n+1+\alpha)\left(\frac{1}{p} - \frac{1}{q}\right).$$  
Moreover, $\|h_{b}\|_{A^p_{\alpha}(\mathbb{B}_{n},X) \rightarrow A^{q}_{\alpha}(\mathbb{B}_{n},Y)} \simeq \|b\|_{ \Lambda_{\gamma_{0}}(\mathbb{B}_{n},\mathcal{L}(\overline{X},Y))}.$
\end{thm}

If $X,Y$ are reflexive complex Banach spaces, then we have the following theorem

\begin{thm}\label{Compactp} Suppose that $1 < p \leq q < \infty,$ and $\alpha >-1$ The little Hankel operator $h_b : A^{p}_{\alpha}(\mathbb{B}_{n},X) \longrightarrow A^{q}_{\alpha}(\mathbb{B}_{n},Y)$ is a compact operator if and only if $$ b \in \Lambda_{\gamma_{0},0}(\mathbb{B}_{n},\mathcal{K}(\overline{X},Y)),$$ where $\Lambda_{\gamma_{0},0}(\mathbb{B}_{n},\mathcal{K}(\overline{X},Y))$ denotes the generalized little vector-valued Lipschitz space and $\gamma_{0} = (n+1+\alpha)\left( \frac{1}{p} - \frac{1}{q}\right),$ see \eqref{eq1}.
\end{thm}

\subsection{Plan of the paper}

The paper is divided into six sections. 
In Section $2,$ we recall some preliminary notions on vector-valued holomorphic functions and we also give the proofs of some important results. Section $3$ contains the proof of Theorem $\ref{profD}$ on the dual of the vector-valued Bergman space $A^p_\alpha(\mathbb{B}_{n},X)$ for $ 0 < p \leq 1.$ In Section $4,$ we give the proof of Theorem $\ref{principe1}$ and Corollary $\ref{cons1}$. In Section $5,$ we give the proof of Theorem $\ref{principal2}$. 
In Section $6,$ We first give some preliminaries results to prepare the proof of Theorem $\ref{Compactp}.$ We recall the result by Oliver \cite{Roc} of the boundedness of the little Hankel operator with operator-valued symbol $h_b$ from $A^p_\alpha(\mathbb{B}_{n},X)$ into $A^{q}_{\alpha}(\mathbb{B}_{n},Y),$ with $1 < p \leq q < \infty$ and we generalize it. In the same section, we give the proof of Theorem $\ref{Compactp}.$\\

Throughout this paper, when there is no additional condition, $X$ and $Y$ will denotes two complex Banach spaces, the real parameter $\alpha$ will be chosen such that $\alpha>-1$ and $c$ will be a positive constant whose value may change from one occurrence to the next. We will also adopt the following notation: we will write $A \lesssim B$ whenever there exists a positive constant $c$ such that $A \leq c B.$ We also write  $A \simeq B$ when $A \lesssim B$ and $B \lesssim A.$

\section{\textbf{Preliminaries}}

\subsection{Vector-valued Bergman projection and integral estimates}

Here we give some definitions and notations which will be used later and can be found in  \cite{Roc} and \cite{Bonami}.

 For $f \in L^{1}_{\alpha}(\mathbb{B}_{n},X)$ and $z \in \mathbb{B}_n,$ the Bergman projection  $P_{\alpha}f$ of $f$ is the integral operator defined by 
$$\displaystyle P_{\alpha}f(z) := \int_{\mathbb{B}_n} K_{\alpha}(z,w) f(w)\mathrm{d}\nu_{\alpha}(w),$$ where $ K_{\alpha}(z,w) := \dfrac{1}{(1-\langle z,w \rangle)^{n+1+\alpha}}$ is the Bergman reproducing kernel of $\mathbb{B}_n.$  In this situation, $P_{\alpha}f$ is also a $X$-valued holomorphic function. 

\begin{lem}[Density]\label{denslip}
 Suppose that $0 < p < \infty.$ Then the space of all bounded vector-valued holomorphic functions $H^{\infty}(\mathbb{B}_{n},X)$ is dense in $A^p_{\alpha}(\mathbb{B}_{n},X).$
\end{lem}

\begin{proof}
We are going to give the proof  for $0 < p < 1,$ since the case $1 \leq p < \infty$ is \cite[Lemma $2.1.4$]{Roc}.
Given a function $f \in A^p_{\alpha}(\mathbb{B}_{n},X),$ let $f_{\rho}$ defined for $z \in \mathbb{B}_n$ by 
$f_{\rho}(z): = f(\rho z),$ where $0 < \rho < 1.$ The function $f_{\rho}$ is  holomorphic in the set $\lbrace z\in \mathbb{B}_{n} : |z| < 1/\rho \rbrace$ hence is bounded on $\mathbb{B}_{n}.$ We first recall that the integral means $$M_{p}(r,f):= \int_{\mathbb{S}_n} \|f(r \zeta) \|^p_{X} \mathrm{d}\sigma(\zeta), \hspace{1cm} 0 \leq r < 1$$ are increasing with $r,$ see \cite[Corollary $4.21$]{Zhu_functions}.
Since  $M_{p}(r,f_{\rho}) = M_{p}(\rho r,f),$ we have by Minkowski's inequality that
$$ M^p_{p}(r,f_{\rho}-f) \leq M^p_{p}(r,f) + M^p_{p}(r,f_{\rho}) \leq 2M^p_{p}(r,f).$$ 
By the formula of \cite[$(1.1.1)$]{Roc}, (integration in polar coordinates formula) we get
\begin{equation}
\|f - f_{\rho}\|^p_{p,\alpha,X} = 2nc_{\alpha}\int_{0}^{1}M^p_{p}(r,f_{\rho}-f)(1 - r^2)^{\alpha}r^{2n-1}\mathrm{d}r.\label{prof1}
\end{equation}
Since  $f \in A^p_{\alpha}(\mathbb{B}_{n},X),$ we have that the function $M^p_{p}(r,f)$ is integrable over the interval $[0, 1)$ with respect to the measure $2n(1 - r^2)^{\alpha}r^{2n-1}\mathrm{d}r.$ It is also clear that $f_\rho \rightarrow f$ on any compact subsets of $\mathbb{B}_n$ which implies that $M^p_{p}(r,f_{\rho}-f) \rightarrow 0$ for each $r \in [0, 1)$ as $\rho \rightarrow 1.$ Applying the dominated convergence theorem in $\eqref{prof1},$ we obtain that $\|f - f_{\rho}\|^p_{p,\alpha,X} \longrightarrow 0,$ as $\rho \rightarrow 1.$
\end{proof}

\begin{cor}\label{densberg} For $0 < p \leq 1,$ the following inclusion is dense
$$A^2_\alpha(\mathbb{B}_{n},X) \subset A^p_\alpha(\mathbb{B}_{n},X).$$
\end{cor}
\begin{proof}
The proof follows directly from Lemma $\ref{denslip}.$
\end{proof}

In \cite{Blasco5}, Oscar Blasco obtained the duality theorem for the vector-valued Bergman spaces in the unit disc  $\mathbb{B}_{1}$ without any restriction on the Banach space. The proof also works for the unit ball $\mathbb{B}_n.$
The result is stated as follows:

\begin{thm}[Duality]\label{dual1} Suppose $1 < p < \infty.$
The dual space $(A^p_\alpha(\mathbb{B}_{n},X))^{\star}$ can be identified with $A^{p'}_\alpha(\mathbb{B}_{n},X^{\star}),$ where $p'$ is the conjugate exponent of $p$ given by $\frac{1}{p}+\frac{1}{p'} =1,$ under the integral pairing defined by 
\begin{equation}
\displaystyle \langle f,g \rangle_{\alpha,X} := \int_{\mathbb{B}_n}\langle f(z),g(z) \rangle_{X,X^{\star}}\mathrm{d}\nu_{\alpha}(z),
\end{equation}
for any $f \in A^p_\alpha(\mathbb{B}_{n},X),$ $g \in A^{p'}_\alpha(\mathbb{B}_{n},X^{\star}).$
\end{thm}
\begin{rem} Suppose $1 <p < \infty.$ If $X$ is a reflexive complex Banach space, then the vector-valued Bergman space $A^p_\alpha(\mathbb{B}_{n},X)$ is a reflexive Banach space. 
\end{rem}

The following reproducing kernel formula also holds for vector-valued Bergman spaces. The proof can be found in \cite[Proposition $2.1.2$]{Roc}.
\begin{prop}\label{prop1} Let $f \in A^1_\alpha(\mathbb{B}_{n},X).$ We have  
$$ \displaystyle f(z) :=  \int_{\mathbb{B}_n} \dfrac{f(w)}{(1-\langle z,w \rangle)^{n+1+\alpha}}\mathrm{d}\nu_{\alpha}(w),$$
for any $z \in \mathbb{B}_n.$
\end{prop}
We have the following pointwise estimate on the vector-valued Bergman spaces. The proof can be found in \cite{Roc}.

\begin{thm}\label{thm1} Let $0 < p < \infty.$ Then 
$$ \|f(z)\|_{X} \leq \dfrac{\|f\|_{p,\alpha,X}}{(1-|z|^2)^{(n+1+\alpha)/p}},$$
for any $f \in A^p_{\alpha}(\mathbb{B}_{n},X)$ and $z \in \mathbb{B}_{n}.$
\end{thm}

The following lemma is critical for many problems concerning the weighted vector-valued Bergman spaces $A^p_{\alpha}(\mathbb{B}_{n},X)$ whenever $0 < p \leq 1$ and will be extensively used.

\begin{lem}\label{lemm22} Let $0 < p \leq 1.$ Then 
$$\displaystyle \int_{\mathbb{B}_n}\|f(z)\|_{X}(1-|z|^2)^{(\frac{1}{p}- 1)(n+1+\alpha)}\mathrm{d}\nu_{\alpha}(z) \leq \|f\|_{p,\alpha,X},$$
for all $f \in A^p_{\alpha}(\mathbb{B}_{n},X).$ 
\end{lem}
\begin{proof} Write 
$$ \|f(z)\|_{X} = \|f(z)\|^p_{X} \|f(z)\|^{1-p}_{X},$$ and estimate the second factor using Theorem $\ref{thm1}.$ The desired result follows.
\end{proof}

The following technical result is proved in \cite[Lemma $3.1$]{Bonami} 
\begin{lem}\label{lemm1} Let $\beta,\delta > 0.$ For all $w \in \mathbb{B}_{n},$ we have 
$$\displaystyle I_{\alpha}(w) := \int_{\mathbb{B}_n}\left|\log\left(\dfrac{1-\langle z,w \rangle}{1-|w|^2} \right)\right|^{\delta}\dfrac{(1-|w|^2)^{\beta}}{|1-\langle z,w \rangle|^{n+1+\alpha+\beta}}\mathrm{d}\nu_{\alpha}(z) \leq C,$$
where $C$ is independent of $w$ and $\log$ is the principal branch of the logarithm. 
\end{lem}

In the sequel, we will also need the following lemma which the scalar version can be found in \cite{Grafakos}.
\begin{lem}\label{nessa} If $0 < q < 1,$ then the identity $ i : L^{1,\infty}_{\alpha}(\mathbb{B}_{n},X) \hookrightarrow L^{q}_{\alpha}(\mathbb{B}_{n},X)$ is continuous in the sense that there exists a constant $C(q) >0$ such that for every $f\in  L^{1,\infty}_{\alpha}(\mathbb{B}_{n},X),$ we have
$$ \|f\|_{q,\alpha,X} \leq C(q) \|f\|_{L^{1,\infty}_{\alpha}(\mathbb{B}_{n},X)}.$$
\end{lem}

The following result will be very useful in many situations. A proof can be found in \cite{Zhu_functions}.
\begin{thm}\label{estimint} For $\beta \in \mathbb{R},$ let 
$$\displaystyle I_{\alpha,\beta}(z) : = \int_{\mathbb{B}_n}\dfrac{(1-|w|^2)^{\alpha}\mathrm{d}\nu(w)}{|1- \langle z,w \rangle|^{n+1+\alpha+\beta}}, \hspace{0,5cm}  z \in \mathbb{B}_{n}.$$
\begin{itemize}
\item[$(i)$] If $\beta = 0,$ there exists a constant $C >0$ such that 
$$I_{\alpha,\beta}(z) \leq C \log \dfrac{1}{1- |z|^2},\quad z\in \mathbb{B}_n.$$
\item[$(ii)$] If $\beta > 0,$ there exists a constant $C >0$ such that 
$$I_{\alpha,\beta}(z) \leq C \dfrac{1}{(1- |z|^2)^\beta},\quad z\in \mathbb{B}_n.$$
\item[$(iii)$] If $\beta < 0,$ there exists a constant $C >0$ such that 
$$I_{\alpha,\beta}(z) \leq C.$$
\end{itemize} 
\end{thm}

\subsection{Differential operators and equivalent norms for $\Gamma_\gamma$}

Given a positive integer $k,$ we define the differential operator $D_k$ by
\begin{equation}
D_{k}:= (2I+N) \circ (3I+N) \circ \ldots \circ ((k+1)I+N),\label{pseudok}
\end{equation}
where $I$ is the identity operator and $N$ is the differential operator given in $\eqref{ref4}.$\\

In the sequel, we denote by $\mathcal{P}(\mathbb{B}_{n},X)$ the space of all vector-valued holomorphic polynomials.
The proof of the following lemma is similar as in the scalar case in \cite{Hedenmalm}.

\begin{lem}\label{primer} 
For all $f \in \mathcal{P}(\mathbb{B}_{n},X)$ and $g \in \mathcal{P}(\mathbb{B}_{n},X^{\star}),$ we have the following identity
\begin{eqnarray*}
\displaystyle \int_{\mathbb{B}_n} \langle f(z),g(z) \rangle_{X,X^{\star}} \mathrm{d}\nu_{\alpha}(z) & = & \displaystyle c_{k}\int_{\mathbb{B}_n}  \langle f(z),D_{k}g(z) \rangle_{X,X^{\star}} (1-|z|^2)^{k} \mathrm{d}\nu_{\alpha}(z),\label{ref56}
\end{eqnarray*}
where $c_k$ is a positive constant depending only on the integer $k.$ The above identities are valid for vector-valued holomorphic functions when both sides make sense.
\end{lem}

The following lemma will be very useful in the sequel.

\begin{lem}\label{lem22} Let $\lbrace a_k \rbrace$ a sequence of positive numbers. For any positive integer $k,$ let $M_{k}$ the differential operator of order $k$ defined by
$$ M_{k} := (a_{0}I+N) \circ (a_{1}I+N) \circ \ldots \circ (a_{k-1}I+N).$$ 
Then a vector-valued holomorphic function $f$ belongs to  $\Gamma_{\gamma}(\mathbb{B}_{n},X)$ if and only if there exists an integer $k > \gamma$  such that
$$\sup_{z \in \mathbb{B}_{n}}(1-|z|^2)^{k-\gamma}\|M_{k}f(z)\|_{X} < \infty.$$
\end{lem}
\begin{proof} 
Let us assume first that $f \in \Gamma_{\gamma}(\mathbb{B}_{n},X),$ and we prove the desired estimate on $M_k.$
By assumption, there exists an integer $k > \gamma $ and a positive constant $C$ such that 
$$ \|N^{k}f(z)\|_{X} \leq C (1 - |z|^2)^{\gamma-k},$$ for any $z \in \mathbb{B}_{n}.$ 
It is enough to prove that the following inequality
$$ \|N^{j}f(z) \|_{X} < C (1 - |z|^2)^{\gamma-k} ,$$ holds for $0 \leq j < k,$ 
since the assumption give the case $j = k.$ For $g \in \mathcal{H}(\mathbb{B}_{n},X)$ and $z = rz',$ where $r = |z|,$ and $z'$ is in the unit sphere. We have 
$$ N g(r z') = r \partial_{r} g(rz').$$ Thus, 
$$ g(rz') - g(z'/2) = \int_{\frac{1}{2}}^{r} Ng(sz')\frac{ds}{s}.$$ Now, for $g \in \mathcal{H}(\mathbb{B}_{n},X)$ such that $\|Ng(z)\|_{X} \leq C (1 - |z|^2)^{\gamma-k}.$ We have that
\begin{eqnarray*}
\|g(rz') - g(z'/2) \|_{X} & \leq & 2 \int_{\frac{1}{2}}^{r} \|Ng(sz')\|_{X} \mathrm{d}s\\
& \leq & 4C \int_{\frac{1}{2}}^{r} (1 - s^2)^{\gamma-k}s \mathrm{d}s \\
& = & -2C\int_{\frac{1}{2}}^{r} -2s(1 - s^2)^{\gamma-k} \mathrm{d}s\\
& = & \left[ \dfrac{-2C}{\gamma-k+1} (1 - s^2)^{\gamma-k+1} \right]_{\frac{1}{2}}^{r}\\
& = &  \dfrac{-2C}{\gamma-k+1} \left\lbrace (1- r^2)^{\gamma-k+1} - (1 - \frac{1}{4})^{\gamma-k+1} \right\rbrace.
\end{eqnarray*}
Now, if $\gamma-k+1 < 0,$ then
$$\|g(rz') - g(z'/2) \|_{X} \leq   \dfrac{-2C}{\gamma-k+1} (1 - r^2)^{\gamma-k} = C_{k,\gamma}(1 - |z|^2)^{\gamma-k}.$$
If $\gamma-k+1 > 0,$ then 
\begin{eqnarray*}
\|g(rz') - g(z'/2) \|_{X} & \leq & \dfrac{2C}{\gamma-k+1}  \left\lbrace (1 - \frac{1}{4})^{\gamma-k+1} - (1 - r^2)^{\gamma-k+1} \right\rbrace \\
 & \leq & \dfrac{2C}{\gamma-k+1} (1 - \frac{1}{4})^{\gamma-k+1} = C'_{k,\gamma}\\
 & \leq & C'_{k,\gamma} (1 - |z|^2)^{\gamma-k},
\end{eqnarray*}
where the last inequality is justified using the fact that $(1 - |z|^2)^{\gamma-k} > 1.$ It then follows that
$$\|g(z)\|_{X} \leq C (1 - |z|^2)^{\gamma-k}.$$ Now, we use this fact inductively for $g = N^{k} f,$ then $g = N^{k-1} f,~~\cdots$ to conclude.
Conversely, assume that there exists an integer $k > \gamma $ and a positive constant $C$ such that $$ \|M_{k} f(z) \|_{X} \leq C (1 - |z|^2)^{\gamma-k},$$ for any $z \in \mathbb{B}_{n}.$ To conclude, it is sufficient to prove that for a fixed positive real $a$, the inequality 
\begin{equation}
\|ag(z) + Ng(z) \|_{X} \leq C (1 - |z|^2)^{\gamma-k} \label{luoluo}
\end{equation}
implies the inequality  $$\|Ng(z) \|_{X} \leq C (1 - |z|^2)^{\gamma-k},$$ for any function $g \in \mathcal{H}(\mathbb{B}_{n},X).$
Choose a real $\beta$ such that $\beta+\gamma - k  > -1.$ By the assumption $\eqref{luoluo},$ we have that 
$$\displaystyle \int_{\mathbb{B}_n} \|ag(z) + Ng(z) \|_{X} (1 - |z|^2)^{\beta} \mathrm{d}\nu(z) < \infty.$$ Thus, for any $z \in \mathbb{B}_{n},$ we have $$\displaystyle ag(z) + Ng(z) = c_{\beta} \int_{\mathbb{B}_{n}}\dfrac{[ag(w) + Ng(w)]}{(1 - \langle z,w \rangle)^{n+1+\beta}} (1 - |w|^2)^{\beta} \mathrm{d}\nu(w).$$ 
Then, differentiating under the integral sign, we obtain that for all $1 \leq i \leq n,$
we get $$\displaystyle \partial_{z_i} \left[ ag(z) + Ng(z) \right] = (n+1+\beta)c_{\beta} \int_{\mathbb{B}_{n}}\dfrac{[ag(w) + Ng(w)]\overline{w}_{i}}{(1 - \langle z,w \rangle)^{n+2+\beta}} (1 - |w|^2)^{\beta} \mathrm{d}\nu(w).$$ Therefore,
$$\displaystyle N \left( ag(z) + Ng(z) \right) = (n+1+\beta)c_{\beta} \int_{\mathbb{B}_{n}}\dfrac{[ag(w) + Ng(w)]\langle z,w \rangle }{(1 - \langle z,w \rangle)^{n+2+\beta}} (1 - |w|^2)^{\beta} \mathrm{d}\nu(w).$$
Applying $\eqref{luoluo},$ and Theorem $\ref{estimint},$ we get that for all $1 \leq i \leq n,$
\begin{eqnarray*}
\displaystyle \|N \left( ag(z) + Ng(z) \right)\|_{X} & \leq & C c_{\beta}\int_{\mathbb{B}_{n}} \dfrac{(1 - |w|^2)^{\gamma-k+\beta}}{|1 - \langle z,w \rangle |^{n+1+\gamma-k+\beta+(k-\gamma+1)}}\mathrm{d}\nu(w)\\
& \leq & C (1 - |z|^2)^{\gamma-k-1}.
\end{eqnarray*}
Thus, the derivative of $ ag(z) + Ng(z)$ is bounded by $(1 -|z|^2)^{\gamma-k-1}.$ So, to prove the inequality above, we are reduced to consider smooth functions $\phi$ of one variable $r \in [0,1),$ and to prove that the inequality $$ \| \psi'(r) \|_{X} \leq C(1-r)^{\gamma-k-1},$$ with $\psi(r) = a\phi(r) + r \phi'(r),$ implies that $$ \|r\phi'(r)\|_{X} \leq C (1-r)^{\gamma-k}$$(here, $\phi(r) = g(rz')$). Now, differentiating $\psi,$ we obtain
$\psi'(s) = (a+1)\phi'(s) + s\phi''(s).$ Multiplying both sides of the previous inequality by $s^{a},$ we obtain that $s^{a}\psi'(s) =  (a+1)s^{a}\phi'(s) + s^{a+1}\phi''(s) = \left[ s^{a+1}\phi'(s)\right]'.$
Then integrating the equality above on $[0,r],$ we obtain that $$ \phi'(r) = \dfrac{1}{r^{a+1}}\int_{0}^{r} s^{a}\psi'(s)\mathrm{d}s.$$ Therefore, the desired estimate follows at once, since $k > \gamma.$
\end{proof}

\begin{rem}\label{rem1} We shall use extensively this lemma for two particular classes of differential operators: first the class  $D_{k},$ then the class $L_{k},$ corresponding to the choice $a_{j} = n+\alpha+j+1.$ For this choice, we have 
$$(a_{j}I+N)(1-\langle z,w \rangle)^{-n-\alpha-j-1} = \dfrac{n+\alpha+j+1}{(1-\langle z,w \rangle)^{n+\alpha+j+2}},$$ and inductively, 
$$ L_{k}(1-\langle z,w \rangle )^{-n-\alpha-1} = \dfrac{c_{k}}{(1-\langle z,w \rangle)^{n+\alpha+k+1}}.$$
\end{rem}

The proof of Lemma $\ref{lem22}$ allows us to define an equivalent norm of $f$ in terms of $M_{k}f.$ Particularly, we will write the equivalent norms of $f$ in terms of $D_{k}f$ and $L_{k}f.$ More precisely, we have the following result:

\begin{cor}\label{lipsch1} Let $D_{k}$ a differential operator of order $k$ defined in $\eqref{pseudok}$ and $L_{k}$ a differential operator of order $k$ defined in Remark $\ref{rem1}.$ For vector-valued holomorphic functions, the following assertions are equivalent:
\begin{enumerate}
\item[$(1)$] $f \in \Gamma_{\gamma}(\mathbb{B}_{n},X).$
\item[$(2)$] There exists an integer $k > \gamma$  such that
$$\sup_{z \in \mathbb{B}_{n}}(1-|z|^2)^{k-\gamma}\|D_{k}f(z)\|_{X} < \infty.$$
\item[$(3)$] There exists an integer $k > \gamma$  such that
$$\sup_{z \in \mathbb{B}_{n}}(1-|z|^2)^{k-\gamma}\|L_{k}f(z)\|_{X} < \infty.$$
Moreover, the following are equivalent
\begin{eqnarray*}
\|f\|_{ \Gamma_{\gamma}(\mathbb{B}_{n},X)} & \simeq & \|f(0)\|_{X} + \sup_{z \in \mathbb{B}_{n}}(1-|z|^2)^{k-\gamma}\|D_{k}f(z)\|_{X} \\
& \simeq & \|f(0)\|_{X} + \sup_{z \in \mathbb{B}_{n}}(1-|z|^2)^{k-\gamma}\|L_{k}f(z)\|_{X}.
\end{eqnarray*}
\end{enumerate} 
\end{cor}

The proof of some of the results obtained in this paper will be based on the following lemma. A proof is in \cite{Roc}, but for the sake of completeness, we will recall the proof.

\begin{lem}\label{hank1}
Let $f \in H^{\infty}(\mathbb{B}_{n},X)$ and  $g \in  H^{\infty}(\mathbb{B}_{n},Y^{\star}).$
If $b \in \mathcal{H}(\mathbb{B}_{n}, \mathcal{L}(\overline{X},Y))$ is such that $\eqref{hypo1}$ and  $\eqref{hypimp}$ hold. Then we have
\begin{equation}
\displaystyle \langle h_{b}f,g \rangle_{\alpha,Y} =  \int_{\mathbb{B}_n} \langle b(z)\overline{f(z)}, g(z) \rangle_{Y,Y^{\star}} \mathrm{d}\nu_{\alpha}(z).
\end{equation}
\end{lem}
\begin{proof}
Let $f\in H^{\infty}(\mathbb{B}_{n},X)$ and $g \in H^{\infty}(\mathbb{B}_{n},Y^{\star}).$ By the definition of $\langle \cdot,\cdot \rangle_{\alpha,Y},$ Fubini's theorem, Lemma $\ref{premierlem}$ and the reproducing kernel property, we have:

\begin{eqnarray*}
\langle h_{b}(f),g \rangle_{\alpha,Y} & = & \displaystyle \int_{\mathbb{B}_{n}} \langle h_{b}(f)(z),g(z) \rangle_{Y,Y^{\star}}\mathrm{d}\nu_{\alpha}(z)\\
& = & \displaystyle \int_{\mathbb{B}_{n}} \langle \int_{\mathbb{B}_{n}} \dfrac{b(w)(\overline{f(w)})\mathrm{d}\nu_{\alpha}(w)}{(1 - \langle z,w \rangle)^{n+1+\alpha}}, g(z) \rangle_{Y,Y^{\star}}\mathrm{d}\nu_{\alpha}(z)\\
& = & \displaystyle \int_{\mathbb{B}_{n}}g(z)\left( \int_{\mathbb{B}_{n}} \dfrac{b(w)(\overline{f(w)})\mathrm{d}\nu_{\alpha}(w)}{(1 - \langle z,w \rangle)^{n+1+\alpha}}\right) \mathrm{d}\nu_{\alpha}(z)\\
& = & \displaystyle \int_{\mathbb{B}_{n}}\int_{\mathbb{B}_{n}}g(z)\left(\dfrac{b(w)(\overline{f(w)})}{(1 - \langle z,w \rangle)^{n+1+\alpha}}\right)\mathrm{d}\nu_{\alpha}(w)\mathrm{d}\nu_{\alpha}(z)\\
& = & \displaystyle \int_{\mathbb{B}_{n}}\left( \int_{\mathbb{B}_{n}}\dfrac{g(z)}{(1 - \langle w,z \rangle)^{n+1+\alpha}}\mathrm{d}\nu_{\alpha}(z)\right)\left( b(w)(\overline{f(w)})\right) \mathrm{d}\nu_{\alpha}(w)\\
& = & \displaystyle \int_{\mathbb{B}_{n}} g(w)\left( b(w)(\overline{f(w)})\right) \mathrm{d}\nu_{\alpha}(w)\\
& = & \displaystyle  \int_{\mathbb{B}_n} \langle b(w)\overline{f(w)}, g(w) \rangle_{Y,Y^{\star}} \mathrm{d}\nu_{\alpha}(w).
\end{eqnarray*}
It remains to show that the assumption of Fubini's theorem is fulfilled. Indeed, since $f\in H^{\infty}(\mathbb{B}_{n},X)$ and $g\in H^{\infty}(\mathbb{B}_{n},Y^{\star}),$ by Tonelli's theorem, Theorem $\ref{estimint}$ and relation $\eqref{hypimp}$ we have that
\begin{eqnarray*}
\displaystyle \int_{\mathbb{B}_{n}} \int_{\mathbb{B}_{n}} \left|\dfrac{g(z)\left( b(w)(\overline{f(w)})\right)}{(1 - \langle z,w \rangle)^{n+1+\alpha}} \right|\mathrm{d}\nu_{\alpha}(w)\mathrm{d}\nu_{\alpha}(z) & \lesssim & \displaystyle \int_{\mathbb{B}_{n}} \int_{\mathbb{B}_{n}} \dfrac{\|b(w)\|_{\mathcal{L}(\overline{X},Y)}}{|1 - \langle z,w \rangle|^{n+1+\alpha}} \mathrm{d}\nu_{\alpha}(w)\mathrm{d}\nu_{\alpha}(z)\\
& \lesssim & \displaystyle \int_{\mathbb{B}_{n}} \|b(w)\|_{\mathcal{L}(\overline{X},Y)} \log \left( \dfrac{1}{1-|w|^2} \right)  \mathrm{d}\nu_{\alpha}(w) < \infty.
\end{eqnarray*}
\end{proof}

\begin{lem}\label{corrver} Let $f \in H^{\infty}(\mathbb{B}_{n},X)$ and $z \in \mathbb{B}_n.$ For $b \in \mathcal{H}(\mathbb{B}_{n}, \mathcal{L}(\overline{X},Y))$ satisfying $\eqref{hypo1}$ and $\eqref{hypimp},$ the function
$$g_{z}(w):= \dfrac{f(w)}{(1 - \langle w,z \rangle)^{n+1+\alpha}}, \hspace{1cm} w \in \mathbb{B}_n$$ belongs to $ H^{\infty}(\mathbb{B}_{n},X)$  and the following identity holds:
$$ \displaystyle h_{b}(f)(z) =  C_{k}\int_{\mathbb{B}_n} L_{k}\left( b(w)(\overline{g_{z}(w)})\right)\mathrm{d}\nu_{\alpha+k}(w),$$ where $k$ is any positive integer and $C_{k}$ is a positive constant depending only on $k.$
\end{lem}
\begin{proof}
It is clear that $g_{z} \in H^{\infty}(\mathbb{B}_{n},X).$ By the definition of the little Hankel operator and the reproducing kernel property, we have  
\begin{eqnarray*}
h_b(f)(z) & = &  \int_{\mathbb{B}_n} \dfrac{b(w)\overline{f(w)}}{(1 - \langle z,w \rangle)^{n+1+\alpha}}\mathrm{d}\nu_{\alpha}(w)\\
& = & \displaystyle \int_{\mathbb{B}_n} b(w)\left( \overline{ \dfrac{f(w)}{(1 - \langle w,z \rangle)^{n+1+\alpha}}}\right) \mathrm{d}\nu_{\alpha}(w)\\
& = & \int_{\mathbb{B}_n} b(w)(\overline{g_{z}(w)})\mathrm{d}\nu_{\alpha}(w)\\
& = & \int_{\mathbb{B}_n} b(w)\left( \overline{\int_{\mathbb{B}_n} \dfrac{g_{z}(\zeta)}{(1 - \langle w,\zeta \rangle)^{n+1+\alpha+k}}\mathrm{d}\nu_{\alpha+k}(\zeta)} \right) \mathrm{d}\nu_{\alpha}(w)\\
& = & \int_{\mathbb{B}_n} \left( \int_{\mathbb{B}_n} \dfrac{b(w)(\overline{g_{z}(\zeta)})}{(1 - \langle \zeta,w \rangle)^{n+1+\alpha+k}} \mathrm{d}\nu_{\alpha}(w) \right) \mathrm{d}\nu_{\alpha+k}(\zeta)\\
& = & c_{k}^{-1} \int_{\mathbb{B}_n}L_{k}\left( \int_{\mathbb{B}_n} \dfrac{b(w)(\overline{g_{z}(\zeta)})}{(1 - \langle \zeta,w \rangle)^{n+1+\alpha}} \mathrm{d}\nu_{\alpha}(w) \right) \mathrm{d}\nu_{\alpha+k}(\zeta)\\
& = & c_{k}^{-1} \int_{\mathbb{B}_n}L_{k}\left( b(\zeta)(\overline{g_{z}(\zeta)})\right)\mathrm{d}\nu_{\alpha+k}(\zeta).
\end{eqnarray*}
The assumption of Fubini's theorem is fulfilled. Indeed by $\eqref{hypimp},$ we have that
$$\displaystyle \int_{\mathbb{B}_n} \left\|\int_{\mathbb{B}_n} \dfrac{b(w)(\overline{g_{z}(\zeta)})}{(1 - \langle \zeta,w \rangle)^{n+1+\alpha+k}}\mathrm{d}\nu_{\alpha}(w)\right\|_{Y}\mathrm{d}\nu_{\alpha+k}(\zeta)$$
\begin{eqnarray*}
 & \leq & \displaystyle \|g_{z}\|_{\infty,X} \int_{\mathbb{B}_n} \int_{\mathbb{B}_n} \dfrac{\|b(w)\|_{\mathcal{L}(\overline{X},Y)}}{|1 - \langle w,\zeta \rangle|^{n+1+\alpha+k}} \mathrm{d}\nu_{\alpha}(w)\mathrm{d}\nu_{\alpha+k}(\zeta)\\
& = & \displaystyle \|g_{z}\|_{\infty,X} \int_{\mathbb{B}_n} \|b(w)\|_{\mathcal{L}(\overline{X},Y)}\\ & \times &\left(\int_{\mathbb{B}_n} \dfrac{\mathrm{d}\nu_{\alpha+k}(\zeta)}{|1 - \langle w,\zeta \rangle|^{n+1+\alpha+k}}\right) \mathrm{d}\nu_{\alpha}(w)\\
& \leq & \displaystyle \|g_{z}\|_{\infty,X}\int_{\mathbb{B}_n} \|b(w)\|_{\mathcal{L}(\overline{X},Y)}\left( \log \dfrac{1}{1 - |w|^2} \right) \mathrm{d}\nu_{\alpha}(w)\\
& < & \infty.
\end{eqnarray*}
\end{proof}

\section{\textbf{The Proof of Theorem $\ref{profD}$}}

\begin{proof} We first suppose that  $g \in \Gamma_{\gamma}(\mathbb{B}_{n},X^{\star}),$ with $\gamma = (n+1+\alpha)\left( \frac{1}{p}-1\right) .$ Given a positive integer $k > \gamma,$ we define the functional

$$\displaystyle \wedge_{g}: A^p_{\alpha}(\mathbb{B}_{n},X) \longrightarrow \mathbb{C}$$ 
$$f \mapsto \wedge_{g}(f) = c_{k}\int_{\mathbb{B}_n} \langle f(z),D_{k}g(z) \rangle_{X,X^{\star}}(1 - |z|^2)^{k}\mathrm{d}\nu_{\alpha}(z),$$ where $c_k$ is the positive constant in Lemma $\ref{primer}.$ It is clear that
$\wedge_{g}$ is linear and is well defined on $A^p_{\alpha}(\mathbb{B}_{n},X).$
Indeed, let $f \in A^p_{\alpha}(\mathbb{B}_{n},X).$ By Lemma $\ref{lemm22},$ we have 
\begin{eqnarray*}
|\wedge_{g}(f)| & = &  \displaystyle c_{k}\left|\int_{\mathbb{B}_n} \langle f(z),D_{k}g(z) \rangle_{X,X^{\star}}(1-|z|^2)^k\mathrm{d}\nu_{\alpha}(z) \right| \\
& \leq & \displaystyle c_{k} \int_{\mathbb{B}_n} \|f(z)\|_{X}\|D_{k}g(z)\|_{X^{\star}}(1-|z|^2)^k \mathrm{d}\nu_{\alpha}(z)\\
& = & \displaystyle c_{k} \int_{\mathbb{B}_n} (1-|z|^2)^{k-\gamma}\|D_{k}g(z)\|_{X^{\star}}(1-|z|^2)^{\gamma}\|f(z)\|_{X} \mathrm{d}\nu_{\alpha}(z)\\
& \leq & \displaystyle c_{k} \sup_{z \in \mathbb{B}_{n}}(1-|z|^2)^{k-\gamma} \|D_{k}g(z)\|_{X^{\star}}\int_{\mathbb{B}_n} (1-|z|^2)^{\gamma} \|f(z)\|_{X}\mathrm{d}\nu_{\alpha}(z)\\
& \lesssim & \displaystyle \|g\|_{\Gamma_{\gamma}(\mathbb{B}_{n},X^{\star})}\int_{\mathbb{B}_n} (1-|z|^2)^{(\frac{1}{p}-1)(n+1+\alpha)} \|f(z)\|_{X}\mathrm{d}\nu_{\alpha}(z)\\
& \lesssim & \|g\|_{\Gamma_{\gamma}(\mathbb{B}_{n},X^{\star})}\|f\|_{p,\alpha,X}.
\end{eqnarray*}
We conclude that $\wedge_{g}$ is bounded on 
$A^p_{\alpha}(\mathbb{B}_{n},X)$ and  $\|\wedge_{g}\| \lesssim \|g\|_{\Gamma_{\gamma}(\mathbb{B}_{n},X^{\star})}.$\\

Conversely, let $\wedge$ be a bounded linear functional on $A^p_{\alpha}(\mathbb{B}_{n},X).$ Let us show that there exists $g \in \Gamma_{\gamma}(\mathbb{B}_{n},X^{\star}),$ with $\gamma = (n+1+\alpha)\left(\frac{1}{p}-1\right)$ such that $\wedge = \wedge_{g}.$
Since $A^2_{\alpha}(\mathbb{B}_{n},X) \subset A^p_{\alpha}(\mathbb{B}_{n},X)$ and $\wedge$ is bounded on $A^p_{\alpha}(\mathbb{B}_{n},X),$ $\wedge$ is also bounded on $A^2_{\alpha}(\mathbb{B}_{n},X).$ Then by Theorem $\ref{dual1},$ there exists $g \in A^2_{\alpha}(\mathbb{B}_{n},X^{\star})$ such that 
\begin{equation}
\displaystyle \wedge(f) = \int_{\mathbb{B}_n} \langle f(z),g(z) \rangle_{X,X^{\star}} \mathrm{d}\nu_{\alpha}(z),\label{intpair2}
\end{equation}
for all $f \in A^2_{\alpha}(\mathbb{B}_{n},X).$ Since $g \in A^2_{\alpha}(\mathbb{B}_{n},X^{\star}),$ for any positive integer $k,$ we have $D_{k}g \in  A^2_{\alpha+k}(\mathbb{B}_{n},X^{\star}).$ Applying Lemma $\ref{primer}$ in $\eqref{intpair2},$ we obtain that

\begin{equation}
\displaystyle \wedge(f) = c_{k}\int_{\mathbb{B}_n} \langle f(z),D_{k}g(z) \rangle_{X,X^{\star}}(1 - |z|^2)^k \mathrm{d}\nu_{\alpha}(z),\label{intpair4}
\end{equation}
for all $f \in A^2_{\alpha}(\mathbb{B}_{n},X).$ Now, we fix $x \in X,$ $w \in \mathbb{B}_n$ and an integer $k > \gamma.$ Let
$$f(z) = \dfrac{(1-|w|^2)^{k-\gamma}}{(1-\langle z,w \rangle)^{n+1+\alpha+k}}x, \hspace{0,3cm} z\in \mathbb{B}_{n}.$$ By Theorem $\ref{estimint},$ we have that $f \in A^2_{\alpha}(\mathbb{B}_{n},X).$ Proposition $\ref{prop1}$ and $\eqref{intpair4},$ give us
\begin{eqnarray*}
\wedge(f) & = & \displaystyle c_{k}\int_{\mathbb{B}_n} \langle f(z),D_{k}g(z) \rangle_{X,X^{\star}}(1 - |z|^2)^k\mathrm{d}\nu_{\alpha}(z)\\
& = & \displaystyle c_{k}\int_{\mathbb{B}_n} \big\langle \dfrac{(1-|w|^2)^{k-\gamma}}{(1-\langle z,w \rangle)^{n+1+\alpha+k}}x, D_{k}g(z) \big\rangle_{X,X^{\star}}(1 - |z|^2)^k\mathrm{d}\nu_{\alpha}(z)\\
& = & \displaystyle \dfrac{c_{\alpha}c_{k}}{c_{\alpha+k}}(1-|w|^2)^{k-\gamma} \big\langle x, \int_{\mathbb{B}_n}\dfrac{D_{k}g(z)}{(1-\langle w,z \rangle)^{n+1+\alpha+k}}\mathrm{d}\nu_{\alpha+k}(z)\big\rangle_{X,X^{\star}}\\
& = & \displaystyle\dfrac{c_{\alpha}c_{k}}{c_{\alpha+k}}(1-|w|^2)^{k-\gamma} \big\langle x,D_{k}g(w)\big\rangle_{X,X^{\star}}.
\end{eqnarray*} 
By Theorem $\ref{estimint},$ $f \in A^p_{\alpha}(\mathbb{B}_{n},X)$ and $\|f\|_{p,\alpha,X} \lesssim \|x\|_{X}.$ Since $x$ is arbitrary, by duality, we have that 
\begin{eqnarray*}
\|D_{k}g(w)\|_{X^{\star}} & = & \sup_{\|x\|_{X} = 1} |\langle x,D_{k}g(w) \rangle_{X,X^{\star}}|\\
\\
& = & \dfrac{c_{\alpha+k}}{c_{\alpha}c_{k}}\sup_{\|x\|_{X} = 1} \dfrac{1}{(1-|w|^2)^{k-\gamma}}|\wedge(f)| \\
\\
& \lesssim & \displaystyle \sup_{\|x\|_{X} = 1} \dfrac{1}{(1-|w|^2)^{k-\gamma}} \| \wedge \| \|f\|_{p,\alpha,X}\\
\\
& \lesssim &  \sup_{\|x\|_{X} = 1} \dfrac{\|\wedge\|}{(1-|w|^2)^{k-\gamma}}\|x\|_{X}\\
\\
& \lesssim & \dfrac{\|\wedge\|}{(1-|w|^2)^{k-\gamma}}.
\end{eqnarray*} 
According to Corollary $\ref{lipsch1},$ we conclude that $$g \in \Gamma_{\gamma}(\mathbb{B}_{n},X^{\star}) ~~~~\mbox{and}~~~~ \|g\|_{\Gamma_{\gamma}(\mathbb{B}_{n},X^{\star})} \lesssim \|\wedge\|,$$ with $\gamma = (n+1+\alpha)\left( \frac{1}{p}-1\right).$
To finish the proof, it remains to show that $\eqref{intpair2}$ remains true for functions in $A^p_{\alpha}(\mathbb{B}_{n},X)$ which is a direct consequence of the density in Corollary $\ref{densberg}$.
\end{proof}

\section{\textbf{The Proofs of Theorem $\ref{principe1}$ and Corollary $\ref{cons1}$}}
In this section, we will give the proofs of Theorem $\ref{principe1}$ and Corollary $\ref{cons1}$.
\subsection{Proof of Theorem $\ref{principe1}$}
\begin{proof}
First assume that $h_{b}$ extends to a bounded operator from $A^p_{\alpha}(\mathbb{B}_{n},X)$ to $A^q_{\alpha}(\mathbb{B}_{n},Y),$ with $q < 1.$ Let $\|h_{b}\| := \|h_{b}\|_{A^{p}_{\alpha}(\mathbb{B}_{n},X) \longrightarrow A^{q}_{\alpha}(\mathbb{B}_{n},Y)}.$ We want to show that $b \in \Gamma_{\gamma}(\mathbb{B}_{n},\mathcal{L}(\overline{X},Y)).$
Since $h_{b}: A^p_{\alpha}(\mathbb{B}_{n},X) \longrightarrow A^{q}_{\alpha}(\mathbb{B}_{n},Y)$ is a bounded operator, we have by Theorem $\ref{profD}$ that
$$|\langle h_{b}(f),g \rangle_{\alpha,Y}| \lesssim \|h_{b}\|\|f\|_{p,\alpha,X} \|g\|_{\Gamma_{\beta}(\mathbb{B}_{n},Y^{\star})},$$
for every $f \in A^p_{\alpha}(\mathbb{B}_{n},X)$ and $g \in \Gamma_{\beta}(\mathbb{B}_{n},Y^{\star}),$ with $\beta = (n+1+\alpha)\left(\frac{1}{q} -1\right).$ Let $x \in X,$ $y^{\star} \in Y^{\star},$ $w \in \mathbb{B}_{n}$ and an integer $k$ such that $k >\gamma = (n+1+\alpha)\left(\frac{1}{p}-1\right).$ Let $g(z) = y^{\star},$ and $f(z) = \dfrac{(1-|w|^2)^{k-\gamma}}{(1-\langle z,w \rangle)^{n+1+\alpha+ k}}x.$ It is clear that $f\in H^{\infty}(\mathbb{B}_{n},X)$
and $g \in \Gamma_{\beta}(\mathbb{B}_{n},Y^{\star}),$ with $\|g\|_{\Gamma_{\beta}(\mathbb{B}_{n},Y^{\star})} = \|y^{\star}\|_{Y^{\star}}.$ We also have by Theorem $\ref{estimint}$ that $f \in A^p_{\alpha}(\mathbb{B}_{n},X),$ with $\|f\|_{p,\alpha,X} \lesssim \|x\|_{X}.$ Hence 
\begin{equation}
|\langle h_{b}(f),g \rangle_{\alpha,Y}| \lesssim \|h_{b}\|\|x\|_{X} \|y{\star}\|_{Y^{\star}},\label{refcor1}
\end{equation}
Applying Lemma $\ref{hank1}$ and the reproducing kernel property, we have that
\begin{eqnarray*}
|\langle h_{b}(f),g \rangle_{\alpha,Y}| & = & \displaystyle \left| \int_{\mathbb{B}_n} \big\langle b(z)\left( \overline{\dfrac{(1-|w|^2)^{k-\gamma}}{(1-\langle z,w \rangle)^{n+1+\alpha+k}}x} \right), y^{\star} \big\rangle_{Y,Y^{\star}} \mathrm{d}\nu_{\alpha}(z) \right| \\
\\
& = & \displaystyle (1-|w|^2)^{k-\gamma}\left|\int_{\mathbb{B}_n} \big\langle b(z)\left( \dfrac{\overline{x}}{(1-\langle w,z \rangle)^{n+1+\alpha+k}} \right), y^{\star} \big\rangle_{Y,Y^{\star}} \mathrm{d}\nu_{\alpha}(z) \right| \\
\\
& = & \displaystyle (1-|w|^2)^{k-\gamma}\left|\int_{\mathbb{B}_n} \big\langle \dfrac{b(z)\left( \overline{x}\right)}{(1-\langle w,z \rangle)^{n+1+\alpha+k}} , y^{\star} \big\rangle_{Y,Y^{\star}} \mathrm{d}\nu_{\alpha}(z) \right| \\
\\
& = & \displaystyle (1-|w|^2)^{k-\gamma}\left| \big\langle \int_{\mathbb{B}_n}\dfrac{b(z)\left( \overline{x}\right)}{(1-\langle w,z \rangle)^{n+1+\alpha+k}}\mathrm{d}\nu_{\alpha}(z) , y^{\star} \big\rangle_{Y,Y^{\star}}  \right| \\
\\
& = & \displaystyle \frac{(1-|w|^2)^{k-\gamma}}{c_{k}}\left| \big\langle \int_{\mathbb{B}_n}L_{k}\left( \dfrac{b(z)\left( \overline{x}\right)}{(1-\langle w,z \rangle)^{n+1+\alpha}}\right) \mathrm{d}\nu_{\alpha}(z) , y^{\star} \big\rangle_{Y,Y^{\star}}  \right| \\
\\
& = & \displaystyle \frac{(1-|w|^2)^{k-\gamma}}{c_{k}}\left| \big\langle L_{k}\left(\int_{\mathbb{B}_n} \dfrac{b(z)\left( \overline{x}\right)}{(1-\langle w,z \rangle)^{n+1+\alpha}}\mathrm{d}\nu_{\alpha}(z)\right) , y^{\star} \big\rangle_{Y,Y^{\star}}  \right| \\
\\
& = & \displaystyle \frac{(1-|w|^2)^{k-\gamma}}{c_{k}} \left| \big\langle L_{k}\left( b(w)\left( \overline{x}\right) \right) , y^{\star} \big\rangle_{Y,Y^{\star}}  \right|.
\end{eqnarray*}
Thus,
\begin{equation}
|\langle h_{b}(f),g \rangle_{\alpha,Y}| = \displaystyle \frac{(1-|w|^2)^{k-\gamma}}{c_{k}} \left| \big\langle L_{k}\left( b(w)\left( \overline{x}\right) \right) , y^{\star} \big\rangle_{Y,Y^{\star}}  \right|.\label{refcor2}
\end{equation} 
From  $\eqref{refcor1},$ $\eqref{refcor2}$ and the fact that $\|x\|_{X} = \|\overline{x}\|_{\overline{X}},$ we deduce that 
\begin{equation}
(1-|w|^2)^{k-\gamma} \left| \big\langle L_{k}\left( b(w)\left( \overline{x}\right) \right) , y^{\star} \big\rangle_{Y,Y^{\star}}  \right| \lesssim \|h_{b}\|\|\overline{x}\|_{\overline{X}} \|y^{\star}\|_{Y^{\star}}.
\end{equation} 
Since $x$ and $y^{\star}$ are arbitrary, we get that
$$ \sup_{w \in \mathbb{B}_n}(1-|w|^2)^{k-\gamma}\|L_{k}b(w)\|_{\mathcal{L}(\overline{X},Y^{\star})} \lesssim \|h_{b}\|.$$ That is, $b \in \Gamma_{\gamma}(\mathbb{B}_{n},\mathcal{L}(\overline{X},Y^{\star}))$ ~~with~~ $\|b\|_{\Gamma_{\gamma}(\mathbb{B}_{n},\mathcal{L}(\overline{X},Y))} \lesssim \|h_{b}\|.$\\

Conversely, assume that $b \in \Gamma_{\gamma}(\mathbb{B}_{n},\mathcal{L}(\overline{X},Y))$ and let us prove that
 $h_{b}$ extends to a bounded operator from $A^p_{\alpha}(\mathbb{B}_{n},X)$ to $A^{1,\infty}_{\alpha}(\mathbb{B}_{n},Y).$
Choose a positive integer $k>\gamma,$ and let $f \in H^{\infty}(\mathbb{B}_{n},X).$ Taking $g_{z}(w) = \dfrac{f(w)}{(1 - \langle w,z \rangle)^{n+1+\alpha}},$ with $w\in \mathbb{B}_{n}$ and applying Lemma $\ref{corrver},$ Lemma $\ref{vers2}$ and the assumption we obtain
\begin{eqnarray*}
\|h_{b}f(z)\|_{Y} & = & \displaystyle  \left\Vert \int_{\mathbb{B}_n} \dfrac{b(w)(\overline{f(w)})}{(1-\langle z,w \rangle)^{n+1+\alpha}}\mathrm{d}\nu_{\alpha}(w) \right\Vert_{Y} \\
\\
& = & \displaystyle c_{k}\left\Vert \int_{\mathbb{B}_n} L_{k}\left( b(w)\overline{g_{z}(w)}\right)\mathrm{d}\nu_{\alpha+k}(w) \right\Vert_{Y} \\
 \\
& = & \displaystyle c_{k}\left\Vert \int_{\mathbb{B}_n} \dfrac{L_{k}\left( b(w)\overline{f(w)}\right) }{(1-\langle z,w \rangle)^{n+1+\alpha}}\mathrm{d}\nu_{\alpha+k}(w) \right\Vert_{Y} \\
\\
& \leq  & \displaystyle c_{k} \int_{\mathbb{B}_n} \left\Vert \dfrac{L_{k}\left( b(w)\overline{f(w)}\right) }{(1-\langle z,w \rangle)^{n+1+\alpha}} \right\Vert_{Y} \mathrm{d}\nu_{\alpha+k}(w)\\
\\
& \leq & \displaystyle \dfrac{c_{k}c_{\alpha+k}}{c_{\alpha}}\int_{\mathbb{B}_n}\dfrac{(1-|w|^2)^{k-\gamma}\|L_{k}b(w)\|_{\mathcal{L(\overline{X}, Y)}}\|\overline{f(w)}\|_{\overline{X}}}{|1 - \langle z,w \rangle|^{n+1+\alpha}}(1 -|w|^2)^{\gamma}\mathrm{d}\nu_{\alpha}(w) \\
\\ 
& \lesssim & \displaystyle \|b\|_{\Gamma_{\gamma}(\mathbb{B}_{n},\mathcal{L}(\overline{X},Y))}\int_{\mathbb{B}_n}
\dfrac{(1-|w|^2)^{\gamma}\|f(w)\|_{X}}{|1 - \langle z,w \rangle|^{n+1+\alpha}} \mathrm{d}\nu_{\alpha}(w)\\
\\
& = & \|b\|_{\Gamma_{\gamma}(\mathbb{B}_{n},\mathcal{L}(\overline{X},Y))}P^{+}_{\alpha} g(z),
\end{eqnarray*}
where the reproducing kernel is justified by $\eqref{hypo1}$ and $$\displaystyle P^{+}_{\alpha} g(z) = \int_{\mathbb{B}_n}
\dfrac{(1-|w|^2)^{\gamma}\|f(w)\|_{X}}{|1 - \langle z,w \rangle|^{n+1+\alpha}} \mathrm{d}\nu_{\alpha}(w)$$ is the positive Bergman operator of the positive function $\displaystyle g(z) = (1-|z|^2)^{\gamma}\|f(z)\|_{X}.$\\
Now, let $\lambda >0.$ We have that
$$\nu_{\alpha} (\lbrace z \in \mathbb{B}_n : \|h_{b}f(z)\|_{Y} > \lambda \rbrace ) \leq \nu_{\alpha}(\lbrace z \in \mathbb{B}_n : c_{k}\|b\|_{\Gamma_{\gamma}(\mathbb{B}_{n},\mathcal{L}(\overline{X},Y))}P^{+}_{\alpha}g(z)  > \lambda \rbrace ).$$
Since the positive Bergman operator $P^{+}_{\alpha} : L^{1}_{\alpha}(\mathbb{B}_n) \longrightarrow L^{1,\infty}_{\alpha}(\mathbb{B}_n)$ is bounded (cf. e.g \cite{Bekolle}), there exists a constant $c$ such that
\begin{eqnarray*}
\nu_{\alpha}(\lbrace z \in \mathbb{B}_n : c_{k}\|b\|_{\Gamma_{\gamma}(\mathbb{B}_{n},\mathcal{L}(\overline{X},Y)}P^{+}_{\alpha}g(z)  > \lambda \rbrace ) & \leq & \displaystyle\dfrac{c}{\frac{\lambda}{c_{k}\|b\|}_{\Gamma_{\gamma}(\mathbb{B}_{n},\mathcal{L}(\overline{X},Y))}}\|g\|_{L^{1}_{\alpha}(\mathbb{B}_n)}\\
\\
& = & \displaystyle \dfrac{cc_{k}}{\lambda}\|b\|_{\Gamma_{\gamma}(\mathbb{B}_{n},\mathcal{L}(\overline{X},Y))}\|g\|_{L^{1}_{\alpha}(\mathbb{B}_n)}.
\end{eqnarray*}
Applying Lemma $\ref{lemm22}$ to the function $f,$ we get that
\begin{eqnarray*}
\|g\|_{L^{1}_{\alpha}(\mathbb{B}_n)} & = & \displaystyle \int_{\mathbb{B}_n}(1-|z|^2)^{\gamma}\|f(z)\|_{X} \mathrm{d}\nu_{\alpha}(z)\\
& = &  \displaystyle \int_{\mathbb{B}_n}(1-|z|^2)^{(\frac{1}{p}-1)(n+1+\alpha)} \|f(z)\|_{X} \mathrm{d}\nu_{\alpha}(z)\\
& \leq & \|f\|_{p,\alpha,X}.
\end{eqnarray*}
It follows that $$\lambda \nu_{\alpha} (\lbrace z \in \mathbb{B}_n : \|h_{b}f(z)\|_{Y} > \lambda \rbrace ) \lesssim \|b\|_{\Gamma_{\gamma}(\mathbb{B}_{n},\mathcal{L}(\overline{X},Y))} \|f\|_{p,\alpha,X}$$ for all $\lambda > 0.$ Therefore, 
$h_{b}$ extends into a bounded operator from $A^p_{\alpha}(\mathbb{B}_{n},X)$ to $A^{1,\infty}_{\alpha}(\mathbb{B}_{n},Y)$ with $$\|h_{b}\|_{A^p_\alpha(\mathbb{B}_{n},X) \longrightarrow A^{1,\infty}_{\alpha}(\mathbb{B}_{n},Y)} \lesssim \|b\|_{\Gamma_{\gamma}(\mathbb{B}_{n},\mathcal{L}(\overline{X},Y))}.$$ By density of $H^{\infty}(\mathbb{B}_{n},X)$ on $A^p_\alpha(\mathbb{B}_{n},X),$ the proof of the theorem is finished.
\end{proof}

\subsection{Proof of Corollary $\ref{cons1}$}

\begin{proof}
Just apply Lemma $\ref{nessa}$ and the second part of Theorem $\ref{principe1}$ to conclude.
\end{proof}

\section{\textbf{The Proof of Theorem $\ref{principal2}$}}

\begin{proof}
We first prove the sufficiency of the theorem. We assume that there exists a constant $C' > 0$ such that
$$ \|N^{k}b(w)\|_{\mathcal{L}(\overline{X},Y)} \leq \dfrac{C'}{(1-|w|^2)^{k-\gamma}}\left(\log \dfrac{1}{1-|w|^2} \right)^{-1}.$$ Likewise by Corollary $\ref{lipsch1},$  we have that, there exists a constant $C >0$ such that
$$ \|L_{k}b(w)\|_{\mathcal{L}(\overline{X},Y)} \leq \dfrac{C}{(1-|w|^2)^{k-\gamma}}\left(\log \dfrac{1}{1-|w|^2} \right)^{-1}.$$
Applying Lemma $\ref{corrver}$ for any $f \in H^{\infty}(\mathbb{B}_{n},X),$ we get 
\begin{eqnarray*}
\displaystyle \int_{\mathbb{B}_n} \dfrac{ b(w)(\overline{f(w)})}{(1 - \langle z,w \rangle)^{n+1+\alpha}}\mathrm{d}\nu_{\alpha}(w) = c_{k}\int_{\mathbb{B}_n} \dfrac{L_{k}b(w)(\overline{f(w)})}{(1 - \langle z,w \rangle)^{n+1+\alpha}}\mathrm{d}\nu_{\alpha+k}(w). 
\end{eqnarray*}
Thus, by the assumption, Lemma $\ref{corrver}$ and Lemma $\ref{lemm22}$ we have that 
\begin{eqnarray*} 
\|h_{b}f\|_{A^{1}_{\alpha}(\mathbb{B}_{n},Y)} & = & \displaystyle \int_{\mathbb{B}_n} \left\Vert c_{k}\int_{\mathbb{B}_n}\frac{ L_{k}b(w)(\overline{f(w)})}{(1- \langle z,w \rangle)^{n+1+\alpha}} \mathrm{d}\nu_{\alpha+k}(w) \right\Vert_{Y} \mathrm{d}\nu_{\alpha}(z)\\
& \lesssim & \displaystyle \int_{\mathbb{B}_n} \int_{\mathbb{B}_n} \left\Vert \dfrac{ L_{k}b(w)(\overline{f(w)})}{(1- \langle z,w \rangle)^{n+1+\alpha}} \right\Vert_{Y}(1-|w|^2)^{k} \mathrm{d}\nu_{\alpha}(w) \mathrm{d}\nu_{\alpha}(z)\\
& \lesssim & \displaystyle  \int_{\mathbb{B}_n} \int_{\mathbb{B}_n} \dfrac{\Vert L_{k}b(w) \Vert_{\mathcal{L}(\overline{X},Y)} }{|1 - \langle z,w \rangle |^{n+1+\alpha}}\|\overline{f(w)}\|_{\overline{X}}(1-|w|^2)^{k} \mathrm{d}\nu_{\alpha}(w) \mathrm{d}\nu_{\alpha}(z)\\
& = & \displaystyle \int_{\mathbb{B}_n} \left( \int_{\mathbb{B}_n} \dfrac{1}{|1 - \langle z,w \rangle |^{n+1+\alpha}}\mathrm{d}\nu_{\alpha}(z)\right) \Vert L_{k}b(w) \Vert_{\mathcal{L}(\overline{X},Y)} \|\overline{f(w)}\|_{\overline{X}}(1-|w|^2)^{k} \mathrm{d}\nu_{\alpha}(w)\\
& \lesssim & \displaystyle \int_{\mathbb{B}_n} \left(\log\dfrac{1}{1-|w|^2}\right)\|\overline{f(w)}\|_{\overline{X}} \dfrac{(1-|w|^2)^{k}}{(1-|w|^2)^{k-\gamma}}\left(\log \dfrac{1}{1-|w|^2} \right)^{-1}\mathrm{d}\nu_{\alpha}(w)\\
& = & \displaystyle \int_{\mathbb{B}_n} \|f(w)\|_{X}(1-|w|^2)^{\gamma}\mathrm{d}\nu_{\alpha}(w)\\
& = & \displaystyle \int_{\mathbb{B}_n} \|f(w)\|_{X}(1-|w|^2)^{(\frac{1}{p}-1)(n+1+\alpha)}\mathrm{d}\nu_{\alpha}(w)\\
& \lesssim & \|f\|_{p,\alpha,X}.
\end{eqnarray*}
Conversely, we assume that $h_{b}$ extends into a bounded operator from  $A^p_{\alpha}(\mathbb{B}_{n},X)$ to $A^{1}_{\alpha}(\mathbb{B}_{n},Y).$ Then for all $f \in H^{\infty}(\mathbb{B}_{n},X)$ and $g \in \mathcal{B}(\mathbb{B}_{n},Y^{\star}),$ we have 
\begin{equation}
 |\langle h_{b}(f), g \rangle_{\alpha,Y}| \leq \|h_b\| \|f\|_{p,\alpha,X}\|g\|_{\mathcal{B}(\mathbb{B}_{n},Y^{\star})}.\label{aquat1}
\end{equation} 
We choose the particular function $g(z) = y^{\star},$ with $y^{\star} \in Y^{\star}.$ Applying Lemma $\ref{hank1},$  relation $\eqref{aquat1}$ becomes
\begin{eqnarray*}
\left| \displaystyle \int_{\mathbb{B}_n} \big\langle h_{b}f(z),y^{\star} \big\rangle_{Y,Y^{\star}}\mathrm{d}\nu_{\alpha}(z)\right| & = & \displaystyle \left|\big\langle \int_{\mathbb{B}_n} b(z)\overline{f(z)}\mathrm{d}\nu_{\alpha}(z),y^{\star} \big\rangle_{Y,Y^{\star}} \right|\\
& \leq & \|h_b\|\|f\|_{p,\alpha,X}\|y^{\star}\|_{Y^{\star}}.
\end{eqnarray*}
Thus 
\begin{equation}
\left| \displaystyle \int_{\mathbb{B}_n} \big\langle b(z)\overline{f(z)},y^{\star} \big\rangle_{Y,Y^{\star}}\mathrm{d}\nu_{\alpha}(z)\right|\leq \|h_b\|\|f\|_{p,\alpha,X}\|y^{\star}\|_{Y^{\star}}\label{ref2}
\end{equation}
for all $f \in H^{\infty}(\mathbb{B}_{n},X)$ and $y^{\star} \in Y^{\star}.$
Now, take $x \in X,$ $y^{\star} \in Y^{\star},$ and an integer $k$ such that $k > \gamma.$ 
Fix $w \in \mathbb{B}_{n}$ and put $$f(z) = \dfrac{(1-|w|^2)^{k-\gamma}}{(1-\langle z,w \rangle)^{n+1+\alpha+k}}x~~; \hspace{1cm} g(z) = \log(1- \langle z,w \rangle)y^{\star},$$ where $\log$ is the principal branch of the logarithm. 
Since $f \in H^{\infty}(\mathbb{B}_{n},X)$ and $g \in \mathcal{B}(\mathbb{B}_{n},Y^{\star}),$ by relation $\eqref{aquat1},$ we have that
\begin{equation}
|\langle h_{b}f,g \rangle|_{\alpha,Y} \leq \|h_b\|\|x\|_{X}\|y^{\star}\|_{Y^{\star}} \label{ref3}.
\end{equation}
Applying Lemma $\ref{hank1}$ for those particular vector-valued holomorphic functions $f$ and $g$ and using the fact that $$\log(1- \langle w,z \rangle) = \log(1-|w|^2)+\log \left(\dfrac{1- \langle w,z \rangle}{1-|w|^2}\right),$$ we obtain 
\begin{eqnarray*}
\langle h_{b}f,g \rangle_{\alpha,Y} & = & \displaystyle \int_{\mathbb{B}_n} \big\langle b(z)\overline{\left( \dfrac{(1-|w|^2)^{k-\gamma}}{(1-\langle z,w \rangle)^{n+1+\alpha+k}}x\right)}, \log(1-\langle z,w \rangle)y^{\star} \big\rangle_{Y,Y^{\star}} \mathrm{d}\nu_{\alpha}(z)\\
\\
& = & \displaystyle \big\langle \int_{\mathbb{B}_n} b(z)\left[ \dfrac{(1-|w|^2)^{k-\gamma} \log(1-\langle w,z \rangle)}{(1-\langle w,z \rangle)^{n+1+\alpha+k}}\overline{x}\right] \mathrm{d}\nu_{\alpha}(z),y^{\star} \big\rangle_{Y,Y^{\star}} \\
\\
& = & \displaystyle \big\langle \int_{\mathbb{B}_n} \dfrac{b(z)(\overline{x})(1-|w|^2)^{k-\gamma}\log(1-|w|^2)}{(1-\langle w,z \rangle)^{n+1+\alpha+k}}\mathrm{d}\nu_{\alpha}(z),y^{\star} \big\rangle_{Y,Y^{\star}} \\
\\
& + &  \displaystyle \big\langle \int_{\mathbb{B}_n} b(z)\left[ \dfrac{(1-|w|^2)^{k-\gamma}}{(1-\langle w,z \rangle)^{n+1+\alpha+k}}\log\left( \dfrac{1-\langle w,z \rangle}{1-|w|^2} \right)\overline{x}\right] \mathrm{d}\nu_{\alpha}(z),y^{\star} \big\rangle_{Y,Y^{\star}}\\
\\
& = & \displaystyle \big\langle (1-|w|^2)^{k-\gamma}\log(1-|w|^2)\int_{\mathbb{B}_n} \dfrac{b(z)(\overline{x})\mathrm{d}\nu_{\alpha}(z)}{(1-\langle w,z \rangle)^{n+1+\alpha+k}},y^{\star} \big\rangle_{Y,Y^{\star}} \\
\\
& + & \displaystyle \big\langle \int_{\mathbb{B}_n} b(z)\left(\overline{\dfrac{(1-|w|^2)^{k-\gamma}}{(1-\langle z,w \rangle)^{n+1+\alpha+k}}\log \left(\dfrac{1- \langle z,w \rangle}{1-|w|^2}\right)x}\right) \mathrm{d}\nu_{\alpha}(z),y^{\star} \big\rangle_{Y,Y^{\star}}\\
\\
& = &  \displaystyle (1-|w|^2)^{k-\gamma}\log(1-|w|^2) \big\langle L_{k}\left(\int_{\mathbb{B}_n} \dfrac{b(z)(\overline{x})}{(1-\langle w,z \rangle)^{n+1+\alpha}}\mathrm{d}\nu_{\alpha}(z)\right) ,y^{\star} \big\rangle_{Y,Y^{\star}} \\
\\
& + & \displaystyle \big\langle \int_{\mathbb{B}_n} b(z)\left(\overline{f(z)\log \left(\dfrac{1- \langle z,w \rangle}{1-|w|^2}\right)}\right) \mathrm{d}\nu_{\alpha}(z),y^{\star} \big\rangle_{Y,Y^{\star}}\\
\\
& = &  \displaystyle (1-|w|^2)^{k-\gamma}\log(1-|w|^2) \langle L_{k}(b(w)(\overline{x})),y^{\star} \rangle_{Y,Y^{\star}}  +  \displaystyle \langle \int_{\mathbb{B}_n} b(z)(\overline{\varphi(z)}) \mathrm{d}\nu_{\alpha}(z),y^{\star} \rangle_{Y,Y^{\star}},
\end{eqnarray*}
where $ \varphi(z) = f(z)\log \left(\dfrac{1- \langle z,w \rangle}{1-|w|^2}\right).$
Therefore, we can write 
$ \langle h_{b}f,g \rangle_{\alpha,Y} = I_{1}+I_{2},$ with 
$$\displaystyle I_{1} = (1-|w|^2)^{k-\gamma}\log(1-|w|^2) \langle L_{k}(b(w)(\overline{x})),y^{\star} \rangle_{Y,Y^{\star}}$$ and 
$$ I_{2} =  \displaystyle \big\langle \int_{\mathbb{B}_n} b(z)(\overline{\varphi(z)}) \mathrm{d}\nu_{\alpha}(z),y^{\star} \big\rangle_{Y,Y^{\star}}.$$
Applying Lemma $\ref{lemm1}$ with $\delta = p,$ and $\beta = p(k-\gamma),$ we obtain that
\begin{eqnarray*}
\|\varphi\|_{p,\alpha,X} & = & \displaystyle \left( \int_{\mathbb{B}_n}\left|\log\left(\dfrac{1-\langle z,w \rangle}{1-|w|^2}\right)\right|^{p}\dfrac{(1-|w|^2)^{p(k-\gamma)}}{|1-\langle z,w \rangle|^{p(n+1+\alpha+k)}}\|x\|^p_{X}\mathrm{d}\nu_{\alpha}(z)\right)^{1/p}\\
\\
& = & \displaystyle \|x\|_{X}\left( \int_{\mathbb{B}_n}\left|\log\left(\dfrac{1-\langle z,w \rangle}{1-|w|^2}\right)\right|^{p}\dfrac{(1-|w|^2)^{p(k-\gamma)}}{|1-\langle z,w \rangle|^{n+1+\alpha+p(k-\gamma)}}\mathrm{d}\nu_{\alpha}(z)\right)^{1/p} \lesssim \|x\|_{X}.
\end{eqnarray*}
According to the relation $\eqref{ref2},$ we obtain the following estimation of $I_{2}$
$$\displaystyle |I_{2}| \leq \|h_b\|\|\varphi\|_{p,\alpha,X}\|y^{\star}\|_{Y^{\star}} \lesssim \|h_b\|\|x\|_{X}\|y^{\star}\|_{Y^{\star}}.$$
Since $I_{1} = \langle h_{b}f,g \rangle_{\alpha,Y} - I_{2},$ by the relation $\eqref{ref3}$ and the previous estimates on $I_{2},$ we have that
$$|I_{1}| \leq |\langle h_{b}f,g \rangle_{\alpha,Y}| + |I_{2}| \lesssim \|h_b\|\|x\|_{X}\|y^{\star}\|_{Y^{\star}}.$$ Since $x \in X,$ $y^{\star} \in Y^{\star}$ are arbitrary and $\|x\|_{X} = \|\overline{x}\|_{\overline{X}},$ we get that
$$\displaystyle |I_{1}| = (1-|w|^2)^{k-\gamma}\log\left( \dfrac{1}{1-|w|^2} \right)|\langle L_{k}(b(w)(\overline{x})),y^{\star} \rangle_{Y,Y^{\star}}| \leq C \|h_b\|\|\overline{x}\|_{\overline{X}}\|y^{\star}\|_{Y^{\star}}.$$ Since $\overline{x} \in \overline{X}$ and $y^{\star} \in Y^{\star}$ are arbitrary, we deduce that :
\begin{eqnarray*}
\|L_{k}b(w)\|_{\mathcal{L}(\overline{X},Y)} & = & \displaystyle \sup_{\|\overline{x}\|_{\overline{X}} = 1, \|y^{\star}\|_{Y^{\star}} = 1} |\langle L_{k}(b(w)(\overline{x})),y^{\star} \rangle_{Y,Y^{\star}}|\\
& \leq & \dfrac{C}{(1-|w|^2)^{k-\gamma}}\left(\log \dfrac{1}{1-|w|^2}\right)^{-1}.
\end{eqnarray*}
The desired result follows at once using Corollary $\ref{lipsch1}.$
\end{proof}

\section{\textbf{Compactness of the little Hankel operator, $h_b$, with operator-valued symbols $b$ from $A^{p}_{\alpha}(\mathbb{B}_{n},X)$ to $A^{q}_{\alpha}(\mathbb{B}_{n},Y),$ with $1 < p \leq q < \infty$}}
In this section, we are going to characterize those symbols $b$ for whch the little Hankel operator extends into a bounded compact oparator from $A^{p}_{\alpha}(\mathbb{B}_{n},X)$ to $A^{q}_{\alpha}(\mathbb{B}_{n},Y),$ where $1 < p \leq q < \infty$ and $X,Y$ are two reflexive complex Banach spaces.
%Before giving the proof of Theorem $\ref{Compactp},$ it will be important to introduce some preliminary notions.

\subsection{ \textbf{Preliminaries notions}}

The proof of the following remark can be found in \cite[Proposition $1.6.1$]{Roc}

\begin{rem} Let $t \geq 0.$ Then the operator $R^{\alpha,t}$ is the unique continuous linear operator on $\mathcal{H}(\mathbb{B}_{n},X)$ satisfying
$$\displaystyle R^{\alpha,t} \left( \dfrac{x}{(1-\langle z,w \rangle)^{n+1+\alpha}}\right) = \dfrac{x}{(1-\langle z,w \rangle)^{n+1+\alpha+t}},$$
for every $z \in \mathbb{B}_{n}$ and $x \in X.$
\end{rem}

We will use the operator $R^{\alpha,t},$ for $t > 0,$ in the vector-valued Bergman space $A^1_{\alpha}(\mathbb{B}_{n},X)$ as follows:

\begin{prop}\label{check2} Let $t > 0$ and $f \in A^1_{\alpha}(\mathbb{B}_{n},X).$ Then 
$$\displaystyle R^{\alpha,t}f(z) = \int_{\mathbb{B}_n} \dfrac{f(w)}{(1-\langle z,w \rangle)^{n+1+\alpha+t}}\mathrm{d}\nu_{\alpha}(w),$$
for each $z \in \mathbb{B}_{n}.$ 
\end{prop}

The proof of the following proposition is not quite different to the proof in \cite[Proposition $1.15$]{Zhu_functions}, but for sake of the completeness, we will recall the proof.

\begin{prop}\label{partialdop} Suppose $N$ is a positive integer and $\alpha$ is a real such that $n+\alpha$ is not a negative integer. Then $R^{\alpha,N}$ as an operator acting on $\mathcal{H}(\mathbb{B}_{n},X)$ is a linear partial differential operator of order $N$ with polynomial coefficients, that is 
$$\displaystyle R^{\alpha,N}f(z) = \sum_{m \in \mathbb{N}^{n},|m| \leq N} p_{m}(z)\dfrac{\partial^{|m|} f}{\partial z^{m}}(z),$$ where each $p_{m}$ is a polynomial.
\end{prop}
\begin{proof}
Let $x \in X$ and $w \in \mathbb{B}_n.$ By using the multi-nomial formula $$\langle z,w \rangle^{k} = \displaystyle \sum_{|m| = k} \dfrac{k!}{m!}z^{m}\overline{w}^{m},$$ it follows that 
\begin{eqnarray*}
\dfrac{x}{(1 - \langle z,w \rangle)^{n+1+\alpha+N}} & = & \dfrac{x(1 -\langle z,w \rangle + \langle z,w \rangle)^{N}}{(1 - \langle z,w \rangle)^{n+1+\alpha+N}}\\
& = & \displaystyle \sum_{k=0}^{N}\dfrac{N!}{k!(N-k)!}\dfrac{\langle z,w \rangle^{k}x\; (1 - \langle z,w \rangle)^{N-k}}{(1 - \langle z,w \rangle)^{n+1+\alpha+N}}\\
& = & \displaystyle \sum_{k=0}^{N}\dfrac{N!}{k!(N-k)!}\sum_{|m| = k} \dfrac{k!}{m!}z^{m}\dfrac{\overline{w}^{m}x}{(1 - \langle z,w \rangle)^{n+1+\alpha+k}}\\
& = & \displaystyle \sum_{k=0}^{N}\sum_{|m| = k} \dfrac{N!}{m!(N-k)!}z^{m}\dfrac{\overline{w}^{m}x}{(1 - \langle z,w \rangle)^{n+1+\alpha+k}}\\
& = & \displaystyle \sum_{k=0}^{N}\sum_{|m| = k} \dfrac{N!}{\prod_{j=0}^{k}(n+1+\alpha+j)m!(N-k)!}z^{m}\dfrac{\partial^{k} }{\partial z^{m}}\left( \dfrac{x}{(1 - \langle z,w \rangle)^{n+1+\alpha}}\right).
\end{eqnarray*}
Therefore, there exists a constant $c_{mk}$ such that
$$R^{\alpha,N}\left(\dfrac{x}{(1 - \langle z,w \rangle)^{n+1+\alpha}}\right)  = \displaystyle \sum_{k=0}^{N}\sum_{|m| = k}c_{mk}z^{m}\dfrac{\partial^{k} }{\partial z^{m}}\left( \dfrac{x}{(1 - \langle z,w \rangle)^{n+1+\alpha}}\right).$$ Thus $$ R^{\alpha,N} = \displaystyle \sum_{k=0}^{N}\sum_{|m| = k}c_{mk}z^{m}\dfrac{\partial^{k} }{\partial z^{m}}.$$
\end{proof}

We will also need the following results whose proofs can be found in \cite{Roc}.

\begin{lem}\label{check1} Let $t > 0.$ Then
$$\displaystyle \int_{\mathbb{B}_n} f(z)\overline{g(z)}\mathrm{d}\nu_{\alpha}(z) = \int_{\mathbb{B}_n} R^{\alpha,t}f(z)\overline{g(z)} \mathrm{d}\nu_{\alpha+t}(z),$$
for all $f \in A^1_{\alpha}(\mathbb{B}_{n},X)$ and $g \in H^{\infty}(\mathbb{B}_{n},\mathbb{C}).$
\end{lem}

\begin{lem} Let $t > 0$ and $X$ a complex Banach space. Then 
\begin{eqnarray*}
\displaystyle  \int_{\mathbb{B}_n} \langle f(z),g(z) \rangle_{X,X^{\star}}\mathrm{d}\nu_{\alpha}(z) & = & \displaystyle \int_{\mathbb{B}_n} \langle R^{\alpha,t}f(z),g(z) \rangle_{X,X^{\star}}\mathrm{d}\nu_{\alpha+t}(z)\\
& = & \displaystyle  \int_{\mathbb{B}_n} \langle f(z),R^{\alpha,t}g(z) \rangle_{X,X^{\star}}\mathrm{d}\nu_{\alpha+t}(z),
\end{eqnarray*}
for every $f \in A^{1}_{\alpha}(\mathbb{B}_{n},X)$ and $g \in H^{\infty}(\mathbb{B}_{n},X^{\star}).$
\end{lem}

\begin{cor}\label{brett1} Suppose $t >0$ and $ 1 < p < \infty.$ If $b \in A^{p'}_{\alpha}(\mathbb{B}_{n},\mathcal{L}(\overline{X},Y)),$ where $p'$ is the conjugate exponent of $p,$ then the following equality holds
$$\displaystyle  \int_{\mathbb{B}_n} \langle b(z)\overline{f(z)},g(z) \rangle_{Y,Y^{\star}}\mathrm{d}\nu_{\alpha}(z) = \int_{\mathbb{B}_n} \langle R^{\alpha,t}b(z)\overline{f(z)},g(z) \rangle_{Y,Y^{\star}}\mathrm{d}\nu_{\alpha+t}(z)$$  for $f \in H^{\infty}(\mathbb{B}_{n},X)$ and $g \in H^{\infty}(\mathbb{B}_{n},Y^{\star}).$
\end{cor}

In the sequel, we will need to interchange the position of the summation symbol and the integral symbol in a particular situation. That is why we introduce this lemma.

\begin{lem}\label{justF} Assume $1 < t < \infty.$ Let $b(z) = \sum_{\beta \in \mathbb{N}^{n}}\hat{b}(\beta)z^{\beta} \in A^t_{\alpha}(\mathbb{B}_{n},\mathcal{L}(\overline{X},Y)).$
Then 
$$\int_{\mathbb{B}_n} \langle b(z)\left(\overline{f(z)}\right),y^{\star}_{0} \rangle_{Y,Y^{\star}}\mathrm{d}\nu_{\alpha}(z) = \sum_{\beta \in \mathbb{N}^n}\int_{\mathbb{B}_n}z^{\beta} \langle \hat{b}(\beta) \left(\overline{f(z)}\right),y^{\star}_{0} \rangle_{Y,Y^{\star}}\mathrm{d}\nu_{\alpha}(z),$$ 
for every $f \in H^{\infty}(\mathbb{B}_{n},X)$ and $y^{\star}_{0} \in Y^{\star}$ with $\|y^{\star}_{0}\|_{Y^{\star}} = 1.$
\end{lem}

\begin{proof}
Since  $b(z) = \sum_{\beta \in \mathbb{N}^{n}}\hat{b}(\beta)z^{\beta} \in A^{t}_{\alpha}(\mathbb{B}_{n},\mathcal{L}(\overline{X},Y)),$ we have that 
$$\lim_{N \rightarrow \infty} \int_{\mathbb{B}_n} \left\|  b(z) - \sum_{\beta \in \mathbb{N}^{n}, |\beta| \leq N }\hat{b}(\beta)z^{\beta}  \right\|^t_{\mathcal{L}(\overline{X},Y)}\mathrm{d}\nu_{\alpha}(z) = 0.$$
We have
\begin{eqnarray*}
\displaystyle \left|\int_{\mathbb{B}_n} \left\langle \left( b(z) - \sum_{\beta \in \mathbb{N}^{n}:|\beta| \leq N}\hat{b}(\beta)z^{\beta}\right) (\overline{f(z)}),y^{\star}_{0} \right\rangle_{Y,Y^{\star}} \mathrm{d}\nu_{\alpha}(z)\right|  \leq \\ \displaystyle \int_{\mathbb{B}_n} \left\| b(z) - \sum_{\beta \in \mathbb{N}^{n}:|\beta| \leq N}\hat{b}(\beta)z^{\beta} \right\|_{\mathcal{L}(\overline{X},Y)} \|\overline{f(z)}\|_{\overline{X}} \|y^{\star}_{0}\|_{Y^{\star}} \mathrm{d}\nu_{\alpha}(z) = \\\displaystyle
\int_{\mathbb{B}_n} \left\| b(z) - \sum_{\beta \in \mathbb{N}^{n}:|\beta| \leq N}\hat{b}(\beta)z^{\beta} \right\|_{\mathcal{L}(\overline{X},Y)} \|f(z)\|_{X} \mathrm{d}\nu_{\alpha}(z) \lesssim \\\displaystyle
\int_{\mathbb{B}_n} \left\| b(z) - \sum_{\beta \in \mathbb{N}^{n}:|\beta| \leq N}\hat{b}(\beta)z^{\beta} \right\|^t_{\mathcal{L}(\overline{X},Y)} \mathrm{d}\nu_{\alpha}(z) \longrightarrow 0
\end{eqnarray*}
as $N \rightarrow \infty.$ Therefore, we have that
\begin{eqnarray*}
\int_{\mathbb{B}_n} \langle b(z)\left(\overline{f(z)}\right),y^{\star}_{0} \rangle_{Y,Y^{\star}}\mathrm{d}\nu_{\alpha}(z) & = & \displaystyle \lim_{N \rightarrow \infty}\int_{\mathbb{B}_n} \left\langle \sum_{\beta \in \mathbb{N}^{n}:|\beta| \leq N} \hat{b}(\beta)z^{\beta} \left(\overline{f(z)}\right),y^{\star}_{0} \right\rangle_{Y,Y^{\star}}\mathrm{d}\nu_{\alpha}(z)\\
& = & \displaystyle \lim_{N \rightarrow \infty}\int_{\mathbb{B}_n} \sum_{\beta \in \mathbb{N}^{n}:|\beta| \leq N} \left\langle  \hat{b}(\beta)z^{\beta} \left(\overline{f(z)}\right),y^{\star}_{0} \right\rangle_{Y,Y^{\star}}\mathrm{d}\nu_{\alpha}(z)\\
& = & \displaystyle \lim_{N \rightarrow \infty}\sum_{\beta \in \mathbb{N}^{n}:|\beta| \leq N} \int_{\mathbb{B}_n} \left\langle  \hat{b}(\beta)z^{\beta} \left(\overline{f(z)}\right),y^{\star}_{0} \right\rangle_{Y,Y^{\star}}\mathrm{d}\nu_{\alpha}(z)\\
& = & \sum_{\beta \in \mathbb{N}^{n}} \int_{\mathbb{B}_n} \left\langle  \hat{b}(\beta)z^{\beta} \left(\overline{f(z)}\right),y^{\star}_{0} \right\rangle_{Y,Y^{\star}}\mathrm{d}\nu_{\alpha}(z).
\end{eqnarray*}
\end{proof}

In the following lemma, we compute the little Hankel operator when the operator-valued symbol is a monomial.

\begin{lem}\label{bosss2} Suppose $1 < p < \infty$ and $\gamma \in \mathbb{N}^{n}.$ If $a_{\gamma} \in \mathcal{L}(\overline{X},Y),$ then for every $\displaystyle f(z) = \sum_{\beta \in \mathbb{N}^{n}}c_{\beta}z^{\beta} \in A^{p}_{\alpha}(\mathbb{B}_{n},X),$ we have $$\displaystyle h_{a_{\gamma}z^{\gamma}}f(z) = \sum_{\beta \in \mathbb{N}^{n},\beta \leq \gamma} a_{\gamma}(\overline{c_{\beta}})\dfrac{\gamma!\Gamma(n+1+\alpha+|\gamma-\beta|)}{(\gamma-\beta)!\Gamma(n+1+\alpha+|\gamma|)}z^{\gamma-\beta}.$$
\end{lem}
\begin{proof}
Since $$\displaystyle f(z) = \sum_{\beta \in \mathbb{N}^{n}}c_{\beta}z^{\beta} \in A^{p}_{\alpha}(\mathbb{B}_{n},X),$$ and $p>1$ by using \cite[Corollary 4]{Zhu_paper1}, it follows that 
\begin{equation}
\displaystyle \int_{\mathbb{B}_n} \left\| \sum_{|\beta| \geq N+1}c_{\beta}z^{\beta} \right\|^p_{X}\mathrm{d}\nu_{\alpha}(z) \rightarrow 0~~~~\mbox{as}~~ N \rightarrow \infty. \label{justf1}
\end{equation}
Firstly, let us prove that 
\begin{equation}
\displaystyle{ \int_{\mathbb{B}_n}\dfrac{\sum_{\beta \in \mathbb{N}^{n}}a_{\gamma}(\overline{c_{\beta}})\overline{w^{\beta}}}{(1 - \langle z,w \rangle)^{n+1+\alpha}}\mathrm{d}\nu_{\alpha}(w) =  \sum_{\beta \in \mathbb{N}^{n}} \int_{\mathbb{B}_n}\dfrac{a_{\gamma}(\overline{c_{\beta}})\overline{w^{\beta}}}{(1 - \langle z,w \rangle)^{n+1+\alpha}}\mathrm{d}\nu_{\alpha}(w)}\label{justf2}
\end{equation}
Let $N \in \mathbb{N}.$ We have that
$$\displaystyle{ \int_{\mathbb{B}_n}\left\|\dfrac{\sum_{\beta \in \mathbb{N}^{n}}a_{\gamma}(\overline{c_{\beta}})\overline{w^{\beta}} - \sum_{|\beta| \leq N}a_{\gamma}(\overline{c_{\beta}})\overline{w^{\beta}}}{(1 - \langle z,w \rangle)^{n+1+\alpha}}\right\|_{Y}\mathrm{d}\nu_{\alpha}(w)}$$ 
\begin{eqnarray*}
& = & \displaystyle{ \int_{\mathbb{B}_n}\left\|\dfrac{\sum_{|\beta| \geq N+1}a_{\gamma}(\overline{c_{\beta}})\overline{w^{\beta}}}{(1 - \langle z,w \rangle)^{n+1+\alpha}} \right\|_{Y}\mathrm{d}\nu_{\alpha}(w)}\\
& = & \displaystyle{ \int_{\mathbb{B}_n}\left\|\dfrac{a_{\gamma}\left( \sum_{|\beta| \geq N+1}(\overline{c_{\beta}})\overline{w^{\beta}}\right)}{(1 - \langle z,w \rangle)^{n+1+\alpha}}\right\|_{Y}\mathrm{d}\nu_{\alpha}(w)}\\
& \leq & \displaystyle{ \dfrac{\|a_{\gamma}\|_{\mathcal{L}(\overline{X},Y)}}{(1 - |z|)^{n+1+\alpha}}\int_{\mathbb{B}_n}\left\|\sum_{|\beta| \geq N+1}c_{\beta}w^{\beta}\right\|_{X}\mathrm{d}\nu_{\alpha}(w)}\\
& \leq & \displaystyle{ \dfrac{\|a_{\gamma}\|_{\mathcal{L}(\overline{X},Y)}}{(1 - |z|)^{n+1+\alpha}}\int_{\mathbb{B}_n}\left\|\sum_{|\beta| \geq N+1}c_{\beta}w^{\beta}\right\|_{X}\mathrm{d}\nu_{\alpha}(w)}\\
& \leq & \displaystyle{ \dfrac{\|a_{\gamma}\|_{\mathcal{L}(\overline{X},Y)}}{(1 - |z|)^{n+1+\alpha}}\left( \int_{\mathbb{B}_n}\left\|\sum_{|\beta| \geq N+1}c_{\beta}w^{\beta}\right\|^p_{X}\mathrm{d}\nu_{\alpha}(w)\right)^{1/p}}.
\end{eqnarray*}
Therefore
$$\displaystyle{ \left\|\int_{\mathbb{B}_n}\dfrac{\sum_{\beta \in \mathbb{N}^{n}}a_{\gamma}(\overline{c_{\beta}})\overline{w^{\beta}} - \sum_{|\beta| \leq N}a_{\gamma}(\overline{c_{\beta}})\overline{w^{\beta}}}{(1 - \langle z,w \rangle)^{n+1+\alpha}}\mathrm{d}\nu_{\alpha}(w) \right\|_{Y}}$$
is less than or equal to
\begin{equation}
\displaystyle{ \dfrac{\|a_{\gamma}\|_{\mathcal{L}(\overline{X},Y)}}{(1 - |z|)^{n+1+\alpha}}\left( \int_{\mathbb{B}_n}\left\|\sum_{|\beta| \geq N+1}c_{\beta}w^{\beta}\right\|^p_{X}\mathrm{d}\nu_{\alpha}(w)\right)^{1/p}.}\label{justice1}
\end{equation}
By using $\eqref{justf1}$ and $\eqref{justice1},$ it follows that 
$$\displaystyle \left\|\int_{\mathbb{B}_n}\dfrac{\displaystyle \sum_{\beta \in \mathbb{N}^{n}}a_{\gamma}(\overline{c_{\beta}})\overline{w^{\beta}} - \sum_{|\beta| \leq N}a_{\gamma}(\overline{c_{\beta}})\overline{w^{\beta}}}{(1 - \langle z,w \rangle)^{n+1+\alpha}}\mathrm{d}\nu_{\alpha}(w) \right\|_{Y} \rightarrow 0$$ as $N \rightarrow \infty,$ and so
\begin{eqnarray*}
\displaystyle \int_{\mathbb{B}_n}\dfrac{\displaystyle \sum_{\beta \in \mathbb{N}^{n}}a_{\gamma}(\overline{c_{\beta}})\overline{w^{\beta}}}{(1 - \langle z,w \rangle)^{n+1+\alpha}}\mathrm{d}\nu_{\alpha}(w) & = & \displaystyle \lim_{N\rightarrow\infty}\int_{\mathbb{B}_n}\dfrac{\displaystyle \sum_{|\beta| \leq N}a_{\gamma}(\overline{c_{\beta}})\overline{w^{\beta}}}{(1 - \langle z,w \rangle)^{n+1+\alpha}}\mathrm{d}\nu_{\alpha}(w) \\
& = & \displaystyle \lim_{N\rightarrow\infty} \sum_{|\beta| \leq N}\int_{\mathbb{B}_n}\dfrac{a_{\gamma}(\overline{c_{\beta}})\overline{w^{\beta}}}{(1 - \langle z,w \rangle)^{n+1+\alpha}}\mathrm{d}\nu_{\alpha}(w)\\
& = & \displaystyle \sum_{\beta \in \mathbb{N}^{n}}\int_{\mathbb{B}_n}\dfrac{a_{\gamma}(\overline{c_{\beta}})\overline{w^{\beta}}}{(1 - \langle z,w \rangle)^{n+1+\alpha}}\mathrm{d}\nu_{\alpha}(w),
\end{eqnarray*}
which is the desired result. Secondly, let us prove that 

\begin{equation}\label{justf3}
%\displaystyle 
\int_{\mathbb{B}_{n}} \sum_{k=0}^{\infty}\dfrac{\Gamma(n+1+\alpha+k)}{\Gamma(n+1+\alpha)k!} \langle z,w \rangle^{k} \mathrm{d}\nu_{\alpha}(w) = \sum_{k=0}^{\infty} \int_{\mathbb{B}_{n}}\dfrac{\Gamma(n+1+\alpha+k)}{\Gamma(n+1+\alpha)k!} \langle z,w \rangle^{k} \mathrm{d}\nu_{\alpha}(w).
\end{equation}
Let $N \in \mathbb{N}.$ We have
\begin{eqnarray*}
\displaystyle \left|\sum_{k=0}^{N}\dfrac{\Gamma(n+1+\alpha+k)}{\Gamma(n+1+\alpha)k!} \langle z,w \rangle^{k}\right| & \leq & \sum_{k=0}^{N}\dfrac{\Gamma(n+1+\alpha+k)}{\Gamma(n+1+\alpha)k!}|z|^{k}\\
& \leq & \sum_{k=0}^{\infty}\dfrac{\Gamma(n+1+\alpha+k)}{\Gamma(n+1+\alpha)k!}|z|^{k}\\
& = & \dfrac{1}{(1 - |z|)^{n+1+\alpha}}.
\end{eqnarray*} 
Since $$ \int_{\mathbb{B}_{n}} \dfrac{1}{(1 - |z|)^{n+1+\alpha}}\mathrm{d}\nu_{\alpha}(w) = \dfrac{1}{(1 - |z|)^{n+1+\alpha}},$$ by the dominated convergence theorem, we have that 
\begin{eqnarray*}
\displaystyle \sum_{k=0}^{\infty} \int_{\mathbb{B}_{n}} \dfrac{\Gamma(n+1+\alpha+k)}{\Gamma(n+1+\alpha)k!} \langle z,w \rangle^{k} \mathrm{d}\nu_{\alpha}(w) & = & \displaystyle \lim_{N\rightarrow \infty}\sum_{k=0}^{N}\int_{\mathbb{B}_{n}}\dfrac{\Gamma(n+1+\alpha+k)}{\Gamma(n+1+\alpha)k!} \langle z,w \rangle^{k} \mathrm{d}\nu_{\alpha}(w)\\
& = & \displaystyle \lim_{N\rightarrow\infty}\int_{\mathbb{B}_{n}} \sum_{k=0}^{N}\dfrac{\Gamma(n+1+\alpha+k)}{\Gamma(n+1+\alpha)k!} \langle z,w \rangle^{k} \mathrm{d}\nu_{\alpha}(w)\\
& = & \displaystyle \int_{\mathbb{B}_{n}}\lim_{N\rightarrow\infty}\sum_{k=0}^{N}\dfrac{\Gamma(n+1+\alpha+k)}{\Gamma(n+1+\alpha)k!} \langle z,w \rangle^{k} \mathrm{d}\nu_{\alpha}(w)\\
& = & \displaystyle \int_{\mathbb{B}_{n}} \sum_{k=0}^{\infty}\dfrac{\Gamma(n+1+\alpha+k)}{\Gamma(n+1+\alpha)k!} \langle z,w \rangle^{k} \mathrm{d}\nu_{\alpha}(w).
\end{eqnarray*}
We are now ready to prove our lemma. For $\displaystyle f(z) = \sum_{\beta \in \mathbb{N}^{n}}c_{\beta}z^{\beta} \in A^{p}_{\alpha}(\mathbb{B}_{n},X),$ by using the following multi-nomial formula (\cite[(1.1)]{Zhu_functions})
$$\langle z,w \rangle^k = \displaystyle \sum_{|m| = k}\frac{k!}{m!}z^{m}\overline{w^m}$$ and the following formula (\cite[(1.23)]{Zhu_functions})
$$\displaystyle \int_{\mathbb{B}_n}|z^m|^2\mathbb{d}\nu_{\alpha}(z) = \dfrac{m!\Gamma(n+\alpha+1)}{\Gamma(n+|m|+\alpha+1)},$$
we get that
\begin{eqnarray*}
h_{a_{\gamma}z^{\gamma}}f(z) & = & \displaystyle \int_{\mathbb{B}_n} \dfrac{ a_{\gamma}w^{\gamma}\left(\overline{ \displaystyle \sum_{\beta \in \mathbb{N}^{n}} c_{\beta}w^{\beta}}\right)}{(1-\langle z,w \rangle)^{n+1+\alpha}} \mathrm{d}\nu_{\alpha}(w)\\
\\
& = & \displaystyle \int_{\mathbb{B}_n}\dfrac{w^{\gamma}\displaystyle \sum_{\beta \in \mathbb{N}^{n}}a_{\gamma}(\overline{c_{\beta}})\overline{w^{\beta}}}{(1 - \langle z,w \rangle)^{n+1+\alpha}}\mathrm{d}\nu_{\alpha}(w)\\
\\
& = & \displaystyle \sum_{\beta \in \mathbb{N}^{n}} \int_{\mathbb{B}_n}\dfrac{w^{\gamma}a_{\gamma}(\overline{c_{\beta}})\overline{w^{\beta}}}{(1 - \langle z,w \rangle)^{n+1+\alpha}}\mathrm{d}\nu_{\alpha}(w)~~(\mbox{by}~~~~~~~~~~~~ \eqref{justf2})\\
\\
& = & \displaystyle \sum_{\beta \in \mathbb{N}^{n}}a_{\gamma}(\overline{c_{\beta}}) \int_{\mathbb{B}_{n}}w^{\gamma}\overline{w^{\beta}} \sum_{k=0}^{\infty}\dfrac{\Gamma(n+1+\alpha+k)}{\Gamma(n+1+\alpha)k!} \langle z,w \rangle^{k} \mathrm{d}\nu_{\alpha}(w)\\
\\
& = &  \displaystyle \sum_{\beta \in \mathbb{N}^{n}}a_{\gamma}(\overline{c_{\beta}}) \sum_{k=0}^{\infty}\int_{\mathbb{B}_{n}}w^{\gamma}\overline{w^{\beta}}\dfrac{\Gamma(n+1+\alpha+k)}{\Gamma(n+1+\alpha)k!} \langle z,w \rangle^{k} \mathrm{d}\nu_{\alpha}(w)~~~~~~~(\mbox{by}~~\eqref{justf3})\\
\\
& = & \displaystyle \sum_{\beta \in \mathbb{N}^{n}}a_{\gamma}(\overline{c_{\beta}}) \sum_{k=0}^{\infty}\dfrac{\Gamma(n+1+\alpha+k)}{\Gamma(n+1+\alpha)k!} \int_{\mathbb{B}_{n}}w^{\gamma}\overline{w^{\beta}} \sum_{|m|=k} \dfrac{k!}{m!}z^{m}\overline{w^{m}} \mathrm{d}\nu_{\alpha}(w)\\
\\
& = & \displaystyle \sum_{\beta \in \mathbb{N}^{n}}a_{\gamma}(\overline{c_{\beta}})\sum_{k=0}^{\infty}\dfrac{\Gamma(n+1+\alpha+k)}{\Gamma(n+1+\alpha)k!} \sum_{|m|=k} \dfrac{k!}{m!} \int_{\mathbb{B}_{n}}w^{\gamma}\overline{w^{\beta}}z^{m}\overline{w^{m}} \mathrm{d}\nu_{\alpha}(w)\\
\\
& = & \displaystyle \sum_{\beta \in \mathbb{N}^{n}}a_{\gamma}(\overline{c_{\beta}}) \sum_{k=0}^{\infty}\sum_{|m|=k}\dfrac{\Gamma(n+1+\alpha+k)}{\Gamma(n+1+\alpha)m!}\int_{\mathbb{B}_{n}}w^{\gamma} z^{m}\overline{w^{m+\beta}} \mathrm{d}\nu_{\alpha}(w)\\
\\
& = & \displaystyle \sum_{\beta \in \mathbb{N}^{n}}a_{\gamma}(\overline{c_{\beta}}) \sum_{m \in \mathbb{N}^{n}}\dfrac{\Gamma(n+1+\alpha+|m|)}{\Gamma(n+1+\alpha)m!}z^{m} \int_{\mathbb{B}_{n}}w^{\gamma}\overline{w^{\beta+m}}  \mathrm{d}\nu_{\alpha}(w)\\
& = & \displaystyle \sum_{\beta \in \mathbb{N}^{n},\beta \leq \gamma}a_{\gamma}(\overline{c_{\beta}}) \dfrac{\Gamma(n+1+\alpha+|\gamma-\beta|)}{\Gamma(n+1+\alpha)(\gamma-\beta)!}z^{\gamma-\beta} \int_{\mathbb{B}_n}|z^{\gamma}|^{2}\mathrm{d}\nu_{\alpha}(w)\\
\\
& = & \displaystyle \sum_{\beta \in \mathbb{N}^{n},\beta \leq \gamma}a_{\gamma}(\overline{c_{\beta}}) \dfrac{\Gamma(n+1+\alpha+|\gamma-\beta|)}{\Gamma(n+1+\alpha)(\gamma-\beta)!} \dfrac{\gamma!\Gamma(n+1+\alpha)}{\Gamma(n+1+\alpha+|\gamma|)}z^{\gamma-\beta}\\
\\
& = & \displaystyle \sum_{\beta \in \mathbb{N}^{n},\beta \leq \gamma}a_{\gamma}(\overline{c_{\beta}}) \dfrac{\gamma!\Gamma(n+1+\alpha+|\gamma-\beta|)}{(\gamma-\beta)!\Gamma(n+1+\alpha+|\gamma|)} z^{\gamma-\beta}.
\end{eqnarray*}
\end{proof}

The goal of the following lemma is to prove that the linear span of the vector-valued Bergman kernel $\dfrac{x^{\star}}{(1-\langle w,z \rangle)^{n+1+\alpha}},$ where $x^{\star} \in X^{\star}$ and $z,w \in \mathbb{B}_{n}$ form a dense subspace in the vector-valued Bergman space $A^{p'}_{\alpha}(\mathbb{B}_{n},X^{\star}),$ with $1 <p < \infty$ and $p'$ is the conjugate exponent of $p.$

\begin{lem}\label{compl2} Suppose that $1 < p < \infty.$ For each $x^{\star} \in X^{\star}$ and $z \in \mathbb{B}_{n},$ let
$$ e_{z,x^{\star}}(w) = \dfrac{x^{\star}}{(1-\langle w,z \rangle)^{n+1+\alpha}};\hspace{1cm}w \in \mathbb{B}_n.$$ Then $e_{z,x^{\star}} \in A^{p'}_{\alpha}(\mathbb{B}_{n},X^{\star})$ and the subspace generated by $e_{z,x^{\star}}$ is dense in $A^{p'}_{\alpha}(\mathbb{B}_{n},X^{\star}).$
\end{lem}

\begin{proof} Let $\phi \in A^p_{\alpha}(\mathbb{B}_{n},X)$ such that $\langle \phi,e_{z,x^{\star}} \rangle_{\alpha,X} = 0$ for all $z \in \mathbb{B}_{n}$ and $x^{\star} \in X^{\star}.$
Let $f^{\star} \in   A^{p'}_{\alpha}(\mathbb{B}_{n},X^{\star}).$ According to the Hahn-Banach theorem, it suffices to prove that $\langle  \phi,f^{\star} \rangle_{\alpha,X} = 0.$ For all $z \in \mathbb{B}_{n}$ and $x^{\star} \in X^{\star},$ using Lemma $\ref{premierlem}$ and the reproducing kernel formula, it follows that  
\begin{eqnarray*}
0 & = & \langle \phi,e_{z,x^{\star}} \rangle_{\alpha,X} \\
& = & \displaystyle \int_{\mathbb{B}_n} \langle \phi(w),e_{z,x^{\star}}(w) \rangle_{X,X^{\star}}\mathrm{d}\nu_{\alpha}(w)\\
& = & \displaystyle \int_{\mathbb{B}_n} \langle \phi(w),\dfrac{x^{\star}}{(1-\langle w,z \rangle)^{n+1+\alpha}} \rangle_{X,X^{\star}}\mathrm{d}\nu_{\alpha}(w)\\
& = & \displaystyle \int_{\mathbb{B}_n} \langle \dfrac{\phi(w)}{(1-\langle z,w \rangle)^{n+1+\alpha}},x^{\star}\rangle_{X,X^{\star}}\mathrm{d}\nu_{\alpha}(w)\\
& = & \langle \phi(z),x^{\star} \rangle_{X,X^{\star}}.
\end{eqnarray*}
Therefore, for all $x^{\star} \in X^{\star},$ we have 
$$\langle \phi(z),x^{\star} \rangle_{X,X^{\star}} = 0.$$
 Thus $\phi (z)= 0$ for every $z \in \mathbb{B}_{n}.$ It follows that for each  $f^{\star} \in   A^{p'}_{\alpha}(\mathbb{B}_{n},X^{\star}),$ we have that
$$\langle \phi,f^{\star} \rangle_{\alpha,X} = \int_{\mathbb{B}_n}\langle \phi(z),f^{\star}(z) \rangle_{X,X^{\star}}\mathrm{d}\nu_{\alpha}(z) = 0.$$
\end{proof}

In the  proof of the following lemma, we use the fact that when $X$ is a reflexive complex Banach space and $1 < p < \infty,$ the dual of the vector-valued Bergman space 
$A^{p'}_{\alpha}(\mathbb{B}_{n},X^{\star})$ can be identified with $A^{p}_{\alpha}(\mathbb{B}_{n},X),$ where $p'$ is the conjugate exponent of $p.$

\begin{lem}\label{bosss1} Suppose that $1 < p < \infty,$ and $X$ is a reflexive complex Banach space. Let $\lbrace f_{j} \rbrace \subset A^{p}_{\alpha}(\mathbb{B}_{n},X)$ such that 
$\displaystyle f_{j} \rightarrow 0$ weakly in $A^{p}_{\alpha}(\mathbb{B}_{n},X)$ as $j \rightarrow \infty.$ Then for each $ \beta \in \mathbb{N}^{n},$ we have that $\partial^{\beta}f_{j}(0) \rightarrow 0$ weakly in $X$ as $j \rightarrow \infty,$ where  $\partial^{\beta} = \frac{\partial^{|\beta|} }{\partial z^{\beta}}.$
\end{lem}
\begin{proof}
Since for each $j \in \mathbb{N},$ $f_{j} \in A^{p}_{\alpha}(\mathbb{B}_{n},X),$ using the reproducing kernel formula  we have that
$$ f_{j}(z) = \displaystyle \int_{\mathbb{B}_n}\dfrac{f_{j}(w)}{(1 - \langle z,w \rangle)^{n+1+\alpha}}\mathrm{d}\nu_{\alpha}(w),\hspace{1cm} z \in \mathbb{B}_{n}.$$ Differentiating both sides of the previous relation with respect to $z,$ we obtain
$$\displaystyle  \partial^{\beta} f_{j}(z) = C(n,\alpha,|\beta|)\int_{\mathbb{B}_n}\dfrac{f_{j}(w)\overline{w}^{\beta}}{(1 - \langle z,w \rangle)^{n+1+\alpha+|\beta|}}\mathrm{d}\nu_{\alpha}(w).$$ Therefore, we have
$$\displaystyle  \partial^{\beta} f_{j}(0) = C(n,\alpha,|\beta|)\int_{\mathbb{B}_n} f_{j}(w)\overline{w}^{\beta}\mathrm{d}\nu_{\alpha}(w).$$ Now, let $x^{\star} \in X^{\star}$ and let us show that $\langle \partial^{\beta} f_{j}(0),x^{\star} \rangle_{X,X^{\star}} \rightarrow 0$ as $j \rightarrow \infty.$ But we have that 
\begin{eqnarray*}
\langle \partial^{\beta} f_{j}(0),x^{\star} \rangle_{X,X^{\star}} & = & \displaystyle C(n,\alpha,|\beta|) \left\langle \int_{\mathbb{B}_n} f_{j}(w)\overline{w}^{\beta}\mathrm{d}\nu_{\alpha}(w),x^{\star} \right\rangle_{X,X^{\star}}\\
& = & \displaystyle \int_{\mathbb{B}_n} \langle f_{j}(w),x^{\star} w^{\beta} \rangle_{X,X^{\star}}\mathrm{d}\nu_{\alpha}(w)\\
& = & \langle f_{j},g \rangle_{\alpha,X} \rightarrow 0~~~~\mbox{as}~~~~ j\rightarrow \infty,
\end{eqnarray*}
with $g(z) = x^{\star}z^{\beta} \in A^{p'}_{\alpha}(\mathbb{B}_{n},X^{\star}).$ Thus, $\langle \partial^{\beta} f_{j}(0),x^{\star} \rangle_{X,X^{\star}} \rightarrow 0$ as $j \rightarrow \infty.$
\end{proof}

We recall that the symbol $b$ used in the following lemma satisfies $\eqref{hypo1}$ and $\eqref{hypimp}.$ 

\begin{lem}\label{derivcompact} Suppose that $X$ is a reflexive complex Banach space and $k$ is a nonnegative integer. If the holomorphic mapping $z \mapsto b(z)$ maps $\mathbb{B}_{n}$ into $\mathcal{K}(\overline{X},Y),$ then
the holomorphic mapping $z \mapsto R^{\alpha,k}b(z)$ also maps $\mathbb{B}_{n}$ into $\mathcal{K}(\overline{X},Y).$
\end{lem}

\begin{proof} Let $z \in \mathbb{B}_{n}.$
Let $\lbrace f_{j} \rbrace$ a sequence of elements of $X$ which converges weakly to $0$ in $X$ as $j$ tends to infinity. Let us prove that $ \lim_{j\rightarrow\infty} \|R^{\alpha,k}b(z)\overline{f_{j}}\|_{Y} = 0.$ We know that the sequence $\lbrace f_j \rbrace$ is strongly bounded in $X.$ Let $j \in \mathbb{N},$ by using $\eqref{hypo1}$ for $z = 0,$ we get that the function $z \mapsto b(z)\overline{f_{j}}\in A^{1}_{\alpha}(\mathbb{B}_{n},Y).$ By the reproducing kernel formula, it follows that
\begin{equation}
b(z)\overline{f_{j}} = \displaystyle \int_{\mathbb{B}_n}\dfrac{b(w)\overline{f_{j}}}{(1 - \langle z,w \rangle)^{n+1+\alpha}}\mathrm{d}\nu_{\alpha}(w).\label{opdiff}
\end{equation}
Applying the partial differential operator $R^{\alpha,k}$ to $\eqref{opdiff},$ we have 
$$R^{\alpha,k}b(z)\overline{f_{j}} = \displaystyle \int_{\mathbb{B}_n}\dfrac{b(w)\overline{f_{j}}}{(1 - \langle z,w \rangle)^{n+1+\alpha+k}}\mathrm{d}\nu_{\alpha}(w).$$ We also have
\begin{eqnarray*}
\displaystyle \dfrac{\|b(w)\overline{f_{j}}\|_{Y}}{|1 - \langle z,w \rangle|^{n+1+\alpha+k}} & \leq & \displaystyle \dfrac{\|b(w)\|_{\mathcal{L}(\overline{X},Y)}\|f_{j}\|_{X}}{(1 - |z|)^{n+1+\alpha+k}}\\
& \leq & \dfrac{C(n+1+\alpha)}{(1-|z|)^{n+1+\alpha+k}}\|b(w)\|_{\mathcal{L}(\overline{X},Y)},
\end{eqnarray*}
and $$ \displaystyle \int_{\mathbb{B}_n}\dfrac{C(n+1+\alpha)}{(1-|z|)^{n+1+\alpha+k}}\|b(w)\|_{\mathcal{L}(\overline{X},Y)}\mathrm{d}\nu_{\alpha}(w) < \infty.$$ Therefore, by applying the dominated convergence theorem, we have that
\begin{eqnarray*}
\limsup_{j\rightarrow\infty}\|R^{\alpha,k}b(z)\overline{f_{j}}\|_{Y} & \leq & \limsup_{j\rightarrow \infty} \displaystyle \int_{\mathbb{B}_n} \dfrac{\|b(w)\overline{f_{j}}\|_{Y}}{|1 - \langle z,w \rangle|^{n+1+\alpha+k}}\mathrm{d}\nu_{\alpha}(w)\\
& = &  \displaystyle \int_{\mathbb{B}_n}\dfrac{\lim_{j\rightarrow\infty}\|b(w)\overline{f_{j}}\|_{Y}}{|1 - \langle z,w \rangle|^{n+1+\alpha+k}}\mathrm{d}\nu_{\alpha}(w) = 0.
\end{eqnarray*}
Thu,s for each $z \in \mathbb{B}_{n}$ $$ \lim_{j\rightarrow\infty}\|R^{\alpha,k}b(z)\overline{f_{j}}\|_{Y} = 0.$$
\end{proof}

The following result will be also important  in the sequel.

\begin{lem}\label{unifbp} Suppose $\beta_{0} \in \mathbb{N}^{n},$ $\lbrace f_j \rbrace$ a sequence of elements of $X$ which converges weakly to $0$ as $j$ tends to infinity. For $z \in \mathbb{B}_{n},$ let $x_{j}(z) = z^{\beta_{0}}f_{j}.$ Then $\lbrace x_{j} \rbrace \subset A^p_{\alpha}(\mathbb{B}_{n},X)$ and $\lbrace x_{j} \rbrace$ converges weakly to $0$ in $A^p_{\alpha}(\mathbb{B}_{n},X).$
\end{lem}
\begin{proof}
Let $j \in \mathbb{N}.$ Since $f_j \rightarrow 0$ weakly in $X$ as $j \rightarrow \infty,$ it follows that $\lbrace f_j \rbrace$ is strongly bounded in $X$ (see \cite{Brezis}). Let $\beta_{0} \in \mathbb{N}^n$ and $x_{j}(z) = z^{\beta_{0}}f_{j}.$ It is clear that $\lbrace x_{j} \rbrace \subset A^{p}_{\alpha}(\mathbb{B}_{n},X).$ For every $g \in A^{p'}_{\alpha}(\mathbb{B}_{n},X^{\star}),$ we have
\begin{eqnarray*}
\langle x_{j},g \rangle_{\alpha,X} & = & \displaystyle \int_{\mathbb{B}_n} \langle x_{j}(z),g(z) \rangle_{X,X^{\star}} \mathrm{d}\nu_{\alpha}(z)\\
& = &  \displaystyle \int_{\mathbb{B}_n}  \langle z^{\beta_{0}}f_{j},g(z) \rangle_{X,X^{\star}}\mathrm{d}\nu_{\alpha}(z)\\
& = &  \displaystyle \int_{\mathbb{B}_n} z^{\beta_{0}} \langle f_{j},g(z) \rangle_{X,X^{\star}}\mathrm{d}\nu_{\alpha}(z).
\end{eqnarray*} 
Since
\begin{eqnarray*}
\displaystyle \left| z^{\beta_{0}} \langle f_{j},g(z) \rangle_{X,X^{\star}} \right| & \leq & \displaystyle  |z^{\beta_{0}} \langle f_{j},g(z) \rangle_{X,X^{\star}}|\\
& \leq &  \|f_{j}\|_{X}\|g(z)\|_{X^{\star}} \\
& \leq &  C\|g(z)\|_{X^{\star}},
\end{eqnarray*}
and $$ \displaystyle \int_{\mathbb{B}_n} \|g(z)\|_{X^{\star}}\mathrm{d}\nu_{\alpha}(z) \leq \left( \int_{\mathbb{B}_n} \|g(z)\|^{p'}_{X^{\star}}\mathrm{d}\nu_{\alpha}(z)\right)^{1/p'} < \infty.$$
By using the dominated convergence theorem and the assumption, it follows that
\begin{eqnarray*}
\displaystyle \limsup_{j\longrightarrow \infty}\langle x_{j},g \rangle_{\alpha,X} & = & \displaystyle \int_{\mathbb{B}_n} z^{\beta_{0}} \lim_{j \longrightarrow \infty} \langle f_{j},g(z) \rangle_{X,X^{\star}} \mathrm{d}\nu_{\alpha}(z) = 0.
\end{eqnarray*}
\end{proof}

\subsection{\textbf{Boundedness of the little Hankel operator with operator-valued symbol on vector-valued Bergman spaces}}

The principal result here is that, the little Hankel operator with operator-valued symbol $h_b$ is a bounded operator form $A^p_{\alpha}(\mathbb{B}_{n},X)$ to $A^q_{\alpha}(\mathbb{B}_{n},Y)$ with $1 < p \leq q < \infty$ if and only if the symbol $b$ belongs to the generalized vector-valued Lipschitz space $\Lambda_{\gamma_{0}}(\mathbb{B}_{n}, \mathcal{L}(\overline{X},Y)),$ where $$\gamma_{0} = (n+1+\alpha)\left( \frac{1}{p} - \frac{1}{q}\right).$$ The result obtained generalize the Oliver's result \cite[Theorem $ 4.2.2$]{Roc}.
In the following lemma, we first prove that the definition of the generalized vector-valued Lipschitz space $\Lambda_{\gamma}(\mathbb{B}_{n},X),$ with $\gamma \geq 0$ is independent of the integer $k$ used.

\begin{lem}\label{prelimn1} Let $f \in \mathcal{H}(\mathbb{B}_{n},X).$ The following conditions are equivalent:
\begin{enumerate}
\item[(a)] There exists a nonnegative integer $k > \gamma$ such that $$ \sup_{z \in \mathbb{B}_{n}}(1 - |z|^2)^{k-\gamma}\|R^{\alpha,k}f(z)\|_{X} < \infty.$$
\item[(b)] For every nonnegative integer $k > \gamma$ we have $$ \sup_{z \in \mathbb{B}_{n}}(1 - |z|^2)^{k-\gamma}\|R^{\alpha,k}f(z)\|_{X} < \infty.$$
\end{enumerate}
\end{lem}

\begin{proof}
It is clear that $(b) \Rightarrow (a).$ So to complete the proof, we will prove that $(a) \Rightarrow (b).$ Suppose that there exists an integer $k > \gamma$ such that $$c:= \sup_{z \in \mathbb{B}_{n}}(1 - |z|^2)^{k-\gamma}\|R^{\alpha,k}f(z)\|_{X} < \infty.$$ We want to prove that $$ \sup_{z \in \mathbb{B}_{n}}(1 - |z|^2)^{k+1-\gamma}\|R^{\alpha,k+1}f(z)\|_{X} < \infty.$$
Since $c < \infty,$ then $f \in A^{1}_{\alpha}(\mathbb{B}_{n},X).$ Indeed, by \cite[Theorem $3.1.2$]{Roc}, we have that
\begin{eqnarray*}
\|f\|_{1,\alpha,X} & \simeq & \displaystyle \int_{\mathbb{B}_n} (1-|z|^2)^{k}\|R^{\alpha,k}f(z)\|_{X}\mathrm{d}\nu_{\alpha}(z)\\
& = & \displaystyle \int_{\mathbb{B}_n} [(1-|z|^2)^{k-\gamma}\|R^{\alpha,k}f(z)\|_{X}](1-|z|^2)^{\gamma}\mathrm{d}\nu_{\alpha}(z)\\
& \lesssim & \displaystyle c \int_{\mathbb{B}_n} (1 - |z|^2)^{\alpha+\gamma} \mathrm{d}\nu(z)\\
& < & \infty.
\end{eqnarray*}
By using Proposition $\ref{check2},$ we have that
$$ \displaystyle R^{\alpha,k+1}f(z) = \int_{\mathbb{B}_n}\dfrac{f(w)}{(1 - \langle z,w \rangle)^{n+1+\alpha+k+1}}\mathrm{d}\nu_{\alpha}(w).$$ Applyng Lemma $\ref{check1},$ it follows that
$$ \displaystyle R^{\alpha,k+1}f(z) = \int_{\mathbb{B}_n}\dfrac{R^{\alpha,k}f(w)}{(1 - \langle z,w \rangle)^{n+1+\alpha+k+1}}\mathrm{d}\nu_{\alpha+k}(w).$$ Thus, 
\begin{eqnarray*}
\|R^{\alpha,k+1}f(z)\|_{X} & \lesssim & \displaystyle \int_{\mathbb{B}_n} \dfrac{[(1 - |w|^2)^{k-\gamma}\|R^{\alpha,k}f(w)\|_{X}](1 - |w|^2)^{\alpha+\gamma}}{|1 - \langle z,w \rangle|^{n+1+\alpha+\gamma+(k+1-\gamma)}}\mathrm{d}\nu(w)\\
& \lesssim & \dfrac{c}{(1 - |z|^2)^{k+1-\gamma}}.
\end{eqnarray*}
Therefore, we have that $$ \sup_{z \in \mathbb{B}_{n}}(1 - |z|^2)^{k+1-\gamma}\|R^{\alpha,k+1}f(z)\|_{X} \lesssim c < \infty.$$ Also, if $k$ is a nonnegative integer with $k > \gamma$ such that $$c':= \sup_{z \in \mathbb{B}_{n}}(1 - |z|^2)^{k+1-\gamma}\|R^{\alpha,k+1}f(z)\|_{X} < \infty,$$ then $$ \sup_{z \in \mathbb{B}_{n}}(1 - |z|^2)^{k-\gamma}\|R^{\alpha,k}f(z)\|_{X} < \infty.$$ 
%Since $c' < \infty,$ then $f \in A^{1}_{\alpha}(\mathbb{B}_{n},X).$ Indeed, by \cite[Theorem $3.1.2$]{Roc}, we have that
%\begin{eqnarray*}
%\|f\|_{1,\alpha,X} & \simeq & \displaystyle \int_{\mathbb{B}_n} (1-|z|^2)^{k+1}\|R^{\alpha,k+1}f(z)\|_{X}\mathrm{d}\nu_{\alpha}(z)\\
%& = & \displaystyle \int_{\mathbb{B}_n} [(1-|z|^2)^{k+1-\gamma}\|R^{\alpha,k+1}f(z)\|_{X}](1-|z|^2)^{\gamma}\mathrm{d}\nu_{\alpha}(z)\\
%& \lesssim & \displaystyle c'\int_{\mathbb{B}_n} (1 - |z|^2)^{\alpha+\gamma} \mathrm{d}\nu(z)\\
%& < & \infty.
%\end{eqnarray*}
Applying Proposition $\ref{check2}$ and Lemma \ref{check1} 
we have that
$$\displaystyle R^{\alpha,k} f(z) = \int_{\mathbb{B}_n} \dfrac{f(w)\mathrm{d}\nu_{\alpha}(w)}{(1 - \langle z,w \rangle)^{n+1+\alpha+k}}  = \int_{\mathbb{B}_n} \dfrac{R^{\alpha,k+1} f(w)\mathrm{d}\nu_{\alpha+k+1}(w)}{(1 - \langle z,w \rangle)^{n+1+\alpha+k}},$$
where $z \in \mathbb{B}_n.$ By using Theorem $\ref{estimint},$ it follows that
\begin{eqnarray*}
\| R^{\alpha,k}f(z)\|_{X} & \lesssim & \displaystyle \int_{\mathbb{B}_n}\dfrac{[(1 - |w|^2)^{k+1-\gamma}\|R^{\alpha,k+1}f(w)\|_{X}](1 - |w|^2)^{\alpha+\gamma}\mathrm{d}\nu(w)}{|1 - \langle z,w \rangle|^{n+1+\alpha+k}}\\
& = & \displaystyle c'\int_{\mathbb{B}_n} \dfrac{(1-|w|^2)^{\alpha+\gamma}\mathrm{d}\nu(w)}{|1 - \langle z,w \rangle|^{n+1+\alpha+\gamma+(k-\gamma)}}\\
& \lesssim &  \dfrac{c'}{(1 - |z|^2)^{k-\gamma}}.
\end{eqnarray*}
Since $z \in \mathbb{B}_n$ is arbitrary, we obtain that  $$ \sup_{z \in \mathbb{B}_{n}}(1 - |z|^2)^{k-\gamma}\|R^{\alpha,k}f(z)\|_{X} \lesssim c' < \infty.$$ The proof of the lemma is complete.
\end{proof}

\begin{prop}\label{closuresubsp} Let $\gamma \geq 0$ and $f \in \Lambda_{\gamma}(\mathbb{B}_{n},X).$ The following conditions are equivalent:
\begin{enumerate}
\item[$(i)$] $ f \in \Lambda_{\gamma,0}(\mathbb{B}_{n},X).$
\item[$(ii)$] $\lim_{s \rightarrow 1^{-}} \|f - f_s \|_{\Lambda_{\gamma}(\mathbb{B}_{n},X)} = 0,$ where $f_s$ is the dilation function defined for $z \in \mathbb{B}_{n}$ by $f_s(z):=f(sz).$
\item[$(iii)$] $f$ belongs to the closure of $\mathcal{P}(\mathbb{B}_{n},X),$ where $\mathcal{P}(\mathbb{B}_{n},X)$ is the space of vector-valued holomorphic polynomials.
\end{enumerate}
\end{prop}
\begin{proof} $(i) \Rightarrow (ii).$ Suppose that $\frac{1}{2} < r < s < 1,$ and let $f_{s}(z) = f(sz),~~ z\in \mathbb{B}_{n}.$ By the definition, we have:
\begin{eqnarray*}
\|f - f_{s} \|_{\Lambda_{\gamma}(\mathbb{B}_{n},X)} & = & \displaystyle \sup_{z \in \mathbb{B}_{n}}(1 - |z|^2)^{k-\gamma}\|R^{\alpha,k}(f - f_{s})(z)\|_{X}\\
& = &  \sup_{z \in \mathbb{B}_{n}}(1 - |z|^2)^{k-\gamma}\|(R^{\alpha,k}f)(z) - (R^{\alpha,k}f_{s})(z)\|_{X}\\
& = &  \sup_{z \in \mathbb{B}_{n}}(1 - |z|^2)^{k-\gamma}\|(R^{\alpha,k}f)(z) - (R^{\alpha,k}f)(sz)\|_{X}\\
& = & \sup_{z \in \mathbb{B}_{n}}(1 - |z|^2)^{k-\gamma}\|(R^{\alpha,k}f)(z) - \chi_{r}(z)(R^{\alpha,k}f)(z)\\
& + & \chi_{r}(z)(R^{\alpha,k}f)(z) -  (R^{\alpha,k}f)(sz)\|_{X}\\
& \leq & \sup_{z \in \mathbb{B}_{n}}(1 - |z|^2)^{k-\gamma}\|(R^{\alpha,k}f)(z) - \chi_{r}(z)(R^{\alpha,k}f)(z)\|_{X} \\
& + & \sup_{z \in \mathbb{B}_{n}}(1 - |z|^2)^{k-\gamma}\|\chi_{r}(z)(R^{\alpha,k}f)(z) - (R^{\alpha,k}f)(sz)\|_{X},
\end{eqnarray*}
where $\chi_{r}$ is the characteristic function of the set $\lbrace |z| \leq r \rbrace.$ We first have the following estimate:
$$\sup_{z \in \mathbb{B}_{n}}(1 - |z|^2)^{k-\gamma}\|(R^{\alpha,k}f)(z) - \chi_{r}(z)(R^{\alpha,k}f)(z)\|_{X} \leq $$
\begin{eqnarray*}
&  & \sup_{|z| \leq r}(1 - |z|^2)^{k-\gamma}\|(R^{\alpha,k}f)(z) - \chi_{r}(z)(R^{\alpha,k}f)(z)\|_{X}\\ & + & 
\sup_{r < |z| < 1}(1 - |z|^2)^{k-\gamma}\|(R^{\alpha,k}f)(z) - \chi_{r}(z)(R^{\alpha,k}f)(z)\|_{X}\\
& = & \sup_{r < |z| < 1}(1 - |z|^2)^{k-\gamma}\|(R^{\alpha,k}f)(z)\|_{X}\\
& \leq & \sup_{r^2 < |z| < 1}(1 - |z|^2)^{k-\gamma}\|(R^{\alpha,k}f)(z)\|_{X}.
\end{eqnarray*}
We secondly have the following estimate:
$$\sup_{z \in \mathbb{B}_{n}}(1 - |z|^2)^{k-\gamma}\| \chi_{r}(z)(R^{\alpha,k}f)(z) - (R^{\alpha,k}f)(sz)\|_{X} \leq $$
\begin{eqnarray*}
&  & \sup_{|z| \leq r}(1 - |z|^2)^{k-\gamma}\|\chi_{r}(z)(R^{\alpha,k}f)(z) - (R^{\alpha,k}f)(sz)\|_{X}\\ & + & 
\sup_{r < |z| < 1}(1 - |z|^2)^{k-\gamma}\|\chi_{r}(z)(R^{\alpha,k}f)(z) - (R^{\alpha,k}f)(sz)\|_{X}\\
& = & \sup_{|z| \leq r}(1 - |z|^2)^{k-\gamma}\|(R^{\alpha,k}f)(z) - (R^{\alpha,k}f)(sz)\|_{X}\\ & + & 
\sup_{r < |z| < 1}(1 - |z|^2)^{k-\gamma}\|(R^{\alpha,k}f)(sz)\|_{X}.
\end{eqnarray*}
Using the change of variables $w = sz,$ we then obtain
\begin{eqnarray*}
\sup_{r < |z| < 1}(1 - |z|^2)^{k-\gamma}\|(R^{\alpha,k}f)(sz)\|_{X} & = & \sup_{rs < |w| < s}\left( 1 - \frac{|w|^2}{s^2}\right)^{k-\gamma}\|(R^{\alpha,k}f)(w)\|_{X}\\
& = & \sup_{rs < |w| < s} \frac{1}{s^{2(k-\gamma)}}\left( s^{2} - |w|^2\right)^{k-\gamma}\|(R^{\alpha,k}f)(w)\|_{X}\\
& \leq & 2^{2(k-\gamma)}\sup_{r^2 < |w| < 1} (1 - |w|^2)^{k-\gamma}\|(R^{\alpha,k}f)(w)\|_{X}.
\end{eqnarray*}
It follows by using the assumption that 
\begin{eqnarray*}
\|f - f_{s} \|_{\Lambda_{\gamma}(\mathbb{B}_{n},X)} & \leq & C_{\gamma}\sup_{r^2 < |w| < 1} (1 - |w|^2)^{k-\gamma}\|(R^{\alpha,k}f)(w)\|_{X} \\ & + & 
\sup_{|z| \leq r}(1 - |z|^2)^{k-\gamma}\|(R^{\alpha,k}f)(z) - (R^{\alpha,k}f)(sz)\|_{X},
\end{eqnarray*}
with $C_{\gamma} = 1+2^{2(k-\gamma)}.$ Since $(R^{\alpha,k}f)(sz) \rightarrow (R^{\alpha,k}f)(z)$ in $X$ uniformly on the compact set $\lbrace |z| \leq r \rbrace$ as $s\rightarrow 1^{-},$ we have 
 $$ \lim_{s \rightarrow 1^{-}}\sup_{|z| \leq r}(1 - |z|^2)^{k-\gamma}\|(R^{\alpha,k}f)(z) - R^{\alpha,k}f(sz)\|_{X} = 0.$$ It follows that
 $$\lim_{s \rightarrow 1^{-}} \|f - f_{s} \|_{\Lambda_{\gamma}(\mathbb{B}_{n},X)} \leq C_{\gamma}\limsup_{|w|\rightarrow 1^{-}} (1 - |w|^2)^{k-\gamma}\|(R^{\alpha,k}f)(w)\|_{X} = 0.$$
 
 $(ii) \Rightarrow (iii).$ Given $\epsilon > 0,$ by the assumption, there exists $s_{0} \in (0, 1)$ such that 
\begin{equation}
\|f - f_{s_{0}}\|_{\Lambda_{\gamma}(\mathbb{B}_{n},X)} < \epsilon.\label{krantzz2}
\end{equation}
Further note that $f_{s_{0}} \in \mathcal{H}(\frac{1}{s_{0}}\mathbb{B}_{n},X)$ and $1 < \frac{2}{1+s_{0}} < \frac{1}{s_{0}}.$ From this, and by using Taylor's formula, it follows that for each $m \in \mathbb{N},$ there exists a $X-$valued polynomial $p_m$ such that
$$\lim_{m \rightarrow \infty} \sup_{z \in \frac{2}{1+s_{0}}\overline{\mathbb{B}}_{n} }\|f_{s_{0}}(z) - p_{m}(z)\|_{X} = 0.$$
Therefore, there exists $m_{0} \in \mathbb{N}$ such that 
\begin{equation}
\sup_{z \in \frac{2}{1+s_{0}}\overline{\mathbb{B}}_{n} }\|f_{s_{0}}(z) - p_{m}(z)\|_{X} < \epsilon,\label{krantz1}
\end{equation}
for $m \geq m_{0}.$ By the Cauchy's inequality, there exists a constant $c_{s_{0}} > 0$ such that for each $i = 1,\cdots,n$ we have
\begin{equation}
\sup_{z \in \overline{\mathbb{B}}_{n}} \left\|\dfrac{\partial f_{s_{0}}}{\partial z_{i}} - \dfrac{\partial p_{m}}{\partial z_{i}}\right\|_{X} \leq c_{s_{0}}\sup_{z \in \frac{2}{1+s_{0}}\overline{\mathbb{B}}_{n} }\|f_{s_{0}}(z) - P_{m}(z)\|_{X}.\label{krantz2}
\end{equation}
Suppose $k$ is a nonnegative integer with $k > \gamma.$ By using $\eqref{krantz2}$ and Theorem $\ref{partialdop},$ there is a constant $c = c(s_{0},n,\alpha, k)$ such that
\begin{equation}
\sup_{z \in \overline{\mathbb{B}}_{n}}\|(R^{\alpha,k}f_{s_{0}})(z) - (R^{\alpha,k}p_{m_{0}})(z)\|_{X} \leq c \sup_{z \in \frac{2}{1+s_{0}}\overline{\mathbb{B}}_{n} }\|f_{s_{0}}(z) - P_{m_{0}}(z)\|_{X}.\label{krantz3}
\end{equation} 
It follows by $\eqref{krantz3}$ and $\eqref{krantz1}$ that
\begin{eqnarray*}
\sup_{z \in \mathbb{B}_{n}}(1- |z|^2)^{k-\gamma}\|R^{\alpha,k}(f_{s_{0}} - p_{m_{0}})(z)\|_{X} & \leq & \sup_{z \in \mathbb{B}_{n}}\|(R^{\alpha,k}f_{s_{0}})(z) - (R^{\alpha,k}p_{m_{0}})(z)\|_{X}\\
& \leq & c \sup_{z \in \frac{2}{1+s_{0}}\overline{\mathbb{B}}_{n} }\|f_{s_{0}}(z) - p_{m_{0}}(z)\|_{X}\\
& < & c\epsilon.
\end{eqnarray*}
Thus
\begin{equation}
\|f_{s_{0}} - p_{m_{0}}\|_{\Lambda_{\gamma}(\mathbb{B}_{n},X)} < c\epsilon.\label{krantzz1}
\end{equation}
Using $\eqref{krantzz2}$ and $\eqref{krantzz1},$ it follows that
\begin{eqnarray*}
\|f - p_{m_{0}}\|_{\Lambda_{\gamma}(\mathbb{B}_{n},X)} & \leq & \|f - f_{s_{0}}\|_{\Lambda_{\gamma}(\mathbb{B}_{n},X)}  +  \|f_{s_{0}} - p_{m_{0}}\|_{\Lambda_{\gamma}(\mathbb{B}_{n},X)}\\
& < & \epsilon + c\epsilon = (1+c)\epsilon.
\end{eqnarray*}
$(iii) \Rightarrow (i).$ Let $f$ in the closure of the set of vector-valued polynomial $\mathcal{P}(\mathbb{B}_{n},X),$ in $\Lambda_{\gamma}(\mathbb{B}_{n},X).$ There exists a sequence of vector-valued polynomials $ \lbrace p_m \rbrace$ in $\mathcal{P}(\mathbb{B}_{n},X)$ such that 
\begin{equation}
\lim_{m \rightarrow \infty}\|f - p_{m}\|_{\Lambda_{\gamma}(\mathbb{B}_{n},X)} = 0.\label{krantz4}
\end{equation}
Let us prove that for each $k > \gamma,$ 
$$\lim_{|z| \rightarrow 1^{-}}(1 - |z|^2)^{k-\gamma}\|(R^{\alpha,k}f)(z)\|_{X} = 0.$$ 
 Let $k > \gamma.$ We have that 
\begin{eqnarray*}
\|(R^{\alpha,k}f)(z)\|_{X} & \leq & \|(R^{\alpha,k}f)(z) - (R^{\alpha,k}p_{m})(z)\|_{X} + \|(R^{\alpha,k}p_{m})(z)\|_{X}\\
& \leq & \|(R^{\alpha,k}f)(z) - (R^{\alpha,k}p_{m})(z)\|_{X} + \|R^{\alpha,k}p_{m}\|_{\infty,X},
\end{eqnarray*}
where $\|R^{\alpha,k}p_{m}\|_{\infty,X} = \max_{z \in \mathbb{B}_{n}}\|(R^{\alpha,k}p_{m})(z)\|_{X}.$      
It follows that for each $m \in \mathbb{N},$ we have
\begin{eqnarray*}
(1 - |z|^2)^{k-\gamma}\|(R^{\alpha,k}f)(z)\|_{X} 
& \leq & (1 - |z|^2)^{k-\gamma}\|(R^{\alpha,k}f)(z) - (R^{\alpha,k}p_{m})(z)\|_{X}\\
&&+ (1 - |z|^2)^{k-\gamma}\|R^{\alpha,k}p_{m}\|_{\infty,X} \\
& \leq & \|f - p_{m} \|_{\Lambda_{\gamma}(\mathbb{B}_{n},X)} + (1 - |z|^2)^{k-\gamma}\|R^{\alpha,k}p_{m}\|_{\infty,X}.
\end{eqnarray*}
Letting $|z| \rightarrow 1^{-},$ we obtain that
$$\limsup_{|z| \rightarrow 1^{-}}(1 - |z|^2)^{k-\gamma}\|R^{\alpha,k}f(z)\|_{X} \leq \|f - p_{m} \|_{\Lambda_{\gamma}(\mathbb{B}_{n},X)},$$ 
for each $m \in \mathbb{N}.$ Now, letting $m \rightarrow \infty$ on both sides of the previous inequality, it follows by $\eqref{krantz4}$ that 
$$\limsup_{|z| \rightarrow 1^{-}}(1 - |z|^2)^{k-\gamma}\|R^{\alpha,k}f(z)\|_{X} = 0.$$ The proof of the proposition is complete.
\end{proof}

\begin{rem}
One of the consequences of the previous result is that, given $ \gamma \geq 0,$ the generalized little vector-valued Lipschitz space $\Lambda_{\gamma,0}(\mathbb{B}_{n},X)$ is a closed subspace of the generalized vector-valued Lipschitz space $\Lambda_{\gamma}(\mathbb{B}_{n},X).$
\end{rem}

From now on, we choose $ \gamma_{0} = (n+1+\alpha)\left(\frac{1}{p} - \frac{1}{q}\right),$ with $1 < p \leq q < \infty,$  
and we consider the generalized vector-valued Lipschitz space $\Lambda_{\gamma_{0}}(\mathbb{B}_{n},X).$

\begin{cor}\label{prelimn2} Suppose $1 \leq t < \infty.$ Then $\Lambda_{\gamma_{0}}(\mathbb{B}_{n},X) \subset A^{t}_{\alpha}(\mathbb{B}_{n},X).$ 
\end{cor}
\begin{proof}
Let $k > \gamma_{0}.$ Applying \cite[Theorem $3.1.2$]{Roc}, for $f \in \Lambda_{\gamma_{0}}(\mathbb{B}_{n},X),$ we have that
\begin{eqnarray*}
\|f\|^{t}_{t,\alpha,X} & \simeq & \int_{\mathbb{B}_n} [(1 - |z|^2)^{k}\|R^{\alpha,k}f(z)\|_{X}]^{t}\mathrm{d}\nu_{\alpha}(z)\\
& = & \int_{\mathbb{B}_n} [(1 - |z|^2)^{k-\gamma_{0}}\|R^{\alpha,k}f(z)\|_{X}]^{t}(1 - |z|^2)^{\gamma_{0}t}\mathrm{d}\nu_{\alpha}(z)\\
& \lesssim & \|f\|_{\Lambda_{\gamma_{0}}(\mathbb{B}_{n},X)}\int_{\mathbb{B}_n}(1 - |z|^2)^{\alpha+\gamma_{0}t}\mathrm{d}\nu(z) < \infty.
\end{eqnarray*}
\end{proof}

In what follows, we assume that $X,Y$ are reflexives complex Banach spaces.
We first introduce the following proposition which will be used in the proof of Theorem $\ref{Compactp}.$

\begin{prop}\label{bosss4} Suppose $1 < p \leq q < \infty,$ $0 \leq r < 1$ and $\gamma \in \mathbb{N}^{n}.$ If $a_{\gamma} \in \mathcal{K}(\overline{X},Y),$ then the little Hankel operator $h_{g^{\gamma}_{r}} : A^{p}_{\alpha}(\mathbb{B}_{n},X) \rightarrow A^{q}_{\alpha}(\mathbb{B}_{n},Y)$ is a compact operator, where $g^{\gamma}_{r}(z) = a_{\gamma}(rz)^{\gamma}$ for every $z \in \mathbb{B}_{n}.$
\end{prop}
\begin{proof}
Let $ \lbrace f_{j} \rbrace$ be a sequence in $A^{p}_{\alpha}(\mathbb{B}_{n},X)$ such that $f_{j} \rightarrow 0$ weakly in $A^{p}_{\alpha}(\mathbb{B}_{n},X)$ as $j$ tends to infinity. We want to prove that $\lim_{j\rightarrow\infty}\| h_{g^{\gamma}_{r}}f_{j}\|_{q,\alpha,Y} = 0.$  Let the Taylor expansion of $f_j$ given by $\displaystyle f_{j}(z) = \sum_{\beta \in \mathbb{N}^{n}} c^{j}_{\beta}z^{\beta} \in A^{p}_{\alpha}(\mathbb{B}_{n},X).$ Since $f_{j} \rightarrow 0$ weakly in $A^{p}_{\alpha}(\mathbb{B}_{n},X),$ applying Lemma $\ref{bosss1},$ using the fact that $c^{j}_{\beta} = \partial^{\beta} f_{j}(0)/\beta!,$ we have that for all $\beta \in \mathbb{N}^{n},$ $c^{j}_{\beta} \rightarrow 0$ weakly in $X$ as $j \rightarrow \infty.$ By Lemma $\ref{bosss2},$ for every $z \in \mathbb{B}_{n},$ we have $$\displaystyle h_{g_{r}^{\gamma}}f_{j}(z) = \sum_{\beta \in \mathbb{N}^{n},\beta \leq \gamma}a_{\gamma}(\overline{c^{j}_{\beta}}) \dfrac{\gamma!\Gamma(n+1+\alpha+|\gamma-\beta|)}{(\gamma-\beta)!\Gamma(n+1+\alpha+|\gamma|)} r^{|\gamma-\beta|} z^{\gamma-\beta}.$$ Therefore, 
\begin{eqnarray*}
\displaystyle \|h_{g_{r}^{\gamma}}f_{j}\|_{q,\alpha,Y} & = & \displaystyle \left( \int_{\mathbb{B}_n} \left\|\sum_{\beta \in \mathbb{N}^{n},\beta \leq \gamma}a_{\gamma}(\overline{c^{j}_{\beta}}) \dfrac{\gamma!\Gamma(n+1+\alpha+|\gamma-\beta|)}{(\gamma-\beta)!\Gamma(n+1+\alpha+|\gamma|)}(rz)^{\gamma-\beta}\right\|^{q}_{Y}\mathrm{d}\nu_{\alpha}(z)\right)^{1/q}\\
\\
& \leq & \displaystyle \left(\int_{\mathbb{B}_n} \left( \sum_{\beta \in \mathbb{N}^{n},\beta \leq \gamma}\|a_{\gamma}(\overline{c^{j}_{\beta}})\|_{Y} \dfrac{\gamma!\Gamma(n+1+\alpha+|\gamma-\beta|)}{(\gamma-\beta)!\Gamma(n+1+\alpha+|\gamma|)}(r|z|)^{|\gamma-\beta|}\right)^{q}\mathrm{d}\nu_{\alpha}(z)\right)^{1/q}\\
\\
& \leq & \displaystyle \sum_{\beta \in \mathbb{N}^{n},\beta \leq \gamma} \left(\int_{\mathbb{B}_n}\left(\|a_{\gamma}(\overline{c^{j}_{\beta}})\|_{Y} \dfrac{\gamma!\Gamma(n+1+\alpha+|\gamma-\beta|)}{(\gamma-\beta)!\Gamma(n+1+\alpha+|\gamma|)}(r|z|)^{|\gamma-\beta|}\right)^q\mathrm{d}\nu_{\alpha}(z)\right)^{\frac 1q}\\
\\
& = & \displaystyle \sum_{\beta \in \mathbb{N}^{n},\beta \leq \gamma} \|a_{\gamma}(\overline{c^{j}_{\beta}})\|_{Y}  \dfrac{\gamma!\Gamma(n+1+\alpha+|\gamma-\beta|)}{(\gamma-\beta)!\Gamma(n+1+\alpha+|\gamma|)} \left(\int_{\mathbb{B}_n} (r|z|)^{|\gamma-\beta|q}  \mathrm{d}\nu_{\alpha}(z)\right)^{\frac 1q}\\
\\
& \lesssim & \displaystyle \sum_{\beta \in \mathbb{N}^{n},\beta \leq \gamma} \dfrac{\gamma!\Gamma(n+1+\alpha+|\gamma-\beta|)}{(\gamma-\beta)!\Gamma(n+1+\alpha+|\gamma|)}\|a_{\gamma}(\overline{c^{j}_{\beta}})\|_{Y},
\end{eqnarray*}
where the third line above is justified by the Minkowsky's inequality for integrals.
Thus,  
\begin{equation}
\displaystyle \|h_{g_{r}^{\gamma}}f_{j}\|_{q,\alpha,Y} \lesssim \sum_{\beta \in \mathbb{N}^{n},\beta \leq \gamma}  \dfrac{\gamma!\Gamma(n+1+\alpha+|\gamma-\beta|)}{(\gamma-\beta)!\Gamma(n+1+\alpha+|\gamma|)} \|a_{\gamma}(\overline{c^{j}_{\beta}})\|_{Y}.\label{convnorm1}
\end{equation}
Now, since $c^{j}_{\beta} \rightarrow 0$ weakly in $X$ as $j\rightarrow \infty,$ it is clear that $\overline{c^{j}_{\beta}} \rightarrow 0$ weakly in $\overline{X}$ as $j \rightarrow \infty.$ By the assumption, we know that $a_{\gamma} \in \mathcal{K}(\overline{X},Y).$ Since $\overline{c^{j}_{\beta}} \rightarrow 0$ weakly in $\overline{X}$ as $j \rightarrow \infty,$ we have that
 $\|a_{\gamma}(\overline{c^{j}_{\beta}})\|_{Y} \rightarrow 0$ as $j \rightarrow \infty.$ It follows that
$$\limsup_{j\rightarrow \infty} \|h_{g_{r}^{\gamma}}f_{j}\|_{q,\alpha,Y} \lesssim \sum_{\beta \in \mathbb{N}^{n},\beta \leq \gamma} \dfrac{\gamma!\Gamma(n+1+\alpha+|\gamma-\beta|)}{(\gamma-\beta)!\Gamma(n+1+\alpha+|\gamma|)} \lim_{j\rightarrow \infty} \|a_{\gamma}(\overline{c^{j}_{\beta}})\|_{Y} = 0.$$
\end{proof}

Let us state Oliver's result on the boundedness of the little Hankel operator with operator-valued symbol between vector-valued Bergman spaces.

\begin{thm}\label{check3} Let $1 < p \leq q < \infty.$ The little Hankel operator $h_{b}: A^{p}_{\alpha}(\mathbb{B}_{n},X) \rightarrow A^{q}_{\alpha}(\mathbb{B}_{n},Y)$ is a bounded operator if and only if $ b \in \mathcal{B}_{\gamma}(\mathbb{B}_{n},\mathcal{L}(\overline{X},Y)),$ where $$ \gamma = 1+(n+1+\alpha)\left( \frac{1}{q} - \frac{1}{p}\right).$$ Moreover $$\|h_b\|_{A^{p}_{\alpha}(\mathbb{B}_{n},X) \rightarrow A^{q}_{\alpha}(\mathbb{B}_{n},Y)} \simeq \|b\|_{\mathcal{B}_{\gamma}(\mathbb{B}_{n},\mathcal{L}(\overline{X},Y))}.$$
\end{thm}

\begin{rem}\label{remark2}
 Suppose $1 < p < q < \infty,$ and  $\gamma = 1 + (n+1+\alpha)\left( \frac{1}{q} - \frac{1}{p}\right).$ Then $\gamma$ is not always positive. Indeed, since $1/q-1/p \in (-1,1),$ then $\gamma \in (-n-\alpha , n+2+\alpha ).$ It follows that when $\gamma \in (-(n+\alpha),0),$ the vector-valued $\gamma$-Bloch space $\mathcal{B}_{\gamma}(\mathbb{B}_{n}, \mathcal{L}(\overline{X},Y))$ is not interesting and does not make sense since the definition of the vector-valued $\gamma$-Bloch space introduced by Oliver only takes into account the case where $\gamma > 0.$ In Theorem $\ref{agene1}$, we correct the problem by replacing the vector-valued $\gamma$-Bloch space with the generalized vector-valued Lipschitz space $\Lambda_{\gamma_{0}}(\mathbb{B}_{n}, \mathcal{L}(\overline{X},Y)),$ where $\gamma_{0}  = (n+1+\alpha)\left(\frac{1}{p}-\frac{1}{q}\right).$ Since $\gamma = 1 - \gamma_{0},$ we see that when $0 < \gamma_{0} < 1,$ we have that $$\mathcal{B}_{\gamma}(\mathbb{B}_{n}, \mathcal{L}(\overline{X},Y)) = \Lambda_{\gamma_{0}}(\mathbb{B}_{n}, \mathcal{L}(\overline{X},Y)).$$
\end{rem}

In what follows, we give the proof of Theorem $\ref{agene1}$ which generalize the Theorem $\ref{check3}$ and correct the mistake mentionned in Remark \ref{remark2}.

\subsection{\textbf{Proof of Theorem $\ref{agene1}$}}

Let us recall the statement of Theorem \ref{agene1}.
\begin{thm}\label{agene11} Suppose $1 < p \leq q <\infty.$   
The little Hankel operator $h_{b}: A^p_{\alpha}(\mathbb{B}_{n},X)\\
\rightarrow A^{q}_{\alpha}(\mathbb{B}_{n},Y)$ is a bounded operator if and only if $b \in \Lambda_{\gamma_{0}}(\mathbb{B}_{n},\mathcal{L}(\overline{X},Y)),$ where $$\gamma_{0} = (n+1+\alpha)\left(\frac{1}{p} - \frac{1}{q}\right).$$  
Moreover, $\|h_{b}\|_{A^p_{\alpha}(\mathbb{B}_{n},X) \rightarrow A^{q}_{\alpha}(\mathbb{B}_{n},Y)} \simeq \|b\|_{ \Lambda_{\gamma_{0}}(\mathbb{B}_{n},\mathcal{L}(\overline{X},Y))}.$
\end{thm}

\begin{proof}
Let $p'$ and $q'$ such that $1/p +1/p' = 1$ and $1/q + 1/q' = 1.$ We first assume that $h_{b}$ is a bounded operator from $A^p_{\alpha}(\mathbb{B}_{n},X)$ to $A^q_{\alpha}(\mathbb{B}_{n},Y)$ with norm $\|h_b\| = \|h_b\|_{A^p_{\alpha}(\mathbb{B}_{n},X) \rightarrow A^q_{\alpha}(\mathbb{B}_{n},Y)}.$
Let $x \in X$ and $k> (n+1+\alpha)/p.$ Let $z \in \mathbb{B}_{n}$ and put 
$$ f(w) = \dfrac{x}{(1-\langle w,z \rangle)^{k}},\hspace{1cm} w \in \mathbb{B}_{n}.$$
Since  $k> (n+1+\alpha)/p,$ by Theorem $\ref{estimint},$ we have that $f \in A^p_{\alpha}(\mathbb{B}_{n},X)$ and
$$ \|f\|_{p,\alpha,X} \lesssim \dfrac{\|x\|_{X}}{(1-|z|^2)^{k-(n+1+\alpha)/p}}.$$ By \cite[Proposition $2.1.3$ ]{Roc}, we have that 
\begin{eqnarray*}
h_{b}f(z) & = & \displaystyle \int_{\mathbb{B}_n}\dfrac{b(w)(\overline{f(w)})}{(1 - \langle z,w \rangle)^{n+1+\alpha}}\mathrm{d}\nu_{\alpha}(w)\\
& = & \displaystyle \int_{\mathbb{B}_n}\dfrac{b(w)(\overline{x})}{(1 - \langle z,w \rangle)^{n+1+\alpha+k}}\mathrm{d}\nu_{\alpha}(w)\\
& = & R^{\alpha,k}b(z)(\overline{x}).
\end{eqnarray*}
It follows by Theorem $\ref{thm1}$ that 
\begin{eqnarray*}
\|R^{\alpha,k}b(z)(\overline{x})\|_{Y} & = & \|h_{b}f(z)\|_{Y}\\
& \leq & \dfrac{\|h_{b}f \|_{q,\alpha,Y}}{(1-|z|^2)^{(n+1+\alpha)/q}}\\ 
&\leq & \dfrac{\|h_{b}\|\|f\|_{p,\alpha,X}}{(1-|z|^2)^{(n+1+\alpha)/q}} \\
& \lesssim & \dfrac{\|h_{b}\|\|x\|_{X}}{(1-|z|^2)^{k+(n+1+\alpha)(1/q-1/p)}}\\
& = & \dfrac{\|h_{b}\|\|x\|_{X}}{(1-|z|^2)^{k-\gamma_{0}}}.
\end{eqnarray*}
Since $x \in X$ is arbitrary and $\|x\|_{X} = \|\overline{x}\|_{\overline{X}}$ we get that 
$$ \|R^{\alpha,k}b(z)\|_{\mathcal{L}(\overline{X},Y)} \lesssim \dfrac{\|h_b\|}{(1-|z|^2)^{k-\gamma_{0}}}.$$
Thus $$ \sup_{z \in \mathbb{B}_{n}} (1-|z|^2)^{k-\gamma_{0}}\|R^{\alpha,k}b(z)\|_{\mathcal{L}(\overline{X},Y)} \lesssim \|h_b\|.$$ By Lemma $\ref{prelimn1}$ this means that $b \in \Lambda_{\gamma_{0}}(\mathbb{B}_{n},\mathcal{L}(\overline{X},Y))$ and $\|b\|_{\Lambda_{\gamma_{0}}(\mathbb{B}_{n},\mathcal{L}(\overline{X},Y))} \lesssim \|h_{b}\|.$\\
Conversely, assume that $b \in \Lambda_{\gamma_{0}}(\mathbb{B}_{n},\mathcal{L}(\overline{X},Y)).$
Let $f \in A^p_{\alpha}(\mathbb{B}_{n},X),$ $g \in A^{q'}_{\alpha}(\mathbb{B}_{n},Y^{\star})$ and $k > \gamma_{0}.$ By Corollary $\ref{prelimn2},$ we have that $$b \in \Lambda_{\gamma_{0}}(\mathbb{B}_{n},\mathcal{L}(\overline{X},Y)) \subset A^{p'}_{\alpha}(\mathbb{B}_{n},\mathcal{L}(\overline{X},Y)),$$ so by \cite[Lemma $4.1.1$]{Roc}, Corollary $\ref{brett1},$  and Lemma $\ref{prelimn1}$ it follows that 
\begin{eqnarray*}
|\langle h_{b}f,g \rangle_{\alpha,Y}| & = & \displaystyle \left|\int_{\mathbb{B}_{n}} \langle b(z)\overline{f(z)},g(z) \rangle_{Y} \mathrm{d}\nu_{\alpha}(z)\right| \\
& = & \displaystyle \left| \int_{\mathbb{B}_n} \langle R^{\alpha,k+1}b(z)\overline{f(z)},g(z) \rangle_{Y} {d}\nu_{\alpha+k+1}(z)\right| \\
&\lesssim & \displaystyle  \int_{\mathbb{B}_n} \|R^{\alpha,k+1}b(z)\|_{\mathcal{L}(\overline{X},Y)}\|\overline{f(z)}\|_{\overline{X}}\|g(z)\|_{Y^{\star}}(1-|z|^2)^{k+1+\alpha}\mathrm{d}\nu(z)\\
& \lesssim & \displaystyle \|b\|_{\Lambda_{\gamma_{0}}(\mathbb{B}_{n},\mathcal{L}(\overline{X},Y))} \int_{\mathbb{B}_n} \|f(z)\|_{X}\|g(z)\|_{Y^{\star}}(1-|z|^2)^{\alpha+\gamma_{0}}\mathrm{d}\nu(z).
\end{eqnarray*}
By H\"older's inequality the last integral is less than or equal to
$$\displaystyle \left( \int_{\mathbb{B}_{n}}\|f(z)\|^q_{X}(1-|z|^2)^{\alpha+q\gamma_{0}}\mathrm{d}\nu(z) \right)^{1/q} \left( \int_{\mathbb{B}_{n}}\|g(z)\|^{q'}_{Y^{\star}}(1-|z|^2)^{\alpha}\mathrm{d}\nu(z) \right)^{1/q'}.$$ 
 
For $q=p$, we have $\gamma_0=0$ and thus 
$$|\langle h_{b}f,g \rangle_{\alpha,Y}|\lesssim \|b\|_{\Lambda_{\gamma_{0}}(\mathbb{B}_{n},\mathcal{L}(\overline{X},Y))} \|f\|_{p,\alpha,X} \|g\|_{p^\prime,\alpha,Y}.$$

For $q-p > 0,$ using Theorem $\ref{thm1},$ we have 
$$
\|f(z)\|^q_{X} = \|f(z)\|^p_{X} \|f(z)\|^{q-p}_{X} \leq \dfrac{\|f(z)\|^p_{X}\|f\|_{p,\alpha,X}^{q-p} }{(1-|z|^2)^{(q-p)(n+1+\alpha)/p}} = \\ \dfrac{\|f(z)\|^p_{X}\|f\|^{q-p}_{p,\alpha,X}}{(1-|z|^2)^{q\gamma_{0}}}.$$
It follows that
\begin{eqnarray*}
\displaystyle \left( \int_{\mathbb{B}_{n}}\|f(z)\|^q_{X}(1-|z|^2)^{\alpha+q\gamma_{0}}\mathrm{d}\nu(z) \right)^{1/q} & \leq & \|f\|^{1-p/q}_{p,\alpha,X} \left( \int_{\mathbb{B}_{n}}\|f(z)\|^{p}_{X}\dfrac{(1-|z|^2)^{\alpha+q\gamma_{0}}}{(1-|z|^2)^{q\gamma_{0}}}\mathrm{d}\nu(z) \right)^{1/q} \\
& = & \|f\|_{p,\alpha,X}.
\end{eqnarray*}
Therefore, by duality, we obtain that
$$\|h_{b}\|_{A^{p}_{\alpha}(\mathbb{B}_{n},X)\rightarrow A^{q}_{\alpha}(\mathbb{B}_{n},Y)} = \sup_{\|f\|_{p,\alpha,X} = 1; \|g\|_{q',\alpha,Y^{\star}}=1}\left| \langle h_{b}f,g \rangle_{\alpha,Y} \right| \lesssim \|b\|_{\Lambda_{\gamma_{0}}(\mathbb{B}_{n},\mathcal{L}(\overline{X},Y))}.$$
\end{proof}

\subsection{\textbf{Proof of Theorem $\ref{Compactp}$}}

We are now ready to give the proof of the main result in this section that is  Theorem $\ref{Compactp}$ that we recall here.

\begin{thm}\label{Compactp1} Let $X$ and $Y$ be two reflexive complex Banach spaces. Suppose that $1 < p \leq q < \infty,$ and $\alpha >-1$ The little Hankel operator $h_b : A^{p}_{\alpha}(\mathbb{B}_{n},X) \longrightarrow A^{q}_{\alpha}(\mathbb{B}_{n},Y)$ is a compact operator if and only if $$ b \in \Lambda_{\gamma_{0},0}(\mathbb{B}_{n},\mathcal{K}(\overline{X},Y)),$$ where $\Lambda_{\gamma_{0},0}(\mathbb{B}_{n},\mathcal{K}(\overline{X},Y))$ denotes the generalized little vector-valued Lipschitz space and $\gamma_{0} = (n+1+\alpha)\left( \frac{1}{p} - \frac{1}{q}\right),$ see \eqref{eq1}.
\end{thm}
\begin{proof} 
First assume that $b \in \Lambda_{\gamma_{0},0}(\mathbb{B}_{n}, \mathcal{K}(\overline{X},Y))$ and denote by $b_{r}(z):= b(rz)$ with $z\in \mathbb{B}_{n}$ and $0 < r < 1.$ 
Since $b \in \Lambda_{\gamma_{0},0}(\mathbb{B}_{n}, \mathcal{K}(\overline{X},Y)),$ by Theorem $\ref{agene1},$ we have that $$\|h_{b}\|_{A^p_{\alpha}(\mathbb{B}_{n},X) \rightarrow A^q_{\alpha}(\mathbb{B}_{n},Y)} \lesssim  \|b \|_{\Lambda_{\gamma_{0}}(\mathbb{B}_{n}, \mathcal{L}(\overline{X},Y))}.$$ Therefore, we have
$$\|h_{b}-h_{b_{r}}\|_{A^p_{\alpha}(\mathbb{B}_{n},X) \rightarrow A^q_{\alpha}(\mathbb{B}_{n},Y)} \lesssim  \|b - b_{r} \|_{\Lambda_{\gamma_{0}}(\mathbb{B}_{n}, \mathcal{L}(\overline{X},Y))}.$$
By using Proposition $\ref{closuresubsp},$ we have that 
$$\displaystyle \lim_{r \rightarrow 1^{-}} \|b - b_r \|_{\Lambda_{\gamma_{0}}(\mathbb{B}_{n}, \mathcal{L}(\overline{X},Y))} = 0,$$ so to prove that $h_b$ is a compact operator, it suffices to prove that $h_{b_{r}}$ is a compact operator.
Since $b_r$ is analytic on a neighbourhood of $\overline{\mathbb{B}}_n,$ it can be approximated by its Taylor polynomial in the Bloch norm. Thus, 
\begin{equation}
\lim_{N \rightarrow \infty} \|b_{r} - P_{N,r} \|_{\Lambda_{\gamma_{0}}(\mathbb{B}_{n}, \mathcal{L}(\overline{X},Y))} = 0,\label{covergpol}
\end{equation}
with $\displaystyle P_{N,r}(z) = \sum_{\beta \in \mathbb{N}^{n},|\beta| \leq N} \hat{b}(\beta) r^{|\beta|}z^{\beta},$ where $\hat{b}(\beta) \in \mathcal{K}(\overline{X},Y)$ are the Taylor coefficients of $b.$ We also have by Theorem $\ref{agene1}$ that
$$ \|h_{b_{r}}-h_{P_{N,r}}\|_{A^p_{\alpha}(\mathbb{B}_{n},X) \rightarrow A^q_{\alpha}(\mathbb{B}_{n},Y)} \lesssim  \|b_{r} - P_{N,r} \|_{\Lambda_{\gamma_{0}}(\mathbb{B}_{n}, \mathcal{L}(\overline{X},Y))}.$$
So by $\eqref{covergpol},$ to prove that $h_{b_{r}}$ is a compact operator, it is enough to prove that $h_{P_{N,r}}$ is a compact operator. Since $P_{N,r}$ is a polynomial, it is enough to do the proof for monomials of the form $\hat{b}(\beta) r^{|\beta|}z^{\beta},$ with $\beta \in \mathbb{N}^{n},$ $z \in \mathbb{B}_{n}$ and $\hat{b}(\beta) \in K(\overline{X},Y).$ Thus, according to Proposition $\ref{bosss4},$ the proof of this part is complete.\\
Conversely, for the ''only if part'', let us assume that $$h_{b}: A^p_\alpha(\mathbb{B}_{n},X) \longrightarrow A^q_\alpha(\mathbb{B}_{n},Y)$$ is a compact operator. 
Since $h_b$ is compact, $h_b$ is then bounded and Theorem $\ref{agene1}$ yields $$b \in  \Lambda_{\gamma_{0}}(\mathbb{B}_{n},\mathcal{L}(\overline{X},Y)).$$
We shall first prove that the Taylor coefficients $\hat{b}(\beta),$ $\beta \in \mathbb{N}^{n}$ of $b$ belongs to $\mathcal{K}(\overline{X},Y).$ 
Let $\lbrace f_j \rbrace \subset X$ such that $f_{j}\longrightarrow 0$ weakly in $X$ as $j\longrightarrow \infty,$
fix $\beta_{0} \in \mathbb{N}^{n},$ and let $x_{j}(z) = z^{\beta_{0}}f_{j}.$ By Lemma $\ref{unifbp},$ we have $\lbrace x_{j} \rbrace \subset A^p_{\alpha}(\mathbb{B}_{n},X)$ and $\lbrace x_{j} \rbrace$ converges weakly to $0$ in $A^p_{\alpha}(\mathbb{B}_{n},X).$ 
Since $$\|\hat{b}(\beta_{0})\overline{f_j}\|_{Y} = \sup_{\|y^{\star}\|_{Y^{\star}} = 1} |\langle \hat{b}(\beta_{0})\overline{f_{j}},y^{\star} \rangle_{Y,Y^{\star}}|$$ and $Y$ is reflexive, by the Kakutani's theorem \cite[Theorem $3.17$]{Brezis} there exists $y^{\star}_{j} \in Y^{\star}$ with $\|y^{\star}_{j}\|_{Y^{\star}} = 1$ such that  $$\|\hat{b}(\beta_{0})\overline{f_j}\|_{Y} = |\langle \hat{b}(\beta_{0})\overline{f_{j}},y^{\star}_{j} \rangle_{Y,Y^{\star}}|.$$
But $y^{\star}_{j} \in A^{p'}_{\alpha}(\mathbb{B}_{n},Y^{\star}).$ By Lemma $\ref{hank1},$ we have
\begin{eqnarray*}
|\langle h_{b}x_{j} , y^{\star}_{j} \rangle_{\alpha,Y}| & = & \displaystyle \left|\int_{\mathbb{B}_{n}} \langle b(z)\overline{x_{k}}(z),y^{\star}_{j} \rangle_{Y,Y^{\star}}\mathrm{d}\nu_{\alpha}(z)\right|\\
\\
& = & \left|\displaystyle \int_{\mathbb{B}_{n}} \overline{z}^{\beta_{0}}\langle \sum_{\beta \in \mathbb{N}^{n}} z^{\beta} \hat{b}(\beta)\overline{f_{j}},y^{\star}_{j} \rangle_{Y,Y^{\star}}\mathrm{d}\nu_{\alpha}(z)\right|\\
\\
& = & \displaystyle \left|\sum_{\beta \in \mathbb{N}^{n}}\langle \hat{b}(\beta)\overline{f_{j}},y^{\star}_{j} \rangle_{Y,Y^{\star}}\int_{\mathbb{B}_{n}}z^{\beta}\overline{z}^{\beta_{0}}\mathrm{d}\nu_{\alpha}(z)\right|\\
& = & \displaystyle |\langle \hat{b}(\beta_{0})\overline{f_{j}},y^{\star}_{j} \rangle_{Y,Y^{\star}}|\int_{\mathbb{B}_{n}}|z^{\beta_{0}}|^{2}\mathrm{d}\nu_{\alpha}(z)\\
& = & \dfrac{\beta_{0}!\Gamma(n+\alpha+1)}{\Gamma(n+|\beta_{0}|+\alpha+1)}|\langle \hat{b}(\beta_{0})\overline{f_{j}},y^{\star}_{j} \rangle_{Y,Y^{\star}}|\\
& = & \dfrac{\beta_{0}!\Gamma(n+\alpha+1)}{\Gamma(n+|\beta_{0}|+\alpha+1)}\|\hat{b}(\beta_{0})\overline{f_j}\|_{Y},
\end{eqnarray*}
where Fubini's theorem is justified by Lemma $\ref{justF}$ with  $\lbrace x_{j} \rbrace \subset H^{\infty}(\mathbb{B}_{n},X).$  
Since $h_b$ is compact and $\lbrace x_{j} \rbrace$ converges weakly to $0$ as $j$ tends to infinity, we have that $\lbrace h_{b}x_{j} \rbrace$ converges strongly to $0$ as $j$ tends to infinity, therefore one gets that $$\lim_{j\rightarrow \infty} \langle h_{b}x_{j} , y^{\star}_{j} \rangle_{\alpha,Y} = 0.$$ Thus
$$\lim_{j \rightarrow \infty}\dfrac{\beta_{0}!\Gamma(n+\alpha+1)}{\Gamma(n+|\beta_{0}|+\alpha+1)}\|\hat{b}(\beta_{0})\overline{f_j}\|_{Y} = 0.$$
We then obtain $$\lim_{j \rightarrow \infty} \|\hat{b}(\beta_{0})\overline{f_j}\|_{Y} = 0.$$ 
In fact, we have shown that $\hat{b}(\beta_{0})$ belongs to $\mathcal{K}(\overline{X},Y)$ and as $\beta_{0}$ is arbitrary, this holds for all $\beta \in \mathbb{N}^{n}.$
Let $1 < t < \infty.$ Since $b \in \Lambda_{\gamma_{0}}(\mathbb{B}_{n}, \mathcal{L}(\overline{X},Y)),$ we have that $b \in A^t_\alpha(\mathbb{B}_{n},\mathcal{L}(\overline{X},Y))$ and 
$$\lim_{N \rightarrow \infty} \int_{\mathbb{B}_n} \| b(w) - \sum_{|\beta| \leq N} \hat{b}(\beta)w^{\beta} \|^t_{\mathcal{L}(\overline{X},Y))} \mathrm{d}\nu_{\alpha}(w) = 0.$$ Let $ z \in \mathbb{B}_n.$ There exists a constant $C_{z} > 0$ such that 
$$\| b(z) - \sum_{|\beta| \leq N} \hat{b}(\beta)z^{\beta} \|^t_{\mathcal{L}(\overline{X},Y))} \leq C_{z} \int_{\mathbb{B}_n} \| b(w) - \sum_{|\beta| \leq N} \hat{b}(\beta)w^{\beta} \|^t_{\mathcal{L}(\overline{X},Y))} \mathrm{d}\nu_{\alpha}(w).$$ Thus,
$$\lim_{N \rightarrow \infty} \| b(z) - \sum_{|\beta| \leq N} \hat{b}(\beta)z^{\beta} \|_{\mathcal{L}(\overline{X},Y))} = 0.$$ Since $z \in \mathbb{B}_n$ is arbitrary, we deduce that $b(z) \in \mathcal{K}(\overline{X},Y),$ for each $z \in \mathbb{B}_{n}.$ It remains to show that $b$ satisfy the \textgravedbl little $\gamma_{0}$- Lipschitz\textacutedbl condition.
Let $x \in X$ and $y^{\star} \in Y^{\star}.$ Since $b \in  \Lambda_{\gamma_{0}}(\mathbb{B}_{n},\mathcal{L}(\overline{X},Y)),$ then the mapping $z \mapsto \langle b(z)\overline{x},y^{\star} \rangle_{Y,Y^{\star}}$ belongs to $A^1_\alpha(\mathbb{B}_{n},\mathbb{C}).$ By  using the reproducing kernel formula, it follows that
\begin{equation}
\displaystyle \langle  b(z)\overline{x},y^{\star} \rangle_{Y,Y^{\star}} = \int_{\mathbb{B}_{n}}  \dfrac{\langle b(w)\overline{x},y^{\star} \rangle_{Y,Y^{\star}}}{(1 - \langle z,w \rangle)^{n+1+\alpha}}\mathrm{d}\nu_{\alpha}(w).\label{agene2}
\end{equation} 
Let $k > \gamma_{0}.$ Applying the operator $R^{\alpha,k}$ in $\eqref{agene2},$ we obtain that
\begin{equation}
\displaystyle \langle R^{\alpha,k} b(z)\overline{x},y^{\star} \rangle_{Y,Y^{\star}} = \int_{\mathbb{B}_{n}} \dfrac{\langle b(w)\overline{x},y^{\star} \rangle_{Y,Y^{\star}}}{(1 - \langle z,w \rangle)^{n+1+\alpha+k}}\mathrm{d}\nu_{\alpha}(w).\label{hank4}
\end{equation}
Let $z \in \mathbb{B}_{n}.$ Since $\|R^{\alpha,k} b(z) \|_{\mathcal{L}(\overline{X},Y)} = \sup_{\|x\|_{X} = 1}\|R^{\alpha,k}b(z)(\overline{x})\|_{Y},$ and by Lemma $\ref{derivcompact},$ the operator $R^{\alpha,k}b(z)$ is compact. So there exists $x_{0}(z) \in X$ with $\|x_{0}(z)\|_{X} = 1$ and 
$$\|R^{\alpha,k}b(z) \|_{\mathcal{L}(\overline{X},Y)} = \|R^{\alpha,k}b(z)\overline{x_{0}(z)}\|_{Y}.$$ Also
$$\|R^{\alpha,k}b(z)\overline{x_{0}(z)}\|_{Y} = \sup_{\|y^{\star}\|_{Y^{\star}} = 1} |\langle R^{\alpha,k} b(z)\overline{x_{0}(z)},y^{\star} \rangle_{Y,Y^{\star}}|.$$ Since $Y$ is reflexive, it follows by the Kakutani's theorem \cite[Theorem $3.17$]{Brezis} that there exists $y^{\star}_{0}(z) \in Y^{\star}$ with $\|y_{0}^{\star}(z)\|_{Y^{\star}} = 1$ such that
\begin{equation}
\|R^{\alpha,k}b(z) \|_{\mathcal{L}(\overline{X},Y)} = \|R^{\alpha,k}b(z)(\overline{x_{0}(z)})\|_{Y} = |\langle R^{\alpha,k} b(z)\overline{x_{0}(z)}, y_{0}^{\star}(z) \rangle_{Y,Y^{\star}}|.\label{hank6}
\end{equation}  
By $\eqref{hank4}$ and $\eqref{hank6}$ we get 
\begin{eqnarray*}
(1 - |z|^2)^{k-\gamma_{0}}\| R^{\alpha,k}b(z)\|_{\mathcal{L}(\overline{X},Y)} & = & \displaystyle \left|\int_{\mathbb{B}_n} \langle b(w)\overline{x_{0}(z)},y^{\star}_{0}(z) \rangle_{Y,Y^{\star}} \dfrac{(1-|z|^2)^{k-\gamma_{0}}}{(1 - \langle z,w \rangle)^{n+1+\alpha+k}}\mathrm{d}\nu_{\alpha}(w)\right|\\
 \\
& = & |\langle h_{b}x_{z}, y_{z}^{\star} \rangle_{\alpha,Y}|,
\end{eqnarray*}
with
$$ x_{z}(w) = \dfrac{x_{0}(z)(1 - |z|^2)^{\beta - (n+1+\alpha)/p}}{(1 - \langle w,z \rangle)^{\beta}}, \hspace{1cm} w \in \mathbb{B}_{n}$$ and
$$y_{z}^{\star}(w) = \dfrac{y^{\star}_{0}(z)(1 - |z|^2)^{k + (n+1+\alpha)/q- \beta}}{(1 - \langle w,z \rangle)^{n+1+\alpha+k-\beta}}, \hspace{1cm} w \in \mathbb{B}_{n},$$ where $\beta$ is chosen such that $$(n+1+\alpha)/p < \beta < k+ (n+1+\alpha)/q .$$ By Theorem $\ref{estimint},$ we have 
$x_{z} \in A^p_{\alpha}(\mathbb{B}_{n},X),$ $y^{\star}_{z} \in A^{q'}_{\alpha}(\mathbb{B}_{n},Y^{\star}),$ and
$$\sup_{z \in \mathbb{B}_{n}} \|x_{z} \|_{p,\alpha,X} < \infty,\hspace{1cm}\sup_{z \in \mathbb{B}_{n}} \|y^{\star}_{z}\|_{q', \alpha, Y^{\star}} < \infty.$$ Let us prove that
\begin{equation}
 x_{z} \longrightarrow 0~~\mbox{weakly in}~~A^p_\alpha(\mathbb{B}_{n},X) ~~\mbox{as}~~|z| \longrightarrow 1^{-}.\label{hank8}
\end{equation}
Since $$ \sup_{z \in \mathbb{B}_{n}}\|x_{z}\|_{p,\alpha,X} < \infty,$$ to prove $\eqref{hank8},$ by Lemma $\ref{compl2},$ it suffices to prove that 
$$\langle x_{z},e_{w,a^{\star}} \rangle_{\alpha,X} \longrightarrow 0~~\mbox{as}~~|z| \longrightarrow 1^{-},$$ 
where for each $a^{\star} \in X^{\star}$ and $w \in \mathbb{B}_{n}$, we have
$$ e_{w,a^{\star}}(\zeta) = \dfrac{1}{(1 - \langle \zeta,w \rangle)^{n+1+\alpha}}a^{\star}, \hspace{1cm} \zeta \in \mathbb{B}_n.$$
By using the definition of $e_{w,a^{\star}}$ and the reproducing kernel formula, it follows that
\begin{eqnarray*}
\displaystyle \langle x_{z},e_{w,a^{\star}} \rangle_{p,\alpha,X} & = & \displaystyle \int_{\mathbb{B}_{n}} \langle x_{z}(\zeta),e_{w,a^{\star}}(\zeta) \rangle_{X,X^{\star}}\mathrm{d}\nu_{\alpha}(\zeta)\\
\\
& = & \displaystyle \int_{\mathbb{B}_{n}} \langle x_{z}(\zeta),\dfrac{1}{(1-\langle \zeta,w \rangle)^{n+1+\alpha}}a^{\star} \rangle_{X,X^{\star}}\mathrm{d}\nu_{\alpha}(\zeta)\\
\\
& = & \displaystyle \big\langle \int_{\mathbb{B}_{n}}\dfrac{ x_{z}(\zeta)}{(1-\langle w,\zeta \rangle)^{n+1+\alpha}}\mathrm{d}\nu_{\alpha}(\zeta),a^{\star} \big\rangle_{X,X^{\star}}\\
\\
& = & \displaystyle \langle x_{z}(w),a^{\star} \rangle_{X,X^{\star}}.
\end{eqnarray*}
Therefore, we have
\begin{eqnarray*}
|\langle x_{z},e_{w,a^{\star}} \rangle_{p,\alpha,X}| & = & |\langle x_{z}(w),a^{\star} \rangle_{X,X^{\star}}|\\
\\
& = & \left|\dfrac{(1-|z|^2)^{\beta-(n+1+\alpha)/p}}{(1-\langle w,z \rangle)^{\beta}} \langle x_{0}(z),a^{\star} \rangle_{X,X^{\star}} \right| \\
\\
& \leq & \dfrac{(1-|z|^2)^{\beta-(n+1+\alpha)/p}}{(1- |w|)^{\beta}}\|a^{\star}\|_{X^{\star}} \longrightarrow 0
\end{eqnarray*}
as $|z| \longrightarrow 1^{-}.$  
By using $\eqref{hank8},$ the compactness of $h_{b}$ and the fact that $$\sup_{z \in \mathbb{B}_{n}}\|y_{z}^{\star}\|_{q',\alpha,Y^{\star}}< \infty,$$ it follows that
$$\lim_{|z|\rightarrow 1^{-}} (1 - |z|^2)^{k-\gamma_{0}}\|R^{\alpha,k} b(z)\|_{\mathcal{L}(\overline{X},Y)} = \lim_{|z|\rightarrow 1^{-}} |\langle h_{b}x_{z}, y_{z}^{\star} \rangle_{\alpha,Y}| = 0,$$ which completes the proof of the theorem.
\end{proof}
%\end{thm}

\vspace{0,5cm}
\end{document}